\documentclass{amsart} 
\usepackage[foot]{amsaddr}
\usepackage{color}
\usepackage[utf8]{inputenc}
\usepackage[backend=biber, isbn=false, doi=false, url=false,maxnames=5] {biblatex}
\renewbibmacro{in:}{} 
\bibliography{gp} 
\usepackage{tikz}
\usepackage{amsmath} 
\usepackage{amssymb}
\usepackage{hyperref} 
\usepackage{enumerate}
\colorlet{symbols}{black} 
\tikzset{
	dot/.style={circle,fill=symbols,draw=symbols,inner sep=0pt,minimum size=0.4pt},
	basic/.style={draw=symbols},
	>=stealth,
	}

\renewcommand{\d}{\partial} 
\newcommand{\ov}{\overline}

\renewcommand{\Re}{{\rm Re}\,}
\renewcommand{\Im}{{\rm Im}\,}
\newcommand{\loc}{{\rm loc}\,}
\renewcommand{\sup}{{\rm sup}\,}

\def\Tc{T_{\hbox{\tiny c}}}

\def\cH{\mathcal{H}}

\def\C{\mathop{\mathbb C\kern 0pt}\nolimits}
\def\N{\mathop{\mathbb N\kern 0pt}\nolimits}
\def\R{\mathop{\mathbb R\kern 0pt}\nolimits}
\def\Z{\mathop{\mathbb Z\kern 0pt}\nolimits}
\def\bq{\mathop{\mathbf q\kern 0pt}\nolimits}

\def\dt{\,{\rm d}t\,}
\def\dy{\,{\rm d}y\,}

\def\dx{\,{\rm d}x\,}   
\def\dm{\,{\rm d}m\,}  
\def\dxi{\,{\rm d}\xi\,} 
\def\dt'{\,{\rm d}t'\,}

\def\ds'{\,{\rm d}s'\,}
\def\sech{\,{\rm sech}\,}
\def\tanh{\,{\rm tanh}\,}

\def\cC{\mathcal{C}}

\def\cE{\mathcal{E}}

\def\cI{\mathcal{I}}

\def\cM{\mathcal{M}}

\def\cU{\mathcal{U}}

\def\cP{\mathcal{P}}
\def\cR{\mathcal{R}}
\def\cS{\mathcal{S}}

\def\fb{\mathfrak{b}}

\newtheorem{thm}{Theorem}[section]
\newtheorem{lem}[thm]{Lemma}
\newtheorem{rmk}[thm] {Remark}
\newtheorem{cor}[thm]{Corollary}
\newtheorem{prop}[thm]{Proposition}
\newtheorem{defi}[thm]{Definition}

\title[Conserved energies for the GP equation: low regularity case]
{Conserved energies for the one dimensional Gross-Pitaevskii equation: low regularity case} 

 \author[H. Koch]{Herbert Koch}
\address [H. Koch]%
 {Mathematisches Institut, Universit\"at Bonn, Endenicher Allee 60, 53115 Bonn, Germany}
\email{koch@math.uni-bonn.de}

 \author[X. Liao]{Xian Liao}
\address [X. Liao]%
 {Institute for Analysis, Karlsruhe Institute for Technology, Englerstrasse 2, 76131 Karlsruhe, Germany}
\email{xian.liao@kit.edu} 

\date{}

\begin{document} 

\begin{abstract} 

We  construct a family of conserved energies for the one dimensional Gross-Pitaevskii equation, but in the low regularity case (in  \cite{KL} we have constructed conserved energies in the high regularity situation).
This can be done thanks to   regularization procedures and a study of 
the topological structure of the finite-energy space. 
The asymptotic  (regularised conserved) phase change on the real line with values in $ \R/2\pi \Z$  is studied. 
We also construct a conserved quantity, the renormalized momentum $H_1$ (see Theorem \ref{thm:E1}), on the universal covering space of the finite-energy space. 
 
\end{abstract}
\maketitle

\vspace{.1cm}



\section{Introduction}
We consider the one dimensional Gross-Pitaevskii equation
\begin{equation}\label{GP}
i\d_t q+\d_{xx}q=2q(|q|^2-1).
\end{equation}
Here $(t,x)\in\R^2$ denote the one dimensional time and space variables respectively and   $q=q(t,x)$ denotes the unknown complex-valued wave function.

The Gross-Pitaevskii equation \eqref{GP} can be viewed as the defocusing cubic nonlinear Schr\"odinger equation (NLS), but assuming a nonzero boundary condition at infinity
$$
|q(t,x)|\rightarrow 1,\quad\hbox{ as }|x|\rightarrow\infty.
$$ 
It is relevant  in the physical contexts of Bose-Einstein condensation, nonlinear optics (e.g.  optical vortices) and fluid mechanics (e.g. superfluidity of Helium II).
Thanks to this nonzero boundary condition, the Gross-Pitaevskii equation \eqref{GP} possesses the following interesting black (with $c=0$) and dark (with $c\neq 0$)  soliton solutions in nonlinear optics
\begin{equation} \label{eq:soliton}
q_c(t,x)=\sqrt{ 1-c^2}\tanh\Bigl( \sqrt{1-c^2} (x-2ct)\Bigr)+ic,\quad -1<c<1.
\end{equation} 
They are called black resp. dark since their density drops from the background density $1$:
$$|q_c(t,x)|^2=1-(1-c^2)\sech^2\bigl(\sqrt{1-c^2}(x-2ct)\bigr)<1,\quad -1<c<1.$$
There are no soliton solutions of \eqref{GP} with travelling speed $|2c|\geq 2$.
There were many works, e.g.  \cite{ BGS15,BGSS08, CJ, GZ, Lin}, contributing to the stability of these soliton solutions.

The Gross-Pitaevskii equation \eqref{GP} is  the Hamiltonian flow with respect to the Hamiltonian function
\begin{equation}\label{cEGL}
\cE(q)=\int_{\R} \bigl((|q|^2-1)^2+ |\d_x q|^2 \bigr)\dx,
\end{equation}  
and the Poisson structure (see \cite{FT}) 
\[ \omega(F,G) =i \int_{\R} \Bigl( \frac{\delta F }{\delta q} \frac{\delta G}{\delta \bar q }-\frac{\delta F}{\delta\bar q}\frac{\delta G}{\delta q}\Bigr)\dx . \] 
Aside from the  energy $\cE$, the mass 
\begin{align}\label{cM}
 \cM=\int_{\R} (|q|^2-1)\dx,   
\end{align} 
and the momentum 
\begin{align}\label{cP}
\cP =\Im\int_{\R} q\d_x \bar q\dx
\end{align} 
are  conserved by the Gross-Pitaevskii flow. 
One computes straightforward for the soliton solution family \eqref{eq:soliton} that 
\begin{align*}
&\mathcal{M}(q_c) = -2(1-c^2),  
\\
& \mathcal{P}(q_c) =2c \sqrt{1-c^2},
\\
&\mathcal{E}(q_c)
=\frac83\sqrt{1-c^2}^3. \end{align*} 

\smallbreak 

The Gross-Pitaevskii equation \eqref{GP} was shown by P. Zhidkov \cite{Zhidkov}  to be locally-in-time well-posed  in the so-called Zhidkov's space 
\begin{equation*}\begin{split}
& Z^k=\{q\in  L^\infty(\R)\,|\, \d_x q\in H^{k-1}(\R)\},
\\
&\hbox{ associated with the norm }\|q\|_{Z^k}=\|q\|_{L^\infty}+\sum_{1\leq l\leq k} \|\d_x^l q\|_{L^2}, 
\end{split}\end{equation*}
for $k\geq 1$, and globally-in-time well-posed in $Z^1$ by use of the  conserved energy $\cE(q)$ in \eqref{cEGL}.  However even the stability with respect to this metric of the trivial solution $q=1$ is expected to be false: We expect that for all $\varepsilon>0$ and $ k \ge 0$ there exists initial data $q_0$ so that $ \Vert q_0-1 \Vert_{Z^k} < \varepsilon$
but the solution $q(t)$ satisfies $ \sup_t \Vert q(t)-1 \Vert_{L^\infty} \ge 2$, with $2$ being the diameter of the unit circle. 

It is also interesting to mention the global-in-time well-posedness results   by P. G\'erard \cite{Gerard06, Gerard08}   in the energy space $Y^1=\{ q\in H^1_{\loc}(\R^n): |q|^2-1\in L^2(\R^n), \, \nabla q\in L^2(\R^n) \}$,
 endowed with the metric distance
 \begin{align*}
&d_{Y^1}(p,q)=\|p-q\|_{Z^1+H^1}+\bigl\||p|^2-|q|^2\bigr\|_{L^2},
\\
&\hbox{ with }\|u\|_{A+B}=\inf\{\|u_1\|_{A}+\|u_2\|_{B}\,|\, u=u_1+u_2,
\, u_1\in A,\, u_2\in B\}.
\end{align*}  
Notice that  there is no simple  relation between the  metric space $(X^1, d^1)$ (see Theorem \ref{thm:Xs} below) and the Zhidkov's space $(Z^1, \|\cdot\|_{Z^1})$ or the G\'erard's space $(Y^1, d_{Y^1})$. In higher dimensions there are  scattering results     in \cite{GNT, GNT09}, which are not expected to hold in one space dimension by virtue of these soliton solutions \eqref{eq:soliton}.


We denote  the finite-energy   space of the Gross-Pitaevskii equation by
$$
q(t,\cdot)\in X^1:= \bigl\{ p\in H^1_{\loc}(\R)\,:\, |p|^2-1,\d_xp\in L^2(\R) 
\bigr\}/ {\mathbb{S}^1},
$$
where $\mathbb{S}^1$ denotes the unit circle (that is,  the functions which differ by a multiplicative constant  of modulus $1$ are identified in $X^1$). 
Let the operator $D=\langle \partial_x\rangle$ be defined by  
     \begin{equation}\label{D}\widehat{Df}(\xi)=\langle\xi\rangle\hat{f}(\xi),\end{equation}
 where $\langle\xi\rangle=\sqrt{2^2+\xi^2}$, and $\hat{f}(\xi)=\frac{1}{\sqrt{2\pi}}\int_{\R}e^{-i\xi x}f(x)\dx$ denotes the Fourier transform of the function $f(x)$. 
 Here the factor $2$ appears due to the mismatch between the usual Fourier transform and the (inverse) scattering transform terminology.  We define 
 \begin{equation} \label{eq:energy}  \Vert q \Vert_{H^s} = \Vert D^s q \Vert_{L^2}  \end{equation} 
 and introduce the   energy ``norm''  in $X^s$
\begin{equation}\label{energy,D}
E^s (q)=\Bigl(\|D^{s-1}(|q|^2-1)\|_{L^2}^2
+\|D^{s-1}(\d_x q)\|_{L^2}^2\Bigr)^{\frac12},
\end{equation}
and in particular 
$$(E^1(q))^2=\cE(q),\quad (E^0(q))^2=\|D^{-1}(|q|^2-1)\|_{L^2}^2+\|D^{-1}(\d_xq)\|_{L^2}^2.$$
Obviously the energy norm  $E^1(q)$ is conserved by the Gross-Pitaevskii flow.
We aim to show (almost) conservation of the energy norm $E^s(q)$ for all $s\geq 0$, and hence the global-in-time wellposedness for all $s\geq 0$.
In the present paper we consider more general finite-energy 
spaces 
\begin{equation}\label{Xs}
X^s= \bigl\{ q\in H^s_{\loc}(\R)\,:\, |q|^2-1\in H^{s-1}(\R), \quad \d_x q\in H^{s-1}(\R) 
\bigr\}/ {\mathbb{S}^1},
\quad s\geq 0.
\end{equation} 
We observe that neither 
 mass $\mathcal{M}$ (see \eqref{cM}) nor momentum $\mathcal{P}$ (see \eqref{cP}) can be defined on any $X^s$, $ s \ge 0$.

\smallbreak

In the following we are going to state  properties of the metric space $(X^s, d^s)$ in Subsection \ref{subs:Results} together with a  conserved quantity, the renormalized momentum  $H_1$, on the universal covering space of the energy space.
Then  we will state the main results Theorem \ref{thm:GP} on conserved energies  for  $s\geq 0$ in Subsection \ref{subs:EnergyResults}, which follows from   more general  properties of the (renormalized) transmission coefficients associated to the Gross-Pitaevskii equation in Theorem \ref{thm:energies} in the low regularity regime.
Some of the proofs and the ideas   will be sketched in Subsection \ref{subs:Ideas}. The proofs of the central results are found in Section \ref{subs:Theta} and Section \ref{sec:T}.

\subsection{The metric space \texorpdfstring{$(X^s, d^s)$}{(X\textasciicircum s,d\textasciicircum s)}}\label{subs:Results} 
We first recall some properties of   the generalized finite-energy  space $X^s$ defined in \eqref{Xs} from \cite{KL}.  
 \begin{thm}[Properties of the metric space, \cite{KL}]\label{thm:Xs}
 The distance function $d^s(\cdot,\cdot)$ on the space $X^s$ is defined by 
 \footnote{Here the weight function $\sech(x)=\frac{2}{e^{x}+e^{-x}}$ can be equivalently replaced by any other strictly positive and smooth function with fast decay at infinity.}
 \begin{equation}\label{ds}\begin{split}
&d^s(p,q)=
\Bigl(\int _{ \R}
 \inf_{|\lambda(y)|=1} \bigl\|\sech( \cdot-y) (\lambda p-q)\bigr\|_{H^{s}(\R)}^2 \dy
 \Bigr)^{\frac 12}.
\end{split}\end{equation} 
  Then the metric space $(X^s, d^s)$, $s\geq 0$ has the following properties:
 \begin{itemize}
\item The space $(X^s, d^s(\cdot,\cdot))$ is a complete metric space.
\item The subset  $\{q\,|\, q-1\in C^\infty_0(\R)\}$
is dense in $X^s$. 

\item Any set $\{ q\in H^s_\loc(\R):  \Vert \d_x q \Vert_{H^{s-1}(\R)} + \Vert |q|^2-1 \Vert_{H^{s-1}(\R)} < C\}$ is contained in some ball $B^s_r(1)$ with $r$ depending on $C$.
\item Any closed  ball $\overline{B_r^s(q)}$  in $X^s$, $s>0$ is weakly sequentially compact. 
\item There is an analytic structure on $X^s$ which is compatible with the metric. 
\item There exists an absolute constant $c$ such that 
  \begin{equation*}\label{diameter}   d^s(1,q) \le cE^s(q). \end{equation*} 
  If $p \in X^s$ and $q \in H^s_{loc} $ so that $d^s(p,q) < \infty$, then
  $ q\in X^s$ and 
  \footnote{The coefficient before $d^s(p,q)$ is corrected from 
  $c(1+\|p'\|_{H^{s-1}}^{\frac12}+\||p|^2-1\|_{H^{s-1}}^{\frac12})$ (in \cite{KL}) to $c(1+\|p'\|_{H^{s-1}}+\||p|^2-1\|_{H^{s-1}}^{\frac12})$ below.}
\begin{equation*}\label{Es,ds,2}
 E^s(q) \le E^s(p) + c(1+\|p'\|_{H^{s-1}}+\||p|^2-1\|_{H^{s-1}}^{\frac12})d^s(p,q)+ c(d^s(p,q))^2. 
 \end{equation*}  

 \end{itemize}
 \end{thm}

The analytic structure defined in \cite{KL} can be described as follows. 
We fix a function $\eta\in C^\infty_0([-4,4])$ with $\eta=1$ on $[-2,2]$. 
 We fix a partition of unity
 \[ 1=\sum_{n\in \Z} \rho(x-n). \] 
 We define the Hilbert space of real valued sequences $ l^2_d$ equipped with the norm 
\[ \|(a_n)_n\|_{l^2_d}=\bigl(\sum_{n}|a_n-a_{n-1}|^2\bigr)^{\frac12} \] 
and, given $q\in X^s$ and given $ L>0$,   we define
 \[ \tilde H^s =\{\fb\in H^s(\R)\,|\,\langle \eta(x/L-n)\fb,\eta(x/L-n) q\rangle_{H^s}\in \R,\quad \forall n\in\Z\}. \] 
We have shown in \cite{KL}:   Given $ q \in X^s$, there exist $L>0$ and $r,R>0$ depending only on $ E^s(q)$, so that  we can  define  a map 
 \[  \Psi:   l^2_d \times \tilde H^s  \to X^s \] 
 by 
\[ \Psi( (a_n), \fb) = e^{i\theta}(q+\fb),\] 
 \[ \hbox{ with } \theta(x) = \sum_{n\in \Z} a_n \rho(x/L-n). \]
Here the map   $\Psi$ restricted to  $B_r^{l^2_d\times \tilde H^s}((0)_n, 0)$ is bilipschitz to its range, 
which contains the ball $ B_R^{X^s}(q)$ in $X^s$,
and the related coordinate changes and their Fr\'echet derivatives are bounded by some constants depending only on $E^s(q)$ and $s$.

 The topology of the metric space $(X^s,d^s)$ is nontrivial. It is described by the following theorem which  we will  prove in Section \ref{subs:Theta}. 
 \begin{thm}[Topology of the metric space] \label{thm:homotopy}
Let $s\geq 0$.
Let 
\[ Q = \{ q_c: -1\le c \le 1 \} \subset X^s \] 
where $q_c$ denotes the `profile' of the soliton solutions given in \eqref{eq:soliton}
$$
q_c=q_c(x)=\sqrt{1-c^2}\tanh(\sqrt{1-c^2}x)+ic,\quad c\in [-1,1],
$$
and we recall that $ q_1= q_{-1}=1$ in $X^s$ due to the identification of functions differing only by phase. 

Then $Q$ is a strong deformation retract of $(X^s, d^s)$,
which means that there is a continuous map (called deformation) 
\[ \Xi: [0,1]\times  X^s  \to X^s \] 
so that 
\begin{enumerate} 
\item $ \Xi(0,q) = q   ,\quad\forall q\in X^s$,
\item $ \Xi(t, q_c) = q_c  ,\quad \forall t\in [0,1]$, $\forall c\in [-1,1]$,
\item $ \Xi(1,q) \subset Q  ,\quad \forall q\in X^s$.
\end{enumerate} 
\end{thm} 
This corrects a statement in \cite{KL} where we stated incorrectly that all balls in $X^s$ are contractible. In particular $(X^s, d^s)$ is homotopy equivalent to the circle. The following argument shows that the topology of the phase space is nontrivial.   If $q=|q|e^{i\varphi}\in X^1$  never vanishes and   $\d_x q\in L^1$, we can  define the asymptotic change of the phase on the real line as

\begin{equation}\label{eq:theta}
\begin{split} 
 \Theta(q)\, &  := \Im \int_{\R} \frac{q}{|q|}\d_x\frac{\bar q}{|\bar q|}\dx
 \\ & =\Im\int_{\R}\frac{\d_x\bar q}{\bar q}\dx
 \\ & = -\lim_{x\rightarrow+\infty} (\varphi(x)-\varphi(-x))
\in \R/(2\pi \Z),
\end{split} 
\end{equation}
which is also conserved by the Gross-Pitaevskii flow (on the time interval where the solution never vanishes). It is a Casimir function which means that it Poisson commutes with every Hamiltonian: $ \frac{\delta \Theta}{\delta q} = 0 $, and hence the Hamiltonian vector field of $ \Theta$ vanishes.

If we consider the finite-energy solution $q\in X^1$ (even with no zeros), then the above phase change function may not be well-defined (keeping in mind of the example $e^{i\ln(1+|x|)}\in X^1$). 
We are going to extend the definition of $\Theta$ to general $q\in X^0$ with $\d_x q\in L^1$ (see Theorem \ref{thm:E1} below for more details), and similar as the mass $\cM$, this quantity $\Theta$    
will play an important role  later in the analysis of low regularity case.

Notice that  the (soliton) solutions defined \eqref{eq:soliton} have the following asymptotic behaviors at infinity (modulo $\mathbb{S}^1$):
$$
\lim_{x\rightarrow\infty} q_c(t,x)=\sqrt{1-c^2}+ic,
\quad \lim_{x\rightarrow-\infty} q_c(t,x)=-\sqrt{1-c^2}+ic,\quad c\in (-1,1),
$$
and in particular under the identification modulo $\mathbb{S}^1$ in $X^s$,
$$
q_{-1}=q_1=1\hbox{ in }X^s.
$$
If we consider  modulo $2\pi \Z$, then 
 \[ \Theta(q_c) = 2 \arccos(c), \]
and $c \to \Theta(q_c) \in \R/(2\pi \Z)$ is continuous. However it is not possible to lift   $ \Theta(q_c): [-1,1]\mapsto \R$ continuously. 
Notice that $c\to 2\arccos(c)$ is a monotonically  decreasing function 
 $$  [-1,1]\ni c\mapsto 2\arccos(c)\in [0,2\pi] \hbox{ with }
 -1,0,1\mapsto 2\pi, \pi, 0 \hbox{ respectively. }$$
 Thus 
 \begin{equation} \label{eq:notsimple}  \{ q \in X^s:  \d_x q \in \mathcal{S} \} \end{equation}
  is not simply connected (indeed Theorem \ref{thm:homotopy} implies that $X^s$ is homotopy equivalent to the circle). 
 Nevertheless the asymptotic phase change $\Theta$ is closely related to the momentum $\cP$ by view of the Hamiltonian $H_1$ defined below.  
\begin{thm}[Asymptotic phase change $\Theta$ and renormalized momentum $H_1$] \label{thm:E1}
The asymptotic change of phase $\Theta$ has a unique continuous and smooth extension to $\{q\in X^0: \d_x q \in L^1\}  $ modulo $2\pi\Z$
\[ \Theta:  \{ q \in X^0: \d_x q \in L^1 \}  \to \R/(2\pi \Z).  \]

The Hamiltonian 
\begin{equation}\label{H1,nonzero} H_1(q) = -\Im \int_{\R} \bigl( (|q|^2 -1) \partial_x \log  q \bigr)\, dx \in \mathbb{R}/ (2\pi \Z)  \end{equation}
which is defined on 
\[ \{ q\in  X^{\frac12+\varepsilon} :  q\ne 0 \}, \quad\varepsilon>0 \] 
has a unique continuous and smooth extension to $X^{\frac12}$ modulo $2\pi \Z$. 
It is conserved under the flow of the Gross-Pitaevskii  equation \eqref{GP}.
\end{thm} 
We sketch the proof of Theorem \ref{thm:E1} and leave the details for Section \ref{subs:Theta}.
For the extensions of both $\Theta$ and $H_1$, given $ q_0 \in X^s$,  we define a continuous map 
\[ B^{X^s}_\delta(q_0) \ni  q \to  \tilde q \] 
so that  $\tilde q \in X^{s+2}$ 
  does not vanish.  
For the extension of $\Theta$ on $ \{ q \in X^0: \d_x q \in L^1 \}$, we can take furthermore $\tilde q$ (see Lemma \ref{lem:tildeq}) such that
\[ q - \tilde q \in L^2, \d_x\tilde  q \in L^1, \] 
  and set (with an abuse of notations)
\[ \Theta = \Theta( \tilde q), \]
where  $ \Theta(\tilde q) $ is the winding number of $\tilde q$ given in \eqref{eq:theta}.
Different choices of $ \tilde q$ may lead to different values of $ \Theta(\tilde q)$, which  is nevertheless unique up to a multiple of $2\pi$, and $\Theta: \{ q \in B^{X^s}_\delta(q_0): \d_x q \in L^1 \} \mapsto \R/(2\pi\Z)$ is continuous.  

For the extension of $H_1$ on $X^{\frac12}$,  we proceed in a similar fashion, so that 
$$ q - \tilde q \in H^{\frac12}.$$
If $q\in X^{\frac12+\varepsilon}$ with $\varepsilon>0$ and $\d_x q, \d_x \tilde q\in L^1$, we define
\begin{equation}\label{H1,tildeq} 
H_1(q) = -\Im \int_{\R}\Bigl( \bar q \d_x q - \d_x \tilde q/\tilde q\Bigr) \dx \in \R / (2\pi \Z),  \end{equation}
which coincides with the definition \eqref{H1,nonzero}   if $ q$ never vanishes.  
The first term is  the momentum $\cP$ (see \eqref{cP}) and the second term is again the winding number of $\tilde q$ if $\d_x \tilde q \in L^1$:
\begin{equation}\label{P,H1}
H_1(q)=\cP(q)-\Theta(q).
\end{equation}
On the other hand, we may rewrite \eqref{H1,tildeq} as
\begin{equation}\label{eq:H1} 
H_1(q)=-\Im   \int_{\R}\Bigl( (\overline{q - \tilde q} ) \d_x q - \d_x \overline{\tilde q} (q-\tilde q) + \frac1{\tilde q} (|\tilde q |^2-1) \partial_x \tilde q\Bigr)\,  \dx,  \end{equation}  
which in this form is easily seen to be well defined on $X^{\frac12}$, and, with a continuous choice of $ \tilde q$, it is continuous resp. smooth modulo $2\pi \Z$.  
Note that   $H_1(q)$  cannot be defined on $X^s$ if $ 0 \le s < \frac12$, since the momentum $ \cP(q)=\Im \int_{\R} q \d_x \bar q dx $ cannot be defined even on $\{ q \in  X^s: q(x) = 1 \text{ if } |x| \ge 1 \} $.

To see the conservation of the Hamiltonian $H_1(q(t,x))$ under the Gross-Pitaevskii flow, we recall the construction of the solution  in  \cite[Section 2]{KL}. Let $q_0 \in X^s$, $s\geq 0$ (more precisely a representative of the equivalence class) and 
 let $ \tilde q \in X^\infty$ be its  regularisation (depending only on the initial data $q_0$, but not on $t$, which may be different from the regularisation above).  
We make the Ansatz   
 $q= \tilde q + b$ for the solution of the Gross-Pitaevskii equation \eqref{GP} and  search for $b$ as a solution to 
\[ i\partial_t b + b_{xx} = 2|b|^2 b + 4|b|^2 \tilde q 
+ 2\overline{\tilde q}b^2
+(4|\tilde q|^2-2)b
+ 2(\tilde q)^2 b
+2 (|\tilde q|^2 -1) \tilde q - \tilde q_{xx}  \] 
with initial data $b(0) = q_0 - \tilde q\in H^s$. 
Using Strichartz estimates we obtain a unique   solution $ b \in C([0,T], H^s)$ for some time $T>0$. 
If $ q_0 \in 1 + \mathcal{S}$, then $q(t)\in 1+\cS$ and with $ \tilde q=1$ we have $\Theta (\tilde q)  =0 $. It is clearly independent of time.  
The momentum is conserved for $ q_0 \in 1 + \mathcal{S}$ and hence by virtue of the relation \eqref{P,H1}, $H_1$ is   conserved on $1+\mathcal{S}$.
By the density of the set $1+\cS$ in $(X^s, d^s)$ and the local continuity of $H_1$ in  $X^{\frac12}$,   $H_1$ is also conserved for initial data in $X^{\frac12}$.

 \subsection{Energies and transmission coefficient}  \label{subs:EnergyResults}
 Recently continuous families of conserved energies have been constructed for the Korteweg-de Vries  equation  (KdV), the modified KdV and NLS equations independently by Killip, Visan and Zhang \cite{KVZ}, and the first author and Tataru \cite{KT}. 
We recall the local-in-time well-posedness result of the Gross-Pitaevskii equation in $(X^s, d^s)$, $s\geq 0$, as well as the global-in-time well-posedness result for the case $s>\frac12$ (by constructing a family of conserved energy functionals) in our previous work \cite{KL} below.
\begin{thm}[LWP in $X^s$, $s\geq 0$ \& GWP in $X^s$, $s>\frac12$, \cite{KL}]\label{thm:KL}
 Let $s\geq 0$.  The Gross-Pitaevskii equation \eqref{GP} is
 locally-in-time well-posed in the metric space $(X^s, d^s)$ in the
 following sense: For any initial data $q_0\in X^s$, there exists a
 positive time $\bar t\in (0,\infty)$ and a unique local-in-time
 solution $q\in \cC((-\bar t, \bar t); X^s)$ of \eqref{GP} and for any
 $t\in (0,\bar t)$, the Gross-Pitaevskii flow map $X^s\ni q_0\mapsto
 q\in \cC([-t, t]; X^s)$ is  
 continuous. 
 \footnote{
Here  the solution $q\in \mathcal{C}((-\bar t, \bar t); X^s)$ is defined in terms of the representatives in $(X^s, d^s)$ as follows. There is $\tilde q :  (-\bar t, \bar t) \to H^s_{loc}(\R)$ which  
 satisfies    that
\begin{equation*} 
(-\bar t, \bar t)\ni t \to  \tilde  q(t)-\tilde q(0) \in L^2(\R), 
\end{equation*}
is weakly continuous   and  
\begin{equation*}\label{Strichartz}   
\Vert \tilde q(\cdot) - \tilde q_{0,\varepsilon} \Vert_{L^4([a,b]  \times \R)} <\infty,  
 \end{equation*}
for some regularized initial data $\tilde q_{0,\varepsilon}$ of $\tilde q(0)$ and for all time intervals $ [a,b] \subset (-\bar t, \bar t)$ with $0\in [a,b] $,
such that the equation \eqref{GP} holds in the distributional sense on $(-\bar t, \bar t)\times\R$
and $\tilde q(t)$ projects to $q(t)$.}

  Let $s>\frac12 $,
 then the above holds for all $\bar t\in\R^+$ and hence the
 Gross-Pitaevskii equation \eqref{GP} is globally-in-time well-posed
 in the metric space $(X^s, d^s)$.
 Furthermore, for any initial data $q_0\in X^s$, there exists $\tau_0 $  depending only on $E^{s}(q_0)$ such that   the unique solution $q\in \cC(\R; X^s)$  of the Gross-Pitaevskii equation \eqref{GP}   satisfies  the following uniform bound
 \begin{equation}\label{ConsEnergy:KL}\begin{split}
&E^{s}_{\tau_0}( q(t)) \leq  2E^{s}_{\tau_0}( q_0),
\quad \forall t\in\R.
\end{split}\end{equation}  
\end{thm}

In the above, we have used the  rescaled energy norm in the metric space $(X^s, d^s)$:
\begin{equation}\label{Estau}
E^s_\tau(q)=\Bigl(\| |q|^2-1\|_{H^{s-1}_\tau }^2
+\| \d_x q\|_{H^{s-1}_\tau }^2\Bigr)^{\frac12},\quad s\geq 0,
\end{equation}
where the rescaled Sobolev norm is defined as
\begin{equation}\label{Sobolev}
\|f\|_{H^{s}_\tau}=\|D^s_\tau f\|_{L^2},
\end{equation} 
and the  rescaled  operator 
$$D_\tau=(-\d_x^2+\tau^2)^{\frac12}$$ 
reads in terms of Fourier transform as 
$$\widehat{D_\tau f}(\xi)=(\xi^2+\tau^2)^{\frac12}\hat{f}(\xi).$$
Notice that when $\tau=2$, $D_2=D=\langle\d_x\rangle$ is defined in \eqref{D} and $E^s_2(q)=E^s(q)$ is the energy norm defined in \eqref{energy,D}.

We aim to show the almost conservation law \eqref{ConsEnergy:KL} for all $s\geq 0$, and in particular we are going to construct a family of conserved energies $\cE^s_\tau(q)$ which is equivalent to $(E^s_\tau(q))^2$ provided  $\tau^{-\frac12}E^0_\tau(q)$ is small.
Our main result reads as follows.
\begin{thm}[GWP and conserved energies in $X^s$, $s\geq 0$]\label{thm:GP} 
Let $s\geq 0$. Then the Gross-Pitaevskii equation \eqref{GP} is globally-in-time well-posed in the metric space $(X^s, d^s)$ in the sense in Theorem \ref{thm:KL}.
 Furthermore, there exist  a constant $C\geq2$  and  a family of  
analytic (in both variables $\tau$ and $q$) energy functionals $(\cE^{s}_{\tau})_{\tau\geq 2}: X^s\mapsto [0,\infty)$, such that
\begin{itemize} 
\item 
 $\cE^{s}_{\tau}(q)$  is equivalent to $(E^{s}_{\tau}(q))^2$ in the following sense:
 If $q\in X^s$ with $\frac{1}{\sqrt\tau } E^0_\tau(q)<\frac{1}{2C},$ then
\begin{equation}\label{Equivalence}\begin{split}
&|\cE^{s}_{\tau}(q)-(E^{s}_{\tau}(q))^2|
\leq  C \frac{E^0_\tau(q)}{\sqrt\tau }  (E^{s}_{\tau}(q))^2,  \quad s\geq 0.
\end{split}\end{equation}   

\item 
$\cE^{s}_{\tau}(\cdot)$, $\tau\geq 2$ is conserved by  Gross-Pitaevskii flow \eqref{GP}.
\end{itemize} 
Correspondingly,  for any initial data $q_0\in X^s$, there exists $\tau_0\geq C$  depending only on $E^{0}(q_0)$ such that   the unique solution $q\in \cC(\R; X^s)$  of the Gross-Pitaevskii equation \eqref{GP}   satisfies  the following energy bound:
 \begin{equation}\label{ConsEnergy}\begin{split}
&E^{s}_{\tau_0}( q(t)) \leq  2E^{s}_{\tau_0}( q_0),
\quad \forall t\in\R.
\end{split}\end{equation}  
\end{thm}
 
The construction of the conserved energy functionals $(\cE^s_\tau)$ is related to the complete integrability fact  (by means of the inverse scattering method) of 
the Gross-Pitaevskii equation \eqref{GP},
which can be viewed as the compatibility condition for the following two ODE systems (see  Zakharov-Shabat \cite{ZS73})
\begin{equation}\label{LaxPair0}\begin{split}
&  u_x = \left( \begin{matrix}
 -i\lambda &  q \\  \bar q & i \lambda 
 \end{matrix} \right) u ,
 \\
 & u_t=i\left(\begin{matrix}
 -2\lambda^2- (|q|^2-1) & -2i\lambda q+ \d_x q
 \\
 - 2i\lambda \ov q- \d_x \ov q & 2\lambda^2+ (|q|^2-1)
 \end{matrix}\right) u,
 \end{split}\end{equation}
 where  $u=u(t,x):\R\times\R\mapsto \C^2$ is the unknown vector-valued solutions.
 The first ODE system can be formulated as the spectral problem $Lu=\lambda u$ of the Lax operator 
\begin{equation}\label{Lax0} L=\left( \begin{matrix} i\partial_x & -i q \\ i \bar q &  -i\partial_x \end{matrix} \right). \end{equation}
If the potential  satisfies $q\in 1+\cS$, the Lax operator is self adjoint with essential spectrum $ (-\infty,-1] \cup [1,\infty) $ and finitely  many discrete eigenvalues in $(-1,1)$ (see \cite{DPVV} and Theorem \ref{thm:energies} below).
For \[  \lambda \in \C \backslash ((-\infty,-1 ]  \cup [1,\infty) ) \] 
we denote by $z$ the square root of $ \lambda^2-1$ with positive imaginary part. The map $ \C \backslash ((-\infty,-1] \cup [1,\infty)) \ni \lambda \to z $ is holomorphic. 

 The construction of the energies $(\cE^s_\tau)$ depends on the  transmission coefficient $T^{-1}(\lambda; q)$ of the Lax operator \eqref{Lax0} (see e.g. \eqref{cEstau0} below for the definition of $\cE^s_\tau$   and see    Section \ref{sec:T} below for   the definition of  $T^{-1}$), and by an abuse of notation we call $T^{-1}$, while not $T$, the transmission coefficient, and we will use the same abuse for the renomarlized transmission coefficient $T_c^{-1}$ below.
 In the classical framework $q\in 1+\cS$,  the logarithm of this transmission coefficient  is shown in \cite{FT} to have an asymptotic expansion as follows
  \begin{equation}\label{lnT:FT0}
 \begin{split}
&\ln T^{-1}(\lambda)
=i\sum_{l=0}^{k-1}  \cH_{l}(2z)^{-l-1}+(\ln T^{-1}(\lambda))_{\geq k+1},
  \quad \Im\lambda>0,
\\
&
 \hbox{ with } |(\ln T^{-1}( \lambda))_{\geq k+1}|=O(|\lambda|^{-k-1})
\hbox{ as }|\lambda|\rightarrow\infty.
 \end{split}
 \end{equation} 
Here $\cH_0$ is the mass  $\mathcal{M} $, $\cH_1$ is the momentum $ \mathcal{P}$, $\cH_2$ is the energy $\mathcal{E}$, and 
\begin{equation} \label{eq:cH3}  
\cH_3 =\Im  \int_{\R} \bigl( \d_x q\d_{xx} \bar q + 3 (|q|^2-1) q \d_x \bar q\bigr)\dx - \mathcal{P}   \end{equation}
which cannot be defined on $X^s$ for any $ s \ge 0$ (since $\cP$ cannot be defined  on $X^s$ for any $ s \ge 0$).

We proved in \cite{KL} that,  
\begin{equation}\label{eq:oldworkexp}  
X^s \ni q \to  \underline{\Tc}^{-1}(\lambda):= \exp\Big(\ln T^{-1}(\lambda)-i\mathcal{M} (2z)^{-1} -i \mathcal{P} (2z (\lambda+z))^{-1}\Big)  
\end{equation} 
defines a continuous map to holomorphic functions if $ s> \frac12$. 
 We expand its logarithm as 
\[ -i\ln\underline{\Tc}^{-1}(\lambda)
\sim   \sum_{l\geq 2}  \underline{H}_l (2z)^{-l-1}, \] 
 where $\underline{H}_{2n}=\cH_{2n}$ and
 \[ \underline{H}_{3} =\Im  \int_{\R} \bigl( \d_x q\d_{xx} \bar q + 3 (|q|^2-1) q \d_x \bar q\bigr)\dx,\]
 which is well-defined on $X^{\frac32}$.

 We prove in this paper that there is a unique extension of a renormalized transmission coefficient 
\begin{equation}   \label{eq:exp}
q \to   T_c^{-1}(\lambda) := \exp\Bigl(  \ln T^{-1}(\lambda)-i\mathcal{M} (2z)^{-1} -i  \Theta (2z(\lambda+z))^{-1} \Bigr)   \end{equation} 
to $X^0$ modulo $ \exp( 2\pi i (2z(\lambda+z))^{-1}\Z)$, or, equivalently, if we fix an element $\tilde 1$ in the fiber of $1$   and $\Theta(\tilde 1) = 0 $, there is a unique continuous extension of \eqref{eq:exp} to the universal covering space of $X^0$.  One obtains an asymptotic expansion
\[ -i \ln T_c^{-1}(\lambda) \sim \sum_{l\geq 1}  H_l (2z)^{-l-1} \] 
on the covering space, 
where $H_1$ is defined in \eqref{eq:H1}, $H_2= \cE$ and $H_{2n}=\underline{H}_{2n}=\cH_{2n}$. Justification of the asymptotic expansion follows from a combination of the techniques in \cite{KL} and in this paper. We do not pursue this here to avoid a distraction by sidelines. 
\begin{thm}[Renormalized transmission coefficient]\label{thm:energies} 
Suppose that $q\in X^0$. Then 
the renormalized transmission coefficient of the Lax operator \eqref{Lax0}: $\lambda \to T_c^{-1}(\lambda)$  is a holomorphic  function on $ \C \backslash ((-\infty, -1] \cup [1,\infty) ) $. The zeroes  are simple and contained  in $(-1,1)$, which
  coincide with the eigenvalues of the Lax operator. 
  
 There exists a constant $C\geq 2$ (as in Theorem \ref{thm:GP}), such that 
 \begin{itemize}
     \item if  $q\in X^s$
 satisfies the smallness condition $\frac{1}{\sqrt\tau } E^0_{\tau} (q)\le \frac{1}{2C} $ for $\tau:=2\Im z\geq 2$, then 
  \begin{equation} \label{eq:energies} \begin{split}
  &\hspace{-.5cm} \Bigl| \frac12\Big( \ln \Tc^{-1}(\lambda) + \ln \overline{\Tc^{-1}(-\bar \lambda)}\Big)
  \\
  &+  \frac i{2z} \int_{\R} \frac1{\xi^2-4z^2}    \Big(     |\widehat{ \d_x q} |^2(\xi) + | (\widehat{|q|^2-1})|^2(\xi)\Big) d\xi\Bigr|
  \leq C(\frac{1}{\sqrt\tau }E^0_\tau(q))^3,
\end{split}\end{equation} 

\item  if $q\in X^s$,  $s \in (\frac12, \frac32)$ never vanishes and satisfies    $\tau_0^{-1-\varepsilon} E_{\tau_0}^{\frac12+\varepsilon}(q) \le\frac{1}{2C}$ for some $ \varepsilon\in (0,\frac12)$ and $\tau_0\geq 2$, then for all $\tau\geq\tau_0$,
\begin{equation} \label{eq:highmomentae}   
\begin{split} \hspace{1cm}& \hspace{-1cm} 
\Bigl| \frac1{2i}  \Big(\ln \Tc^{-1}(\lambda) -\ln  \overline{\Tc^{-1}(-\bar \lambda)} \Big) 
 -\frac{\lambda}{4z^3} H_1 
-\frac {\lambda}{4z^3} \int_{\R} \frac{\xi} {\xi^2-4z^2}        |\widehat{ \d_x q} |^2 d\xi \Bigr|
\\ & \leq C  \tau^{-1-2s} (\tau_0^{-1-\varepsilon} E^{\frac12+\varepsilon}_{\tau_0}  (q)) (E^{s}_{\tau_0} (q))^2,
\end{split} 
\end{equation}  
where $H_1$ is given in \eqref{H1,nonzero} in Theorem \ref{thm:E1}. 

 \end{itemize}
\end{thm} 
\begin{rmk} We will also obtain a more detailed expression for the odd part:
We  take the function $r= \tau^2 D_\tau^{-2} q\in X^2 $ and a non-vanishing function $\tilde q\neq 0$ with $ \tilde q -q \in L^2$, such that
if $\frac{1}{\sqrt\tau } E^0_{\tau} (q)\le \frac{1}{2C} $ for $\tau=2\Im z\geq 2$ and $\Im\lambda>0$, then
 \begin{equation} \label{eq:lowmomentae}
\begin{split} &
 \Bigl|\frac1{2i} \Big(\ln \Tc^{-1}(\lambda) -\ln  \overline{\Tc^{-1}(-\bar \lambda)} \Big)  
 -\frac1{4z^2}  \int_{\R} \frac{\xi(\tau^4+4z^2(2\tau^2+\xi^2))}{(\tau^2+\xi^2)^2(\xi^2-4z^2)}  |\widehat{\d_x q}|^2 d\xi  
 \\
 &
 + \frac{\lambda-z}{2z} \Im  \int_{\R} \bigl( \bar r\d_x r - \partial_x \log \tilde q\bigr) dx  +\frac{(\lambda-z)^2}{4z^2}\Im\int_{\R} (|r|^2-1)(\bar r\d_xr)\dx  
\\ &  
- \frac{ \lambda-z}{4z^2} \int_{x<y} e^{2iz(y-x)} \Big( (|q|^2-1)(x) \Im [ r\d_x\bar r  ] (y)
\\
&\qquad\qquad\qquad + \Im [r\d_x\bar r ](x) (|q|^2-1)(y)\Bigr) dx dy
\\ &  - \frac{\lambda-z}{2z} \int_{x<y} e^{2iz(y-x)} \Big((|q(x)|^2-1)\Im [ \bar r(q-r)](y) \\
&\qquad\qquad\qquad+ \Im [   r(\overline{q-r})](x) (|q(y)|^2-1)\Big)   dx dy\Bigr| 
  \leq C(\frac{1}{\sqrt\tau }E^0_\tau(q))^3.   
\end{split} 
\end{equation} 
If $s\in [0,\frac74]$ and the smallness condition $\frac{1}{\sqrt{\tau_0} } E^0_{\tau_0} (q)\le \frac{1}{2C} $ holds for some fixed  $\tau_0\geq 2$, then for all $\tau\geq\tau_0$, 
\begin{equation}\label{eq:ems}
\hbox{Lefthandsides of \eqref{eq:energies} and \eqref{eq:lowmomentae}}
\leq C\tau^{-\frac32-2s}   E^0_{\tau_0} (q)   (E^s_{\tau_0}  (q))^2.
\end{equation}
\end{rmk} 

\begin{cor}[Conserved energy and momentum sequences] \label{cor:hamiltonian} 
There exist a sequence of conserved energies $H_{2n}$, $n\geq 1$ with 
\begin{equation}\label{eq:energiesweak}    \Big| H_{2n} - (E^n(q))^2 \Big| 
\le c(E^0(q))\, E^0(q) (E^n(q))^2, \end{equation}   
and a sequence of conserved Hamiltonians $ \tilde H_{2n+1}$, $n\geq 1$ with 
\begin{equation}\label{eq:momentum}
\Big|\tilde H_{2n+1} - \Im \int_{\R}  q^{(n)} \bar q^{(n+1)}\dx  \Big| 
\le c(\varepsilon, E^{\frac12+\varepsilon}(q))\, E^{\frac12+\varepsilon}(q) (E^{n+1/2}(q))^2. \end{equation} 
\end{cor} 

 \begin{rmk}
The conserved energy $\cE^s_\tau(q)$ given in Theorem \ref{thm:GP}  here is indeed the same as the conserved energy given in \cite{KL}, which is defined in terms of the real part of the renormalized transmission coefficient located on the imaginary axis (see e.g. \eqref{cEstau0} below): Indeed, if we take the purely imaginary points $(\lambda, z)=(i\sqrt{\tau^2/4-1}, i\tau/2)$, $\tau\geq 2$, then the real parts of the renormalized transmission coefficients in \cite{KL} (see \eqref{eq:oldworkexp}) and here (see \eqref{eq:exp}) coincide.   
\end{rmk}

 \subsection{Ideas of the proofs}  
 \label{subs:Ideas} 
 In order to consider the low regularity case $q\in X^0$,  we first perform a regularisation procedure: We  define the regularisation  $r = \tau^2( \tau^2- \partial_x^2)^{-1}q = \tau^2 D_\tau^{-2}q\in X^2$ for   $q\in X^0$ and $\tau \ge 2$. 
If $ E_\tau^0 (q) \le \tau^{1/2} $ then (see Lemma \ref{lem:regular} below) 
\begin{equation}\label{r,intro}
\|r\|_{L^\infty} \le c \tau, \quad 
\quad\|q-r\|_{L^2}+ \tau^{-1}  \||r|^2-1\|_{L^2}  + \tau^{-1} \|\d_x r\|_{L^2}\le cE^0_\tau(q).
\end{equation}
This regularisation $r$ will play an important role in the proofs.

We will prove Theorem \ref{thm:GP} by use of Theorem \ref{thm:energies}, and mention the proof of Corollary \ref{cor:hamiltonian}  below.
We will also sketch below the proof ideas of Theorem   \ref{thm:homotopy} and Theorem \ref{thm:E1} (whose detailed proofs can be found in Section \ref{subs:Theta}), and of Theorem \ref{thm:energies} (whose detailed proofs can be found in Section \ref{sec:T}).  

 \subsubsection{Idea  of the proof  of Theorem \ref{thm:homotopy}  }  
 
 The construction for the deformation in Theorem \ref{thm:homotopy} is as follows: First we define a deformation to the set of functions which do not have zeros outside an interval $[-R_0,R_0]$.  This is followed by a second essentially linear deformation to the set $ Q$, basically by unwinding the rotation near infinity. Implementing these ideas is more delicate than this (simple) description: Because of the low regularity assumption $q\in X^0$, we have to resort to its regularisation \eqref{r,intro}.
  
 \subsubsection{Idea  of the proof of Theorem \ref{thm:E1}}  
The crucial point in the proof of Theorem \ref{thm:E1} is a further regularisation map $q\mapsto \tilde q$, with the regularisation $\tilde q$ having on zeros.
More precisely, given $q_0\in X^s$, $s\geq 0$, we will construct (see Lemma \ref{lem:tildeq} below)  a  continuous map defined on a small ball $B_\delta^{X^0}(q_0)$ in $X^0$:
 \[  B_\delta^{X^0}(q_0) \ni q \to \tilde q \in X^{s+2} \] 
 so that $ |\tilde q|=1$, $ q-\tilde q \in H^s$.
 Furthermore, with the assumption $\d_x q\in L^1$, we have $ \partial_x \tilde q\in L^1$, and   we  can then define 
\[ \Theta (q)= -\Im \int_{\R} \partial_x \ln \tilde q \dx\in \R / (2\pi\Z).  \] 
Since $X^s$ is homotopy equivalent to $\mathbb{S}^1$, $ \Theta $ can be lifted to a map from the universal covering space to $ \R$.

Direct estimates show then that 
\[ H_1(q) = \Im \int_{\R} \Bigl( ( q- \tilde q ) \partial_x \bar q - \overline{\big( q- \tilde q \big)} \partial_x \tilde q\Bigr) \dx\]  
can be defined on the universal covering space of $X^{\frac12}$, and  it is unique up to the addition of multiplies of $ 2\pi $. 
 
 \subsubsection{Proof of Theorem \ref{thm:GP} by virtue of Theorem \ref{thm:KL} and Theorem \ref{thm:energies}} 
 Thanks to the local-in-time well-posedness result in Theorem \ref{thm:KL},   the equivalence relation \eqref{Equivalence} between the conserved energy $\cE^s_\tau(q)$ and the square of the energy norm $(E^s_\tau(q))^2$  implies immediately the (almost) conservation of the energy \eqref{ConsEnergy} and then the global-in-time well-posedness  in Theorem \ref{thm:GP},  
 by a standard continuation argument (as e.g.  the end of   \cite[Section 1]{KL}), which is omitted here.

It remains to prove the equivalent relation \eqref{Equivalence} and the conservation of $\cE^s_\tau(q)$.
We are going to construct the energy functionals $(\cE^s_\tau(q))_{\tau\geq 2}$ in terms of the time-independent renormalized transmission coefficient, and then show   that \eqref{Equivalence}   is a  consequence of the more precise statements in Theorem \ref{thm:energies}, in the low regularity regime  $s\in [0,2)$. 
   In this paper we focus on the case $ s<2$ in order to avoid a  more technical  presentation. 
   \\
\noindent\textbf{Case $s\in [0,1)$.}
  We first consider the following integral for $\tilde s\in (-1,0)$  and $\tau_0>0$,
 \begin{equation*}  \begin{split} 
 \hspace{1cm} & \hspace{-1cm}  -\frac2\pi \sin (\pi \tilde  s) \int_{\tau_0}^\infty (\tau^2-\tau_0^2)^{\tilde  s}  \tau (\tau^2+\xi^2)^{-1} d\tau ,
\end{split} 
 \end{equation*}  
 which reads as, by using the standard branches of the logarithm and fractional powers   in the cut domain  $ \C \backslash (-\infty,0]$, 
 \begin{equation*}  \begin{split} 
&  \lim_{\varepsilon \downarrow  0 } 
 \Big[ \frac1{\pi i}  \int_{\tau_0}^\infty  (( i\tau- \varepsilon )^2 + \tau^2_0)^{\tilde  s} i\tau (\xi^2+ \tau^2)^{-1}  i d\tau
\\ & \hspace{2cm}   + \frac1{\pi i}    \int_{\tau_0}^\infty ((i\tau +\varepsilon )^2 + \tau^2)^{\tilde  s} i\tau  (\xi^2+ \tau^2)^{-1} i d\tau \Big].
\end{split} 
 \end{equation*}  
Let  $ \gamma$ be the path from $ i \infty -0 $ to $i\tau_0$ and then to $ i\infty+0$,
then we derive
 \begin{equation}\label{eq:contour}  \begin{split} 
 \hspace{1cm} & \hspace{-1cm}  -\frac2\pi \sin (\pi \tilde  s) \int_{\tau_0}^\infty (\tau^2-\tau_0^2)^{\tilde  s}  \tau (\tau^2+\xi^2)^{-1} d\tau  
\\ &  = \frac1{\pi  i } \int_{\gamma}   (z^2+\tau_0^2)^{\tilde  s} z ( \xi^2 -z^2)^{-1} dz    = (\tau_0^2 + \xi^2)^{\tilde  s},\quad \tilde s\in (-1,0),
\end{split} 
 \end{equation}   
where we moved the contour of integration to the real axis for the last equality where only 
  residues contribute since the integrand is odd.  Thus, for $-1<\tilde s<0$  ,
    \begin{equation}\label{Hstau0}\begin{split}   \Vert f \Vert_{H^{\tilde s}_{\tau_0}}^2 
  \, &  = \int_{\R} |\hat f(\xi)|^2 (\tau_0^2 + \xi^2)^{\tilde s} d\xi 
\\ &  = - \frac2\pi \sin(\pi\tilde  s) \int_{\tau_0}^\infty \int_{\R} (\tau^2-\tau_0^2)^{\tilde s} \tau ( \tau^2+ \xi^2)^{-1} |\hat f(\xi)|^2 d\xi d\tau 
 \\ & = - \frac2\pi \sin(\pi\tilde  s) \int_{\tau_0}^\infty (\tau^2- \tau_0^2)^{\tilde s} \tau \Vert f \Vert_{H^{-1}_\tau}^2 d\tau.
\end{split} 
 \end{equation}
 Observe that the limit of the right hand side is $ \Vert f \Vert_{H^{-1}_{\tau_0}}^2$ as $\tilde  s \to -1$. 
 
 For $s\in (0,1)$ and $\tau_0\geq 2$, we define our conserved energies in terms of the real part of the renormalized transmission coefficient given by Theorem \ref{thm:energies} on the imaginary axis as
\begin{equation}\label{cEstau0}
\cE^s_{\tau_0}(q)=\frac2\pi\sin(\pi(s-1))\int^\infty_{\tau_0} (\tau^2-{\tau_0}^2)^{s-1}
\tau^2 \Re \ln\Tc^{-1}\Bigl (i\sqrt{\frac{\tau^2}{4}-1}\Bigr)   d\tau ,
\end{equation}
and for $s=0$, we define
\begin{equation}\label{cEstau00}
\cE^0_{\tau}(q)=-\tau\Re \ln\Tc^{-1}\Bigl (i\sqrt{\frac{{\tau}^2}{4}-1}\Bigr).
\end{equation}

We derive \eqref{Equivalence} for $s=0$ straightforward from the estimate \eqref{eq:energies} with $(\lambda, z)=(i\sqrt{\tau^2/4-1}, i\tau/2)$ and $  \tau \ge \max\{2, 2C (E_\tau^0(q))^2 \} $ in  Theorem \ref{thm:energies}:  
 \begin{equation}   \label{Estau0}  \left| -\cE^0_\tau(q) +   (E^0_\tau(q))^2 \right| \le C \frac{E^0_\tau(q)}{\sqrt{\tau}}  \big( E^0_\tau(q) \big)^2. \end{equation} 
 Similarly, for  $ 0 < s < 1$, and $  \tau_0 \ge \max\{2,2 C (E_{\tau_0}^0(q))^2 \} $ (such that $  \tau \ge \max\{2, 2C (E_\tau^0(q))^2 \} $ holds for all $\tau\geq \tau_0$ trivially), we derive from   \eqref{eq:energies}, \eqref{Hstau0} 
 and the trivial estimates  $ E^0_\tau(q) \le \tau^{-s} E^s_{\tau}(q)$, $ E^0_\tau(q) \le E^0_{\tau_0}(q)$ and $E^s_\tau(q)\leq E^s_{\tau_0}(q)$, that
\[\begin{split} 
 &\Big| -\cE^s_{\tau_0}(q)+ (E^s_{\tau_0} (q))^2   \Big|
 \\
& \leq -\frac2\pi\sin(\pi(s-1))\int^\infty_{\tau_0} (\tau^2-{\tau_0}^2)^{s-1}
\Bigl| \tau^2 \Re \ln\Tc^{-1}\Bigl (i\sqrt{\frac{\tau^2}{4}-1}\Bigr)
+\tau (E^0_\tau(q))^2\Bigr| d\tau
\\&\lesssim  -\frac2\pi\sin(\pi(s-1))\int^\infty_{\tau_0} (\tau^2-{\tau_0}^2)^{s-1} \tau^{\frac12-2s}   d\tau E_{\tau_0}^0(q) (E_{\tau_0}^{s}(q))^2 
\\ & \lesssim   \frac{E^0_{\tau_0}(q)}{\sqrt{ \tau_{0} }}  (E^s_{\tau_0}(q))^{2} . 
\end{split}\]  

 \smallbreak
 
\noindent\textbf{Case $s\in [1,2)$.}
If $s=1$, then $\cE^1_\tau(q)$ is simply the conserved energy \eqref{cEGL}:
 $$
 \cE^1_\tau(q)=(E^1_\tau(q))^2=(E^1(q))^2=\cE=\|(|q|^2-1, \d_x q)\|_{L^2}^2=H_2(q).
 $$
The same calculation as for \eqref{eq:contour} gives for $ 0\le \tilde s<1$
 \begin{equation}\label{eq:contour,larges}   -  \frac2\pi \sin (\pi \tilde s) \int_{\tau_0}^\infty (\tau^2-\tau_0^2)^{\tilde s}  \tau \big[(\tau^2+\xi^2)^{-1}-\tau^{-2}\big]  d\tau 
 +{\tau_0}^{2\tilde s}
 = (\tau_0^2 + \xi^2)^{\tilde s},
 \end{equation}
 and hence for $s\in (1,2)$,
 \begin{equation}\label{Estau0,larges}\begin{split}
     (E^s_{\tau_0}(q))^2
     &= -\frac{2}{\pi}\sin (\pi(s-1))\int^\infty_{\tau_0} (\tau^2-{\tau_0}^2)^{s-1}  \bigl[\tau(E^0_\tau(q))^2 
     \\
&\quad\qquad\qquad\qquad\qquad
-\tau^{-1}(E^1(q))^2\bigr] d\tau    +{\tau_0}^{2(s-1)}(E^1(q))^2.
 \end{split}\end{equation}
Therefore, for $s\in (1,2)$,  we define 
\begin{equation}\label{cEstau1}\begin{split} \hspace{.5cm} & \hspace{-1cm} 
\cE^s_{\tau_0}(q)=\frac2\pi\sin(\pi(s-1))\int^\infty_{\tau_0} (\tau^2-{\tau_0}^2)^{s-1}
\\ &  \times \left( 
\tau^2 \Re \ln\Tc^{-1}\Bigl (i\sqrt{\frac{\tau^2}{4}-1}\Bigr) +  \tau^{-1}   (E^1(q) )^2 \right)   d\tau +{\tau_0}^{2(s-1)}(E^1(q))^2.
\end{split} 
\end{equation} 
The difference  $\mathcal{E}^s_{\tau_0}(q)- (E^s_{\tau_0}(q))^2$ is given  by 
\[ \frac{2}{\pi} \sin( \pi (s-1)) \int_{\tau_0}^\infty (\tau^2-\tau_0^2)^{s-1} \Big( \tau^2 \Re \ln T_c^{-1}\Bigl (i\sqrt{\frac{\tau^2}{4}-1}\Bigr) + \tau (E_\tau^0(q))^2\Big)  d\tau,  \] 
and we estimate the integrand using \eqref{eq:energies} and indeed \eqref{eq:ems} by 
\[  (\tau^2-\tau_0^2)^{s-1}  \tau^{\frac12-2\sigma} E^0_{\tau_0}(q) (E^{\sigma}_{\tau_0}(q))^2 \hbox{ for some }\sigma\in [1,\frac74].      \] 
Integration yields the searched  estimate \eqref{Equivalence} for $s=\sigma \in (1,\frac74]$. If $ \frac74 \le s < 2$ we take $\sigma=\frac74$ above 
\[ ( \tau^2-\tau_0^2)^{s-1} \tau^{-3} E^0_{\tau_0}(q) (E^{\frac74}_{\tau_0}(q))^2,
\]
which decays at the rate $\tau^{2s-5}$ with $2s-5<-1$ as $\tau\rightarrow\infty$,
and the integration with respect to $ \tau$ gives  \eqref{Equivalence}:
\[   \Big|\mathcal{E}^s_{\tau_0}(q) - (E^s_{\tau_0})^2 \Big| \lesssim  \tau_0^{2(s-2)}      E^0_{\tau_0}(q) (E^{\frac74}_{\tau_0}(q))^2  \lesssim  \frac{E^0_{\tau_0}(q)}{\tau_0^{1/2}} (E^s_{\tau_0}(q))^2, \]
where we used  the trivial estimate $E^{\frac74}_{\tau_0}(q)\leq {\tau_0}^{-s+\frac74}E^s_{\tau_0}(q)$ for the second inequality.



 \smallbreak

 In this paper we only consider the case  $ s<2 $, and for $s>1$ we have expanded above  the rescaled energy norm  $E^0_\tau(q)$   with respect to $\tau$ in the  formulation \eqref{Estau0,larges} of $(E^s_\tau(q))^2$, and similarly we expanded the real part $\Re\ln\Tc^{-1}$  on the imaginary axis $(\lambda, z)=(i\sqrt{\tau^2/4-1}, i\tau/2)$ in the definition \eqref{cEstau1} of $\cE^s_\tau(q)$.
 For  $s\geq 2$  one has to expand
  these terms up to higher-order terms exactly as in \cite{KL}. 
  Then one may replace the norm $ \Vert |q|^2-1 \Vert_{l^2 DU^2} + \Vert \d_x q \Vert_{l^2 DU^2}$ in \cite{KL} by $E^{0}(q)$ here, which is \eqref{Equivalence} for general $s\geq 0$.
 This yields immediately the estimates \eqref{eq:energiesweak} in Corollary \ref{cor:hamiltonian} as a special case. 
 We do not go into details to avoid  a more technical  presentation.

  \subsubsection{Ideas of the proof of Theorem \ref{thm:energies}}
 Let $q\in 1+\cS$.  Let   $T^{-1}(\lambda; q)$ be the time independent transmission coefficient of the associated Lax operator to the one dimensional Gross-Pitaevskii equation \eqref{GP}, and $\Tc^{-1}(\lambda; q)$ be the renormalized transmission coefficient modulo the mass and the asymptotic  phase change as in \eqref{eq:exp}:
\begin{equation} \label{eq:renorm} 
 \Tc^{-1}(\lambda)=   T^{-1}(\lambda)e^{-i\cM (2 z )^{-1}-i   \Theta (2z(\lambda+z) )^{-1}}.
\end{equation}
This renormalized transmission coefficient $\Tc^{-1}$ 
is defined on $\{q\in X^0: |q|^2-1, \d_x q\in L^1, \, q\neq 0\}$, and has a unique continuous extension to $X^0$ modulo $\exp( 2\pi i(2z(\lambda+z))^{-1}\Z)$, or equivalently, on the universal covering space of $X^0$
 with values in $ \mathbb{C} $.

We approximately diagonalize the Lax operator using the regularisation $r$ (see \eqref{r,intro} above) and $\zeta=\lambda+z$ with $\Im\zeta>0$ (see \eqref{eq:approxdiag} below), and obtain an expansion 
\begin{equation*} 
T_c^{-1}(\lambda) =e^{\Phi(\lambda)} \sum_{n=0}^\infty T_{2n}(\lambda),
\end{equation*}
where the phase correction function $\Phi$ reads explicitly as
\begin{equation}\label{Phi}
   \Phi(\lambda)= -\int_{\R} \frac{ i\zeta \big[ |r|^2-1 + 2\Re (\bar r (q-r)) \big]  -  \bar r \d_x r}{|r|^2-\zeta^2}\dx-\frac i{2z}\cM  -\frac{i}{2z(\lambda+z)}\Theta,
\end{equation}
and   $T_{2n}$ are $2n$ dimensional integrals essentially from a Picard iteration from $x=-\infty$ which can be found in \eqref{T2n} below.  
A direct estimate gives, e.g. on the imaginary axis $(\lambda, z)=(i\sqrt{\tau^2/4-1}, i\tau/2)$ and $  \tau \ge \max\{2,   2C(E_\tau^0(q))^2 \} $,
\[ \Big| \sum_{n=2}^\infty T_{2n}(\lambda)  \Big|
\leq C(\frac{E^0_\tau(q)}{\sqrt{\tau}})^4.
\] 
such that if we take the logarithm: 
\[ \Big| \ln T_c^{-1} -(\Phi+ T_2) \Big|
\leq C(\frac{E^0_\tau(q)}{\sqrt{\tau}})^4.
\] 
Careful estimates give 
\[ \Bigl| \Re (\Phi+T_2)\Big|_{\lambda= i\sqrt{\tau^2/4-1}}  + \frac{1}{\tau }(E^0_\tau(q))^2 \Bigr|
\leq C(\frac{E^0_\tau(q)}{\sqrt{\tau}})^3,
\]  
which implies the energy estimate \eqref{eq:energies} of Theorem \ref{thm:energies} on the imaginary axis.

 The odd conserved Hamiltonians $\tilde H_{2n+1}$ in \eqref{eq:momentum} in Corollary \ref{cor:hamiltonian} are defined by the asymptotic series 
 \[ \frac1{2i} \Big( \ln T_c^{-1} ( \lambda)- \ln \overline{T_c^{-1} (- \bar \lambda)} \Big) \sim  \frac{\lambda}z  \sum_{n=0}^\infty \tilde H_{2n+1} (2z)^{-2(n+1)}. \] 
 The claim follows by combining the estimates of \cite{KL} with \eqref{eq:highmomentae}.

\subsection*{Organisation of the paper} 
We will discuss the topology of the metric space $(X^s,d^s)$ in Section \ref{subs:Theta} where we will prove Theorem \ref{thm:homotopy} and Theorem \ref{thm:E1}.  Theorem \ref{thm:energies}   will be proven in Section \ref{sec:T}.

 \section{The topology of \texorpdfstring{$X^s$}{X\textasciicircum s}}\label{subs:Theta}
 
 In this section we prove  Theorem \ref{thm:homotopy} 
 that $Q$ is a strong deformation retract 
 and Theorem  \ref{thm:E1} about the asymptotic phase shift $ \Theta$ and the renormalized momentum $H_1$. The first step consists in estimates for an approximation, which is nontrivial since $X^s$ is not a linear function space.

\subsection{The regularization: Estimates for \texorpdfstring{$r$}{r}} 
\label{subs:rentrans}

One of the difficulties of dealing with $X^s$ with $ s \le \frac12 $ is that $H^s(\R)$ is not an algebra. 
We regularize $q$ by $ r = \tau^2 D_\tau^{-2}q$
and derive estimates for the crucial functions 
$ |r|^2-1$, $ r'$ and $ q-r$ in terms of $E^0_\tau (q)$   in Lemma \ref{lem:regular} below. 

\begin{defi} 
Let $ \tau \geq  2$, $ s \in \R $ and $1<p< \infty$. We define 
\[ \Vert f \Vert_{W^{s,p}_\tau(\R)} =   \Vert D_\tau^s f \Vert_{L^p(\R)} \] 
where $D_\tau=(-\d_x^2+\tau^2)^{\frac12}$ is defined by the Fourier multiplier as in \eqref{Sobolev}.
\end{defi} 

The case $s=-1$ is of particular interest. 
Let $ j_\tau = \chi_{\{ x< 0 \} }   e^{\tau x} =\left\{\begin{array}{ll} e^{\tau x}& \hbox{ if }x<0 \\ 0&\hbox
{ if } x\geq 0\end{array}\right.$. Then   for $ 1<p< \infty$, we have by the above definition
\begin{equation}\label{W-1,j}
\Vert f \Vert_{W^{-1,p}_\tau(\R)}=\|D_\tau^{-1}f\|_{L^p(\R)} \sim 
\Vert j_\tau\ast  f \Vert_{L^p(\R)}. \end{equation}

\begin{lem} \label{lem:regular} 
Let $ \tau \geq 2$ and $ q \in X^0$.
There exists $r$ so that   $r \in X^2$ and
\begin{equation}\label{eq:estimate:r}  \begin{split} 
&C_p^{-1}\Vert q' \Vert_{W^{-1,p}_\tau} \leq \,   \Vert q-r \Vert_{L^p}  + \tau^{-1} \Vert r' \Vert_{L^p}  \leq C_p    \Vert q' \Vert_{W^{-1,p}_\tau},  \quad \forall p\in (1,\infty),   \\
&\Vert r \Vert_{L^\infty}\,    \le C\tau  \big( 1 + \tau^{-1/2}    E_\tau^0(q)\big),  \\ 
&\Vert |r|^2-1 \Vert_{L^p}\,    \le C_p  \tau  \bigl(  1 + \tau^{-1/2}  E^0_\tau(q)\bigr)    \Vert\bigl( |q|^2-1, q'\bigr) \Vert_{W^{-1,p}_\tau},
\quad\forall p\in [2,\infty),
\end{split}
\end{equation}
where $C_p>0$ is some universal constant depending only on $p$. 
\end{lem}

In particular  if $q \in X^0$, then $r \in X^2$,  with the energy 
$$ E^1(r)=\bigl\| (|r|^2-1, r')\bigr\|_{L^2}\leq C  \bigl(  1+ E^0 (q)\bigr) E^0 (q).$$

\begin{proof}  
We choose 
\[ r = \tau^2 D_\tau^{-2} q=\tau^2(-\d_{xx}+\tau^2)^{-1}q, \] 
and obtain by a multiple use of H\"ormander's multiplier theorem the first line of \eqref{eq:estimate:r} for $p\in (1,\infty)$ by 
\begin{align*}
    & \Vert r' \Vert_{W^{-1,p}_\tau} = 
\Vert  \tau^2 D_\tau^{-3}  q' \Vert_{L^p} \le C_p\Vert D_\tau^{-1} q' \Vert_{L^p},
\\
& \Vert r' \Vert_{L^p}\,  =    \Vert \tau^2D_\tau^{-2} q' \Vert_{L^p}
\le C_p \tau \Vert D_\tau^{-1} q' \Vert_{L^p}, 
\\
& \Vert r-q \Vert_{L^p} = \Vert   D_\tau^{-2} \d_x  q' \Vert_{L^p}
\le C_p \Vert D_\tau^{-1} q' \Vert_{L^p}, 
\\
&\Vert D_\tau^{-1} q'  \Vert_{L^p}\,  
   \le  \Vert D_\tau^{-1}r' \Vert_{L^p} + \Vert D_\tau^{-1} (q-r)' \Vert_{L^p} 
  \le C_p \tau^{-1} \Vert   r' \Vert_{L^p} +C_p \Vert q-r \Vert_{L^p}. 
\end{align*}     

\smallbreak  
We turn to the proof of the second line in \eqref{eq:estimate:r}.
Let $I$ be an interval with the center $x_0\in\R$ and the length $\tau^{-1}$.
Let $J\supset I$ have the same center $x_0\in\R$ and the length $3\tau^{-1}$.
Let $\eta \in C^\infty_c(\R; [0,1])$ be supported in $(-1,1)$, identically $1$ on $[-1/2,1/2]$,   and $\eta_{\tau} = \tau \eta( \tau x) $, then
  \[ \begin{split}  \int_I |q|^2 dx  
  &= \int_I \eta(\tau(x_0-x)) |q|^2 (x)\dx 
  \\&\leq \int_J \eta(\tau(x_0-x)) (|q|^2-1) (x)\dx+\int_J \eta(\tau(x_0-x)) \dx 
  \\ &  \lesssim  \Vert |q|^2 -1 \Vert_{H^{-1}_\tau }  \tau^{-1}  \Vert \eta_\tau \Vert_{H^{1}_\tau} + \tau^{-1}  \\ & 
   \lesssim  \tau^{1/2}  \Vert |q|^2 -1 \Vert_{H^{-1}_\tau} +  \tau^{-1} .\end{split}   \] 
  Hence, by use of $\eqref{eq:estimate:r}_1$: $\|q-r\|_{L^2}\lesssim \|q'\|_{H^{-1}_\tau}$, 
 \begin{equation} \label{eq:rl2} 
 \begin{split}
 \Vert q \Vert_{L^2(I)} + \Vert r \Vert_{L^2(I)} 
& \leq     \Vert r-q \Vert_{L^2(I)} + \Vert  q \Vert_{L^2(I)}
\\
& \lesssim   \tau^ {1/4}\Vert |q|^2-1 \Vert_{H^{-1}_\tau}^{1/2} +  \Vert q' \Vert_{H^{-1}_\tau} +  \tau^{-1/2}  . \end{split}\end{equation}  
 For the  second estimate $\eqref{eq:estimate:r}_2$ we observe that 
 \[ \tau^2 r -r_{xx} = \tau^2 q, \] 
 and hence by Sobolev's inequality
 \[\begin{split}  \Vert r \Vert_{L^\infty(I)}\, & \lesssim \tau^{1/2}\|r\|_{L^2(I)}+\tau^{-3/2}\|r''\|_{L^2(I)} \\ & \lesssim      \tau^{1/2}  \big( \Vert r \Vert_{L^2(I)}+ \Vert q \Vert_{L^2(I)}\big) 
 \\ & \lesssim \tau+\tau ^{1/2} E^0_\tau(q). 
 \end{split} \]    
 
\medskip 
Before we address the third line of \eqref{eq:estimate:r} we   claim the Sobolev inequality (recalling   $j_\tau=\chi_{\{x<0\}}e^{\tau x}$ and $\Vert j_\tau*f\Vert_{L^1}$ can be used to define the norm  $\|f\|_{W^{-1,1}_\tau}$)      \begin{equation} \label{eq:interpolation} \Vert f \Vert_{L^p(I)} \lesssim  \tau^{-1/p} \Vert f' \Vert_{L^1(I)}  + \tau^{2-1/p} \Vert (1+\tau^2|x-x_0|^2)^{-1} (j_\tau * f)(x) \Vert_{L^1 (\R)}.
  \end{equation}  
    By scaling and translation it suffices to prove inequality \eqref{eq:interpolation} for $\tau =1$ and $x_0=0$. This simplifies the notation. By the fundamental theorem of calculus, if $I=[-1/2, 1/2]$ is an interval of length $1$ and $ \eta \in C^1_b(I)$
 with integral $1$ (we can choose $ \eta $  so that 
 $ \Vert \eta+ \eta_x  \Vert_{L^\infty}  \le  6 $), then
\begin{equation}\label{Sobolev:g}\begin{split} \Vert g \Vert_{L^p(I)} \, & \le  \Vert g \Vert_{L^\infty(I)}
 \le \Vert g' \Vert_{L^1(I)} + \Big|\int_I g \eta \dx \Big|
 \\ & = \Vert g' \Vert_{L^1(I)} + \Big| \int_{\R} (j_1* g) (\partial_x +1 ) \eta \dx    \Big| 
\\ &  \le  \Vert g' \Vert_{L^1(I)} + 6\Vert j_1\ast g \Vert_{L^1(\R)}. 
 \end{split} 
 \end{equation}
 This is almost what we need, up to  the nonlocality of the  second term  on the right hand side of \eqref{eq:interpolation}. 
  Notice that by the exponential decay property of the kernel $j_1(x)=\chi_{\{x<0\}}e^{x}$, we have  
 \begin{align}\label{Equiv:rho}
     \|j_1\ast (\rho  f)\|_{L^1(\R)}\sim \|\rho  (j_1\ast f)\|_{L^1(\R)},
 \end{align}
 where $\rho$ is chosen to be continuous and positive  so that $ \rho=1$ on $I$ and 
 \[ \rho(x) = (1+ |x-x_0|^2)^{-1}\qquad \text{ for } |x-x_0| \ge 3/2.   \]  
 Indeed, in order to show $\|\rho(j_1\ast f)\|_{L^1}\lesssim \|j_1\ast(\rho f)\|_{L^1}$, we take   $g = j_1 * (\rho  f)$ resp. 
 \[ f = \rho^{-1} ( \partial_x - 1) g, \] 
and it suffices to show 
 \[ \Vert \rho \bigl( j_1 * ( \rho^{-1} (\partial_x -1 ) g) \bigr) \Vert_{L^1} \lesssim \Vert g \Vert_{L^1} .\] 
 Similarly, in order to show $\|j_1\ast (\rho f)\|_{L^1}\lesssim \|\rho(j_1\ast f)\|_{L^1}$, it suffices to show  
 $$\|j_1\ast \bigl(\rho (\partial_x-1) \rho^{-1}g\bigr)\|_{L^1}\lesssim \|g\|_{L^1}.$$
If we exchange the orders of $ \rho^{-1}$ and the derivative $(\d_x-1)$, then one arrives at  the identity, and hence it suffices to bound the commutators $\rho j_1\ast [\rho^{-1}, \d_x]g$ and $j_1\ast (\rho [\d_x,\rho^{-1}]g)$ as follows
 \[ \Vert \rho j_1 * ( \rho^{-1} \frac{\rho'}{\rho}  g)  \Vert_{L^1} \lesssim \Vert g \Vert_{L^1},
 \quad \|j_1\ast(\frac{\rho'}{\rho}g)\|_{L^1}\lesssim \|g\|_{L^1}. \] 
The second inequality holds obviously by Young's inequality and it suffices to show 
\begin{align*}
    \|\rho j_1\ast (\rho^{-1}h)\|_{L^1}\lesssim \|h\|_{L^1}.
\end{align*}
This is ensured by the boundedness  of the integral operator $h\mapsto \int_{\R} K(x,y)h(y)\dy$ in $L^1$ with the integral kernel
$$
K (x,y) = \left( \frac{\rho(x)}{\rho(y)}\right)  1_{\{y<x\}} e^{-(x-y)}
$$
  satisfying Schur's criterium  
  \[ \sup_x \Vert K  (x, . ) \Vert_{L^1(\R) }
  + \sup_y \Vert K  (.,y) \Vert_{L^1(\R)}<\infty.\] 
  Thus \eqref{Equiv:rho} follows.
Finally we apply \eqref{Sobolev:g} with $g=\rho f$,  \eqref{Equiv:rho} and rescaling to derive \eqref{eq:interpolation}. 
  
  \medskip   
 We turn to the proof of the last line of \eqref{eq:estimate:r}, by use of  \eqref{eq:rl2},  \eqref{eq:interpolation} and \eqref{Equiv:rho}. Let  $p\ge 2$, then
 \[ \begin{split} \Vert |r|^2-1 \Vert_{L^p(I)}\,
 &  \lesssim  \tau^{-\frac1p}  \Vert \bar r r' \Vert_{L^1(I)}+ \tau^{2-\frac1p}   \Vert (1+\tau^2|x-x_0|^2)^{-1} \bigl( j_\tau *  (|r|^2-1) \bigr)(x) \Vert_{L^1(\R)  }
 \\ & \lesssim \tau^{-\frac12}   \Vert  r \Vert_{L^2(I)} \Vert  r' \Vert_{L^p(I)}  + \tau^{1-\frac1p} \Vert (1+ \tau^2 |x-x_0|^2)^{-1} (  |q|^2-|r|^2)  \Vert_{L^1(\R)}
 \\ & \qquad 
 + \tau^{2-\frac1p} \Vert (1+\tau^2|x-x_0|^2)^{-1}\bigl(  j_\tau *  (|q|^2-1) \bigr)(x) \Vert_{L^1(\R)}
 \\ & \lesssim  \tau^{-\frac12}   \Vert  r \Vert_{L^2(I)} \Vert  r' \Vert_{L^p(I)} \\ &\qquad + \tau^{\frac12}     \mathop{\sup}\limits_{I', |I'|=\tau^{-1}}  \Vert (q,r) \Vert_{L^2(I')}   \Vert (1+\tau^2 |x-x_0|^2)^{-\frac1p} (q-r)  \Vert_{L^p(\R)} \\ & \qquad +  \tau \Vert (1+
 \tau^2 |x-x_0|^2)^{-1/p} \bigl( j_\tau *(|q|^2-1 )\bigr)(x) \Vert_{L^p(\R)}. 
 \end{split} 
 \] 
 We have used (the rescaled version of) \eqref{Equiv:rho}   for the second  inequality with $\rho_\tau(x):=(1+ \tau^2|x-x_0|^2)^{-1}$:
 \[ 
 \begin{split}
     \tau^{2-\frac1p}\|\rho_\tau(j_\tau \ast (|q|^2-|r|^2))\|_{L^1(\R)}\, & \lesssim \tau^{2-\frac1p}\|j_\tau\ast (\rho_\tau (|q|^2-|r|^2)) \|_{L^1(\R)}
     \\ & \leq \tau^{1-\frac1p}\| \rho_\tau  (|q|^2-|r|^2) \|_{L^1(\R)},
 \end{split}
 \] 
 which has been further bounded in the last inequality by H\"older's inequality and change of variables $\tau x\mapsto x$  as
 \begin{align*}
   &  \tau^{\frac{1}{p'}} \Vert {\rho_\tau}^{\frac1{p'}}  |(q,r)|   \Vert_{L^{p'}(\R)}  \Vert {\rho_\tau}^{\frac1p}  (q-r)\Vert_{L^p(\R)} 
     =  \tau^{\frac{1}{p'}} \Bigl( \int_{\R}\rho_\tau |(q,r)|^{p'} \dx\Bigr)^{\frac{1}{p'}}\Vert {\rho_\tau}^{\frac1p}  (q-r)\Vert_{L^p(\R)} 
     \\
&   \lesssim   \tau^{\frac{1}{p'}}
\Bigl( \mathop{\sup}\limits_{I', |I'|=\tau^{-1}}   \Vert (q,r) \Vert_{L^2(I')}\Bigr)
\tau^{-(\frac12-\frac1p)} \Bigl( \sum_{n\in \Z} \|\rho\|_{L^{\frac{1-\frac1p}{\frac12-\frac1p}}([n,n+1])}\Bigr)^{\frac{1}{p'}}\Vert {\rho_\tau}^{\frac1p}  (q-r)\Vert_{L^p(\R)}
\\
&\lesssim \tau^{\frac12} \mathop{\sup}\limits_{I', |I'|=\tau^{-1}}   \Vert (q,r) \Vert_{L^2(I')}\Vert {\rho_\tau}^{\frac1p}  (q-r)\Vert_{L^p(\R)}  .
 \end{align*}  
 We raise the inequality to the power $p$ and  sum over  intervals  $I$
 to obtain
 \[ \Vert |r|^2-1 \Vert_{L^p} \le  C \tau  \bigl( 1+ \tau^{-1/2} E^0_\tau(q)\bigr)\Big( \Vert q' \Vert_{W^{-1,p}_\tau} + \Vert |q|^2-1 \Vert_{W^{-1,p}_\tau }\Big) . \]

  The proof for the estimates \eqref{eq:estimate:r} is finished. \end{proof}

\subsection{Proof of Theorem \ref{thm:homotopy} } 
 
\begin{proof}[Proof of Theorem \ref{thm:homotopy}]
We will assume $s=0$ for notational simplicity and clarity, but we will  give an argument which immediately applies to $ s \in [0,\infty)$. We define $ r = \tau^2 (\tau^2-\partial^2)^{-1} q=h\ast q $, $h:=\frac{\tau}2 e^{-\tau|x|}$  as in Lemma \ref{lem:regular}. 
We fix $\tau$ large enough such that the regularised solitons $r_c:=h\ast q_c$  satisfy
\begin{equation}\label{rc,R}
|r_c(R)|\in [\frac34,1],\quad \forall |R|\geq 1,\quad \forall c\in [-1,1].
\end{equation}
The map $X^0\ni q\mapsto r\in X^2$ is Lipschitz continuous (see Lemma 2.1 \cite{KL}).
 
We are going to construct the deformation from  $X^0$ to $Q$ essentially in two steps:
We first construct a deformation $\Xi_1(t,q)$ defined on $(t,q)\in [0,1]\times X^0$, such that some regularised function $\tilde r$ of $q\in \{\Xi_1(1,p)|p=\Xi_1(0,p)\in X^0\hbox{ and   away from }Q: d^0(p, Q)>2\delta\}$   has no zeros outside the fixed space interval $[-1,1]$. 
At the same time we   keep $Q$ fixed under  $\Xi_1$: $Q=\Xi_1(t,Q)$. 

 We then investigate the asymptotic behaviours of the regularised dark solitons $r_c$ of $q_c\in Q$, such that the regularised function $\tilde r$ of $q\in \{\Xi_1(1,p)|p\in X^0\hbox{ and close to }Q: d^0(p, Q)<2\delta\}$   has no zeros outside some large space interval $[-R,R]\supset [-1,1]$.
 We construct the second deformation $\Xi_2$ by unwinding the rotations at infinity.
 \smallbreak
 
 \noindent\textbf{Step 1: Construction of $\Xi_1$.}
We first have the following observation. Given an interval $I$ we define  
\[ E^1(r;I) = \Bigl( \int_I \bigl( |r'|^2 + (|r|^2-1)^2\bigr) \dx \Bigr)^{\frac12}.  \] We assume that  $\varepsilon$ is sufficiently small so that $|r(x)|\in (\frac12,2)$ if $E^1(r;(x-1,x+1)))< 2\varepsilon$.
  For $q\in X^0$ with $ E^1(r) >  \varepsilon $, we can define  a continuous map 
$$ q \to (a_-, a_+)\in\R^2$$ 
where 
\[
\begin{split} \hspace{2cm} & \hspace{-2cm} 
 E^1(r; (-\infty, a_-+1))+ \frac{\varepsilon}2 (1+ \tanh(a_-+1)) \\ &  = E^1(r;(a_+-1,\infty))+ \frac{\varepsilon}2 (1-\tanh(a_+-1))   = \varepsilon. \end{split} 
 \] 
 Then $a_\pm $ are uniquely defined since the functions are strictly monotone, and $ a_-+1 < a_+-1$ since otherwise 
 \[ \begin{split}\hspace{2cm} & \hspace{-2cm} \varepsilon < E^1(r; (-\infty, a_-+1))+  E^1(r;(a_+-1,\infty))
 \\ & = \varepsilon + \frac{\varepsilon}2 ( \tanh(a_+-1) - \tanh(a_-+1)) \le  \varepsilon, 
 \end{split} 
 \] 
 which is impossible. We want to extend the definition of $a_\pm $ as continuous functions $x_\pm$ on the whole of $X^0$.
Let 
\[ \rho_0(t) = \left\{ \begin{array}{cl} 0 & \text{ if } t \le 1 \\ 
t-1 & \text{ if } 1< t \le 2 \\
1 & \text{ if } t >2. \end{array} \right. \]   
 We define  
\[ x_+(q) = 1+\rho_0( E^1(r)/\varepsilon) (a_+(q)-1), \qquad  x_-(q) =-1+ \rho_0(E^1(r)/\varepsilon) (a_-(q)+1)  \]  
which is now continuous on $X^0$.
Then  $x_\pm = \pm 1$ if $ E^1(r) \le \varepsilon$, and
$$ |r(x)| \in (\frac12,2),\quad \forall x\in (-\infty, x_-)\cup (x_+,\infty),$$
and hence the function $r$ does not vanish outside $(x_-,x_+)$ with 
$$
x_+-x_-\geq 2.
$$

\smallbreak

Given $ \delta >0$ we define a continuous function $\rho: X^0\mapsto [0,1]$ as follows:
\[ \rho(q) = \rho_0 ( d^0(q, Q)/\delta ),\hbox{ where }Q=\{ q_c\,|\, c\in [-1,1]\}. \] 
Notice that $\rho(q_c)=0$, $\forall q_c\in Q$, and $\rho(q)=1$ for $q$ away from $Q$: $d^0(q,Q)>2\delta$.
  We define the first deformation via 
  \begin{equation}\label{eq:firstdeformation}  \bigl( \Xi_1(t,q)\bigr)(y) =  q\Bigl( (1+ t \rho(q)   (x_+-x_-)/2)y +   t\rho(q) (x_++x_-)/2  \Bigr),\quad t\in [0,1]. \end{equation}  
   It is the identity if $t=0$. It is clearly  continuous and the identity on $Q$.
 Furthermore, since $t\rho(q)$ take  values in $[0,1]$ and $x_+-x_-\geq 2>0$, the function $x=x(t,y)=   (1+ t \rho(q)   (x_+-x_-)/2)y +   t\rho(q) (x_++x_-)/2$ as the argument for the $q$-function above in \eqref{eq:firstdeformation} satisfies
 \begin{align*}
    &x\geq 1+t\rho(q)x_+\hbox{ if }y\geq 1,\\
    &x\leq -1+t\rho(q)x_-\hbox{ if }y\leq -1.
 \end{align*} 
We  denote the composition of $r(x)$ with this transformation of coordinates $y\mapsto x(t,y)$ by $\tilde r$, 
\begin{equation} \label{eq:tilder} 
\tilde r(t,y)=  r\Bigl( (1+ t \rho(q)   (x_+-x_-)/2)y +   t\rho(q) (x_++x_-)/2 \Bigr). 
\end{equation} 
We observe that for $q\in X^0$ away from $Q$ with $d^0(q, Q)>2\delta$, 
\begin{equation} \label{eq:tilder,away}\begin{split} 
&\tilde r(1,y)=r\bigl( (1+   (x_+-x_-)/2)y +    (x_++x_-)/2\bigr)
\\
&\hbox{ with }|\tilde r(1,y)|\in (\frac12,2) \hbox{ for }|y|\geq 1.
\end{split}\end{equation}
For $q\in X^0$ close to $Q$: $ d^0(q, Q)<\delta$, 
\begin{equation}\label{eq:tilder,close}
 \tilde r(t,y)=r(y).
 \end{equation}

We are going to choose $\delta>0$ small enough in next step such that the above \eqref{eq:tilder,away} holds also for $q$ close to $Q$: $d^0(q,Q)<2\delta$ if $|y|\geq R_0$ for $R_0>0$ large enough.

\medskip
\noindent\textbf{Step 2. Construction of $\Xi_2$.}
We first investigate  the regularised dark solitons $r_c(R):=(h\ast q_c)(R)$, $h=\frac{\tau}2 e^{-\tau|x|}$,   for large $|R|$   in the following lemma. 
 \begin{lem} \label{lem:inj} 
 There exists $R_0>0$ so that for all $R\geq R_0$, the  map 
 \[  (-1,1]\ni c \to  \phi(c) :=     r_c(R)/  r_c(-R)  \] 
 is injective  and 
 \begin{equation} \label{eq:uniform}   \frac34 \le  |r_c(R) |, |r_c(-R)| \le 1,
 \quad \forall c\in (-1,1].
 \end{equation} 
\end{lem} 
 \begin{proof}
By virtue of \eqref{rc,R}, the bound \eqref{eq:uniform} is clearly achieved by choosing $R_0\geq 1$. 
 By construction $ r_c(x) = - \overline{r_c(-x) } $
 and  hence 
 \[  r_c(R)/  r_c(-R) =  - \left( r_c(R)/| r_c(R)| \right)^2. \] 
 It suffices to prove that 
\begin{equation} \label{eq:homo}  \partial_c \arg r_c(R) >0, \quad \lim_{c\to 1} \arg r_c (R) = \frac\pi2, \quad \lim_{c\to -1} \arg r_c(R) =-\frac \pi2.   \end{equation} 
The limits $ c \to \pm 1$ are obvious and we will prove that the derivative $ \partial_c \arg  r_c (R) $
never vanishes for $ R \ge R_0$.  
For fixed $R>0$, wse compute first
\[ 
\begin{split} 
\partial_c q_c(R) \, & = \partial_c \Big[ic + \sqrt{1-c^2} \tanh( \sqrt{1-c^2}R ) \Big]
\\ & = i -cR \sech^2(\sqrt{1-c^2} R) - \frac{c}{\sqrt{1-c^2}}   \tanh(\sqrt{1-c^2} R). 
\end{split} 
\] 
and hence, if $ -1 < c < 1$ 
\[ 
\begin{split} 
\partial_c \arg q_c(R)\, &  = \Im \frac{\partial_c   q_c(R)}{q_c(R)}
\\ & = \frac{ \frac1{\sqrt{1-c^2}}  \tanh(\sqrt{1-c^2}R) + c^2R\sech^2(\sqrt{1-c^2}R)    }{ 1-(1- c^2) \sech^2(\sqrt{1-c^2} R) } 
\\ & \ge \frac12  \Big( \tanh(\sqrt{1-c^2} R) + c^2 \sqrt{1-c^2}  R  \sech^2(\sqrt{1-c^2} R)\Big) 
\\ & \ge \mu >0 
\end{split} 
\] 
for some $ \mu >0 $ if  $R >1$. This implies the claim for $r_c$ if $R$ is sufficiently large since 
$ \lim_{R\to \infty}   |q_c'(R)|=0$.  
 \end{proof} 

Recall the definition of $\tilde r(t,y)$ in  \eqref{eq:tilder} and the fact \eqref{eq:tilder,away}  for $q$ away from $Q$.  
We choose $R_0\ge 1$ large enough and $\delta$ small enough, so that (by virtue of the continuity of the map $X^0\ni q\mapsto r\in X^2$)   for all $q\in X^0$ close to $Q$: $d^0(q,Q)<2\delta$, 
\begin{align*}
&  |\tilde r(1, y) |=\Big|r\Bigl( \big(1+\rho_0(\frac{d^0(q,Q)}\delta)\frac{x_+-x_-}2\big)y
+\rho_0(\frac{d^0(q,Q)}\delta)\frac{x_++x_-}2\Bigr)\Big|  \in  [\frac12,2] ,
\\
&\quad \forall y\not\in [-R_0,R_0].
\end{align*}
Thus $\tilde r(1,y)$ does not have zeros for all $q\in X^0$ and $|y|\geq R_0$.
 
  \smallbreak

In the following we simply take  $R=R_0$ and denote (with an abuse of notations) $\tilde r(1,y)$ by $\tilde r(y).$
Since the map
\[\beta:  Q \ni q_c \to\frac{r_c(R)}{r_c (-R)}  \in \mathbb{S}^1 \] 
is a homeomorphism, we define the retraction $B$ by
\[      B(q)= \beta^{-1} \left(   \frac{\tilde r(R)}{|\tilde r(R)|} \frac{\overline{\tilde r(-R)}}{|\tilde r(-R)|}\right): X^s\mapsto Q . \] 
This is a continuous map from $X^s$ to $Q$ which is the identity on $Q$.  

Recall the first deformation \eqref{eq:firstdeformation}.
 We fix the representative in the equivalence class of $q \in X^s$ by requiring  
 $$ \tilde r(-R)\in (\frac12,2).$$
 There exists a smooth $ \alpha(y)=\alpha( y; \tilde r)$ supported  on $(-\infty, -R+1) \cup (R-1,\infty)$ with bounded derivatives  so that 
 \[ \tilde r(y) = \left\{\begin{array}{cl}  \exp(i\alpha(y)) b(y) & \text{ if } y < -R \\ 
\exp(i \alpha(y)) b(y) \frac{\tilde r(R)}{|\tilde r(R)|} & \text{ if } y >R \end{array} \right.  \] 
where $ \alpha(-R)= \alpha(R)=0 $ and $ b-1 \in  L^2$ are real valued with $ \frac12 \le b \le 2$. 
  We define the second deformation 
 \[ \Xi_2(t, \Xi_1(1,q)) = \Xi_1(1, q) \exp( it ( \alpha(y; B(q))-\alpha(y; \tilde r) )),   \] 
 which again fixes $Q$. As a consequence 
 \[  \Xi_2(1, \Xi_1(1,q)) - B(q)\in L^2  .\]
 We define the final deformation by 
 \[ \Xi_3(t, \Xi_2(1,\Xi_1(1,q))) = (1-t) \Xi_2(1, \Xi_1(1,q)) + t B(q). \]

\end{proof} 
 
 \subsection{Proof of Theorem \ref{thm:E1}} 
 
 Theorem \ref{thm:E1} studies the asymptotic phase shift $ \Theta$ and the modified momentum $H_1$. 
 Before we start the construction rigorously, we recall that if $ q \in X^0 $ never vanishes with $\d_x q\in L^1$, we define 
\[ \Theta(q) = -\Im \int_{\R}   \frac{ \d_x q}{  q} dx \] and if $q \in X^s$ with $s >\frac12 $ never vanishes, then we define
\[ H_1(q) =-\Im  \int_{\R} (|q|^2-1)\frac{\d_x q}q dx. \]
The issue is to find a continuous resp. smooth extension across $q$ with zeros, on the covering space of $X^0$.

We first regularise $q\mapsto \tilde q$ as follows, such that $\tilde q$ does not vanish. 
\begin{lem} \label{lem:tildeq} 
Given $ q_0 \in X^s$, $s\geq 0$, there exist   $ \delta\in (0,1) $ depending only on  $E^0(q_0)$ and a continuous map  
\[  B^{X^0}_\delta(q_0) \ni q \to \tilde q\in X^{s+2} \] 
so that  $|\tilde q| =1 $, 
\begin{equation}\label{q,tildeq}
\Vert q - \tilde q \Vert_{H^s} \le c (E^0(q_0))E^s(q), 
\end{equation}
and, if $\d_x q \in L^1$ then  $ q - \tilde q \in L^1$ and $ \d_x \tilde q \in L^1$.
\end{lem}  
For $ q \in 1+ \mathcal{S}$ we may choose $ \tilde q=1$ if we give up the estimate \eqref{q,tildeq}. 
We postpone the proof of Lemma \ref{lem:tildeq} and complete the proof of Theorem \ref{thm:E1}. 

 \begin{proof}  [Proof of Theorem \ref{thm:E1}]
We define for $q\in X^0$ with $\d_x q\in L^1$,
\[ \Theta = -\Im \int_{\R} \partial_x \log \tilde q\, \dx. \] 
Since $ \Theta$ is the winding number of $ x \to \tilde q(x)$ if $\d_xq \in L^1$, it is obvious that different choices lead to the same $ \Theta$, up to  $2\pi\Z$.

For $q\in X^s$,  $ s \ge \frac12 $ and $ \tilde q $ from Lemma \ref{lem:tildeq}, we define    $H_1$ as in  \eqref{eq:H1}:
\[ H_1(q) = \Im \int_{\R}  ( q- \tilde q ) \partial_x \bar q - \overline{\big( q- \tilde q \big)} \partial_x \tilde q dx.\]  
By Lemma \ref{lem:tildeq}, with the continuity of the maps 
\[ B^{X^0}_\delta(q_0)\cap X^s\ni q \to   q-\tilde q \in H^{1/2} \] 
and 
\[ B^{X^0}_\delta(q_0)\cap X^s\ni q \to  \d_x \tilde q  \in H^{-1/2},\]
the map $  q \to H_1(q)$ is continuous on $B_\delta^{X^0}(q_0) 
\cap X^s$.
Hence we obtain a unique continuous map for $ s\ge \frac12 $ 
\[ H_1(q): X^s \to  \R / (2\pi \Z). \] 
Conservation of $H_1(q)$ by the Gross-Pitaevskii flow has been proven at the end of Subsection \ref{subs:Results}.
\end{proof}

\begin{proof}[Proof of Lemma \ref{lem:tildeq}] 
The construction of the map $ q \to \tilde q$
is elementary but delicate.  It suffices to construct $ \tilde q $ with $ |\tilde q| \ge \frac14$  since then 
$ \frac{\tilde q}{|\tilde q|} $   has the same properties.

Let $\delta>0$ and $q_0 \in X^s$, $s\geq 0$.
Let  
 $r =  h \ast  q \in X^{s+2}$, $h=\frac{\tau}{2}e^{-\tau|x|}$  (as in Lemma \ref{lem:regular}) for some $\tau\geq 2  $ large  enough, so that $q$ in the unit ball $B_1^{X^0}(q_0)$ centered at $q_0$ are uniformly close to  its regularisation $r$:
 \begin{equation}\label{q,r,ball}
 \Vert q - r \Vert_{L^2} < \delta,\quad \forall  q \in B_1^{X^0}(q_0).\end{equation}  
 This is the only relevance of $\tau$ (which depends on $\delta$ and $E^0(q_0)$).
 We are going to construct $q\mapsto \tilde q$ for a single $q\in X^s$ first, and then to extend this construction $q\mapsto \tilde q$ continuously to the small ball $B^{X^0}_{\delta}(q_0)$ for some  $\delta<1$ small enough.
 \smallbreak
 
 \noindent\textbf{Step 1. Construction of $q\mapsto \tilde q$.}
 Let $q\in X^s$, $s\geq 0$, and $r=h\ast q\in X^{s+2}$, $h=\frac{\tau}{2}e^{-\tau|x|}$ for some $\tau\geq 2$.
 The set $\{ x: |r(x)| <\frac12\} $ is an at most countable union of open intervals.  In the complement we would like to  define $\tilde q=r$.
If $I$ is one of the intervals and $|r|\geq \frac14 $ on $I$, we also want to set $ \tilde q = r$ on $I$. 
There are at most \textit{finitely many} other intervals
\footnote{Indeed, the energy on the interval $J_k=(a_k, b_k)$   reads roughly as
$\int_{J_k} (|r|^2-1)^2+|r'|^2 \gtrsim |J_k|+\frac{1}{|J_k|}$, since $||r|-1||_{J_k}>\frac12$, and the derivative of $|r|$ on some subinterval $\tilde J_k=(a_k, c_k)\subset (a_k, b_k)$ with $|r(c_k)|=\frac14$ is comparable with $\frac{1}{|\tilde J_k|}$ and hence $\int_{\tilde J_k} (|r|')^2\gtrsim \frac{1}{|\tilde J_k|}\geq\frac{1}{|J_k|}$.
 Finite energy assumption implies the finiteness of the number of such intervals $J_k$, and the upper and lower bounds for its length $|J_k|$ in terms of the energy of $r$.}. The actual construction is more involved.
  
 There is a finite number $K$ (possibly $0$) bounded by a constant depending only on the $E^0(q)$ of disjoint intervals
 \begin{equation}\label{Ik}\begin{split}
 &   I_k=(a_k, b_k), \quad 0 < k\leq K<\infty,
 \\
 &\hbox{with  }|r(x)| \ge \frac12  \quad \text{ if }     \frac{b_k-a_k}6 \le  \big| x-\frac{a_k+b_k}2\big|  \le \frac{b_k-a_k}2  
 \\ 
& \hbox{ and }|r(x)|< \frac14\hbox{ for at least one  }x \in I_k,
\end{split}\end{equation} 
with a length $|I_k|$ bounded from below and above by a fixed constant depending only on $E^0(q)$, so that 
$|r(x)| \ge \frac12 $ if $ x\not\in \cup_{k=1}^K I_k$.
Notice that that   $|r|\geq\frac12$ also in the outer thirds of the interval $I_k$: $[a_k, a_k+\frac{b_k-a_k}{3}]\cup [b_k-\frac{b_k-a_k}{3}, b_k]$, while $|r|<\frac14$ on some subinterval in the inner third of the interval $I_k$: $[a_k+\frac{b_k-a_k}{3}, b_k-\frac{b_k-a_k}{3}]$.

We define $\tilde q = r$ on the complement of these intervals.  It remains to define $\tilde q$ on the intervals $I_k=(a_k,b_k)$. To do that we need some preparations. We fix a smooth monotone function $\phi$ identically $0 $ on $(-\infty, 1/3]$ and identically $1$ for $x \ge 2/3 $ and a smooth nonnegative function $ \eta$ supported in $(0,1)$, identically $1$ on $[1/3,2/3]$. We define 
\[ \phi_k(x) = \phi\big( \frac{x-a_k}{b_k-a_k} \big), \quad \eta_k(x) = \eta( \frac{x-a_k}{b_k-a_k} ). \] 
We write $r$ in polar coordinates in the outer thirds of the interval $I_k$, 
\[ r(x) = \rho(x)  e^{i \theta(x) } \quad \text{ on }  \left\{ \frac{b_k-a_k}6\le  \Big|x-\frac{b_k+a_k}2 \Big|\le \frac{b_k-a_k}2   \right\}  \]  
with $0 \le  \theta(b_k) -  \theta(a_k) < 2\pi $. 
{}
We  define  on $I_k$ 
\begin{equation} \label{eq:interpo} \begin{split}  \hspace{1cm} & \hspace{-1cm}   \tilde q(x) = \Big( (1-\eta_k(x)) \rho(x) + \eta_k(x) \big( \rho(a_k) + \phi_k(x)(\rho(b_k)-\rho(a_k)) \big)             \Big)  \\ &\times  \exp\Big( i \big(  (1-\eta_k(x)) \theta(x) + \eta_k(x) \big( \theta(a_k) + \phi_k(x) (\theta(b_k)-\theta(a_k) ) \big) \big)\Big). \end{split}  
\end{equation} 
Then $\tilde q\in X^{s+2}$ satisfies $|\tilde q|\geq\frac14$,  and
\begin{equation}\label{eq:polar}  
\|q-\tilde q\|_{H^s}\leq \|q-r\|_{H^s}+\|r-\tilde q\|_{H^s} \le c (E^0(q),\tau) E^s(q).    \end{equation}  

\smallbreak 

\noindent\textbf{Step 2. Construction of $q\mapsto \tilde q$ on $B^{X^0}_\delta(q_0)$.}
Now we want   to fix a continuous   choice of  $\tilde q$  \textit{locally} in a neighborhood of $q_0$.
 We do the previous construction   for $q_0$ and fix the chosen intervals $I_k=(a_k, b_k)$. Let $q \in B^{X^0}_\delta(q_0)$ with $\delta<1$ sufficiently small (to be determined later),  that is,
\[ d^0(q,q_0) = \left( \int_{\R}  \inf_{|\lambda|=1}  \Vert \sech(x-y) ( q(x)- \lambda q_0(x)) \Vert_{L^2_x(\R)}^2 dy \right)^{\frac12} < \delta. \]  
 Let $r=h\ast q$, $h=\frac{\tau}{2}e^{-\tau|x|}$  with $\tau$ large enough such that \eqref{q,r,ball} holds.

   In the construction of $\tilde q$ below  we are going to use some facts on the complete metric space $(X^s, d^s)$ in Theorem \ref{thm:Xs}, as well as some regularity results, which can   be found in \cite[ Section 6]{KL}.
Indeed, let $ \lambda(y) $ in the definition of $d^0(q,q_0)$ above be the constant at $y$. Then
  we have the following regularity for the function $\lambda$ (see also \cite[Lemma 6.4]{KL})
\begin{equation} 
\label{eq:lambda} |\lambda(x)-\lambda(y)| \le C \delta  |x-y|,\end{equation}  
for some constant $C$ depending on $E^0(q_0)$.
We write $\lambda$   as $\lambda(y) = \exp(i \alpha(y))$ for some real smooth function
$\alpha$ with derivatives bounded by $C\delta$. 
We fix a point $x_0$ such that  $|r(x)| \ge \frac12 $ on $x \in (x_0-1,x_0+1) $ and add a multiple of $2\pi$ so that \[ |\alpha(x_0)| < \frac{\pi}2. \] 
With this choice the function $ \alpha$ is uniquely determined. 
Then, with a possibly different constant,   we derive from \eqref{q,r,ball} that
\[ \int_{\R} \Vert \sech(x-y) ( r(x)-e^{i\alpha(x)} r_0(x)) \Vert_{L^2_x}^2 dy \le C \delta^2.\] 
We take a small enough $\delta$, such that (recalling by the construction in Step 1, $|r_0|\geq \frac12$ on the outer thirds of the interval $I_k$)
\[  | r(x) | \ge \frac14 \qquad \text{ for } x \in \R \backslash \bigcup_k 
\Big( a_k+\frac13(b_k-a_k), b_k-\frac13 (b_k-a_k)\Big).  \] 
We define 
\[ \tilde q(x) = r(x) \quad \text{ in } \R \backslash \bigcup_k 
( a_k, b_k)  \] 
and write $r$ again in polar coordinates in outer thirds of the intervals $I_k$
\[  r(x) = \rho(x) \exp( i \theta(x) ). \] 
On each side $ \theta$ is uniquely defined up to the addition of a multiple of $2\pi$. 
We choose the multiples of $2\pi$ so that
\[ |\theta(a_k) - \theta_0(a_k)|, |\theta(b_k)-\theta_0(b_k) | < \pi/2 \] 
which we can do by choosing $ \delta $ sufficiently small. 

With this choice we define $\tilde q$ by \eqref{eq:interpo} inside the intervals $(a_k, b_k)$. Then $|\tilde q|\geq\frac14$ and $\tilde q\in X^{s+2}$.
Since the parameter $\tau$ depends only on  $E^0(q_0)$, the estimate \eqref{q,tildeq} follows as in \eqref{eq:polar}.
By the construction, $ \tilde q' \in L^1$ if $ q' \in L^1$.
\end{proof}

\setcounter{equation}{0}
\section{The transmission coefficient} \label{sec:T}
In this section we will introduce the (renormalized) transmission coefficient   associated to the Gross-Pitaevskii equation, give and analyze its asymptotic expansion, and finally prove  Theorem \ref{thm:energies}.

We will first explain the Lax-pair  formulation $(L,P)$ of the Gross-Pitaevskii equation in Subsection \ref{subs:Lax}.
In the classical framework $q\in1+\cS$,
we will introduce the   transmission coefficient $T^{-1}(\lambda)$ associated to the Lax operator  
$$L=\left(\begin{matrix}
i\d_x&-iq \\ i\ov q&-i\d_x
\end{matrix}\right),$$
 where the spectral parameters $(\lambda, z)\in \cR$  will be defined in \eqref{cR}  below.

Recall the regularization $r$ of $q$ in Lemma \ref{lem:regular}, such that 
$$
\||r|^2-1\|_{L^2}, \, \|\d_x r\|_{L^2},\, \|r\|_{L^\infty}
$$
can be bounded in terms of the (rescaled) energy norm $E^0(q)$.
Using this regularization $r$ we will in Subsection \ref{subsec:renormalized} approximately diagonalize the spectral equation of the Lax operator $Lu=\lambda u$ as
$$
\tilde L v=z\,v,
\hbox{ with }\tilde L=\begin{pmatrix}
i\d_x +iq_1 & -iq_2\\ iq_3& -i\d_x+i(q_4-q_1)
\end{pmatrix}
$$
 where the elements 
 $ q_j$, $j=1,2,3,4
 $ are formulated in terms of $q,r$ (and $\lambda$) explicitly, such that
 $$
 \|q_j\|_{L^2},\quad j=1,2,3,4
 $$
 can also be bounded in terms of the (rescaled) energy norm $E^0(q)$.

\smallbreak 
We will then in Subsection \ref{subsec:T2n} solve the renormalized Lax equation
\begin{align*}
   & \tilde{\tilde L}w=\begin{pmatrix}
    0&0\\ 0&2z\end{pmatrix}w,
    \hbox{ with }\tilde{\tilde L}=\begin{pmatrix}
    i\d_x&-i q_2\\ iq_3&-i\d_x+iq_4
    \end{pmatrix}
    \\
&    \hbox{ provided with the  initial condition  }\lim_{x\rightarrow-\infty}w(x)=\begin{pmatrix}
    1\\0
    \end{pmatrix}
\end{align*}
by Picard iteration.
Observe in the classical setting $q\in 1+\cS$ that the limit of the first component of the solution $w^1$ at infinity is related to the transmission coefficient $T^{-1}(\lambda)$ as $
    \lim_{x\rightarrow\infty}w^1(x)=e^{\int_{\R}q_1\dx} T^{-1}(\lambda)$, such that we have the asymptotic expansion of the transmission coefficient (up to the correction in terms of $q_1$) as follows
    $$
    e^{\int_{\R}q_1\dx} T^{-1}(\lambda)=\sum_{n=0}^\infty T_{2n},
    $$
   where $T_{2n}$ are $2n$ dimensional integrals essentially from the Picard iteration for fixed $ \lambda$. They are holomorphic functions in  $ \lambda $ for $ \lambda \in \C\backslash \bigl( (-\infty, -1]\cup [1,\infty)\bigr)$ and they depend on the choice of $r$ which we have fixed. 
   We will also estimate the terms $T_{2n}$, $n\geq 2$   and  their sum
 in Subsection \ref{subsec:T2n}, which  yields an estimate for the following difference for $q\in X^0$:
 \[\left | \ln \Bigl( \sum_{n=0}^\infty  T_{2n} \Bigr) - T_2\right |
 =\left | \ln T^{-1}(\lambda)+\int_{\R}q_1\dx  - T_2\right |. \]

   We will simplify $T_2$ up to tolerable cubic error terms (denoted by $O_\tau (\cE^3)$)  in Subsection \ref{subsec:T2}, which is the main result in this section.
   Correspondingly we will rewrite the corrected term $T_2+\Phi$ in Subsection \ref{subsec:T2,Phi} as
   $$
   T_2+\Phi=A+B+O_\tau (\cE^3),
   $$
   where  the   correction function $\Phi$ is given in \eqref{Phi} (with $\cM, \Theta$ denoting the mass and the asymptotic phase change respectively):
\begin{equation}\label{Phi,q1,Theta}
    \Phi=- \int_{\R}q_1\dx  -\frac{i}{2z}\cM-\frac{i}{2z(\lambda+z)}\Theta.
\end{equation} 
Both the terms $A,B$ are well-defined for $q\in X^0$, and we will precise them. We recall that $ \Theta$ depends on the choice of $ \tilde q$ and can be considered as analytic function in $q$ for $q\in X^{0}$ with $\d_x q\in L^1$ modulo $2\pi \mathbb{Z}$, or as an analytic function on the universal covering space. 
  
  \smallbreak
   
  To conclude, for $q\in X^0$, we  will in Subsection \ref{subsec:conclusion} define our renormalized transmission coefficient 
 \begin{equation} \label{eq:tc} 
 \begin{split} 
 &\Tc^{-1}(\lambda)\,  
  =e^{  \Phi(\lambda)} \Big(\sum_{n=0}^\infty  T_{2n}(\lambda)\Bigr)
  \\
&  \hbox{ that is, }\ln\Tc^{-1}(\lambda)=\bigl( \Phi+T_2\bigr)
  +\Bigl( \ln\bigl(\sum_{n=0}^\infty T_{2n}\bigr)-T_2\Bigr).
 \end{split} 
\end{equation}  
It is indeed   $T_c^{-1}:=   T^{-1}(\lambda) \exp\Big(- \frac{i\mathcal{M}}{2z} - \frac{i\Theta}{2z(\lambda + z ) }    \Big) $ if $q\in 1+\cS$.     
We will then complete the proof of Theorem \ref{thm:energies}.
Recall that we   always take $(\lambda, z)\in \cR$ (see \eqref{cR} below, and in particular $ z  = \sqrt{\lambda^2-1 }$ has positive imaginary part). 
We  restrict ourselves  to  the case $\Im\lambda\geq 0$ and 
$$
\tau:=2\Im z\geq 2
$$
from Subsection \ref{subsec:T2n} to Subsection \ref{subsec:T2,Phi}, to simplify the presentation.
We will consider other cases for $\lambda\in \C\backslash{(-\infty, -1]\cup [1,\infty)}$ and make the conclusions in Subsection \ref{subsec:conclusion}. 

\subsection{The Lax-Pair and the transmission coefficient}\label{subs:Lax}
The Gross-Pitaevskii equation \eqref{GP} is completely integrable   by means of the  inverse scattering method. It can be viewed as the compatibility condition for the following two ODE systems (see  Zakharov-Shabat \cite{ZS73})
\begin{equation}\label{LaxPair}\begin{split}
&  u_x = \left( \begin{matrix}
 -i\lambda &  q \\  \bar q & i \lambda 
 \end{matrix} \right) u ,
 \\
 & u_t=i\left(\begin{matrix}
 -2\lambda^2- (|q|^2-1) & -2i\lambda q+ \d_x q
 \\
 - 2i\lambda \ov q- \d_x \ov q & 2\lambda^2+ (|q|^2-1)
 \end{matrix}\right) u,
 \end{split}\end{equation}
 where  $u=u(t,x):\R\times\R\mapsto \C^2$ is the unknown vector-valued solutions.
Equivalently,  the Gross-Pitaevskii equation can formally be reformulated in the   Lax-pair $(L,P)$ form
$$
L_t=PL-LP,
$$
where $L$ is the self-adjoint Lax operator 
\begin{equation}\label{LaxOp}
L=\left(\begin{matrix}
i\d_x&-iq \\ i\ov q&-i\d_x
\end{matrix}\right),
\end{equation} 
and $P$ is the following skewadjoint differential operator 
\begin{equation*}
P=i\left(\begin{matrix}
2\partial_x^2-(|q|^2-1)&-q\partial_x-\partial_x q
 \\
  \bar q\partial_x+\partial_x\bar q & -2\partial_x^2+(|q|^2-1)
\end{matrix}\right).
\end{equation*} 
In the above, the first system in \eqref{LaxPair} reads as the spectral problem $Lu=\lambda u$  and the righthand side of the second system of \eqref{LaxPair} reads  as $Pu$. 
The operators $L(t)$ and $L(t')$ at different times are related by the unitary family $U(t',t)$ generated by the skewadjoint operator $P$ as 
$$
L(t)=U^\ast(t', t)L(t')U(t',t),
$$
and the spectra of the Lax operator $L$ is formally invariant by time evolution.
This inverse scattering transform relates the evolution of the Gross-Pitaevskii flow to the spectral property of the Lax operator $L$. See  \cite{AKNS, CJ, DZ, DPVV, FT, GZ, ZS73} for the study between the potential $q$ and the spectral information of $L$.  

If $q-1$ is Schwartz function, then by  \cite{DPVV, FT}, the self-adjoint operator $L$ has essential spectrum $\cI=(-\infty, -1]\cup [1,\infty)$ and  finitely many \emph{simple real} eigenvalues $\{\lambda_m\}$  in $(-1,1)$. 
We are going to define its  transmission coefficient 
$$T^{-1}(\lambda)
\hbox{ such that }
|T^{-1}(\lambda)|\big|_{\lambda\in \cI}\geq 1,\quad
T^{-1}(\lambda_m)=0,$$
by solving the spectral problem  of the Lax operator $Lu=\lambda u$, i.e. the ordinary differential equation $  u_x =  \begin{pmatrix}
 -i\lambda &  q \\  \bar q & i \lambda 
 \end{pmatrix} u$ in \eqref{LaxPair}. 
 
  We consider the spectral problem on the Riemann surface 
 $  \{ ( \lambda, z) \,|\, z^2 = \lambda^2-1\}$. 
More precisely, we first notice that if $q=1$ then the matrix $ \begin{pmatrix}
 -i\lambda & 1 \\ 1 & i \lambda 
 \end{pmatrix} $
 has eigenvalues 
 $$
iz\hbox{ and }-iz,\quad \hbox{ with }z=z(\lambda)=\sqrt{\lambda^2-1},
 $$
together with the corresponding eigenvectors
 $$
 \begin{pmatrix}1\\i(\lambda+z)\end{pmatrix}
 \hbox{ and }\begin{pmatrix}1\\i(\lambda-z)\end{pmatrix}.
 $$
In the following of this paper we will always take $(\lambda, z)$  on the upper  sheet of the Riemann surface
\begin{equation}\label{cR}\begin{split}
&\mathcal{R}=\bigl\{(\lambda,z)\,|\, \lambda\in \mathcal{V},
\quad 
z=z(\lambda)=\sqrt{\lambda^2-1}\in \mathcal{U}\bigr \}, 
\\
&\hbox{where }\mathcal{V}:=\C\setminus \cI,\quad \cI=(-\infty, -1]\cup [1,\infty),
\\
&\hbox{and }\mathcal{U}:=\{z\in \C\,| \, \Im z>0\}\hbox{ is the open upper complex plane.}
\end{split}\end{equation}  
The conformal mapping
\begin{equation}\label{zeta}\zeta=\zeta(\lambda)=\lambda+z\end{equation}
maps from $(\lambda,z)\in \cR$ to $\zeta\in\cU$
and has an inverse mapping  $\zeta\mapsto\lambda=\lambda(\zeta)=\frac 12(\zeta+\frac{1}{\zeta})$.

In the  classical framework $q-1\in \cS$, as $q\rightarrow 1$ at infinity, we solve indeed the  boundary value problem of the above ODE with respect to the space variable $x\in \R$ in \eqref{LaxPair}:
\begin{equation}\label{Lax} 
 u_x 
 = \left( \begin{matrix} -i\lambda & q \\ \bar q & i \lambda \end{matrix} \right) u, 
 \quad 
 u
= e^{-iz x }
\left(\begin{matrix}
1 \\ i(\lambda-z)
\end{matrix}\right)
+o(1)e^{(\Im z) x}
\hbox{ as }x\rightarrow -\infty.
 \end{equation} 
We define the transmission coefficient $T^{-1}(\lambda)$ on $\cR$ by the asymptotic behavior of the (Jost) solution $u$ of \eqref{Lax} at infinity
\begin{equation}\label{Asymul}\begin{split} 
&u
= e^{-iz x} T^{-1} (\lambda)
\left(\begin{matrix}
  1
\\ 
i(\lambda-z)
\end{matrix}\right)
+o(1) e^{(\Im z) x}
\hbox{ as }x\rightarrow +\infty.
\end{split}\end{equation}
Then
 $ T^{-1}(\lambda) $
 is a holomorphic function on $\mathcal{R}$,  $\lim_{|\lambda|\rightarrow \infty} T^{-1}(\lambda)=1$ and
\begin{equation}\label{T:symm}
\ov T^{-1}(\lambda)=T^{-1}( \ov \lambda),
\quad \hbox{ for } (\lambda, z),\,( \ov \lambda, -\ov z)\in\mathcal{R}.
\end{equation}
We multiply the above   solution $u$ by $e^{-2iz\lambda t}$ such that it solves also the time evolutionary equation in \eqref{LaxPair} and $\d_tT^{-1}(\lambda)=0$.
Hence the transmission coefficient $T^{-1}(\lambda)$ is conserved by the Gross-Pitaevskii flow.

The logarithm of the transmission coefficient has the asymptotic expansion \eqref{lnT:FT0} if $q\in 1+\cS$.

\subsection{Approximate diagonalization of the  Lax equation: Estimates for \texorpdfstring{$q_j$}{qj}}\label{subsec:renormalized}

Let $(\lambda, z)\in \cR$ (defined in \eqref{cR}), such that 
\begin{equation}\label{tau,2}
\Im\lambda\geq 0,\quad \tau:=2\Im z\geq  2.
\end{equation} 
In this situation,  $\zeta=\lambda+z\in \cU$ satisfies $\Im\zeta=\Im\lambda+\Im z\in [\Im z, 2\Im z)\subset\R^+$.

Let $q\in X^0$, and let $r\in X^2$ be given in Lemma \ref{lem:regular}.
We will diagonalize the   Lax equation \eqref{Lax} for the Jost solution $u$ into the ordinary differential equations \eqref{eq:approxdiag} for $v$ below, where the elements $q_j$, $j=1,2,3,4$ will be estimated in Lemma \ref{lem:qj} afterwards. 

A straightforward calculation shows that  
 \[\begin{split} \hspace{.1cm} & \hspace{-1cm}  \frac1{|r|^2-\zeta^2} 
\Bigl[  \left(\begin{matrix} -i\zeta &  r \\
                          \bar r & i\zeta \end{matrix} \right) 
 \left( \begin{matrix} -i\lambda  & q \\ \bar q & i\lambda \end{matrix} \right) 
 -\left( \begin{matrix} |r|^2-1  & 0 \\ 0 & |r|^2-1 \end{matrix} \right) \Bigr]
\left(\begin{matrix} -i\zeta & r \\ \bar r &i\zeta   \end{matrix} \right) 
  \\ & =
 \left( \begin{matrix} -iz & 0 \\ 0 & iz \end{matrix} \right)
+ \frac1{|r|^2-\zeta^2}  \left(\begin{matrix} -i\zeta &  r \\
                          \bar r & i\zeta \end{matrix} \right) 
 \left( \begin{matrix} 0  & q-r \\ \bar q-\bar r & 0  \end{matrix} \right) 
\left(\begin{matrix} -i\zeta & r \\ \bar r &i\zeta   \end{matrix} \right) 
\\ & = \left( \begin{matrix} -iz & 0 \\ 0 & iz \end{matrix} \right)
+  \frac1{|r|^2-\zeta^2}\left( \begin{matrix} -i \zeta
  \big[\bar r (q-r) + r (\bar q -\bar r) \big]  &r^2(\bar q-\bar r) + \zeta^2 (q-r) \\  \zeta^2(\bar q -\bar r ) + \bar r^2(q-r)  &   i\zeta \big[ r (\bar q-\bar r)+ \bar r (q-r) \big]  \end{matrix} \right).
 \end{split} \] 
If $u$ satisfies the Lax equation (the first equation in \eqref{LaxPair}) then   we take
\begin{equation}\label{u,v}
v= \left( \begin{matrix} -i\zeta & r \\ \bar r & i\zeta \end{matrix} \right)u,
\hbox{ or equivalently,  }u=\frac1{|r|^2-\zeta^2} 
\left( \begin{matrix}-i\zeta & r \\ \bar r & i\zeta \end{matrix} \right)v, 
\end{equation}
such that $v$ solves
\begin{equation} \label{eq:approxdiag}  
\begin{split} 
v_x = & \frac1{|r|^2-\zeta^2}\Big[ 
\left( \begin{matrix}-i\zeta & r \\ \bar r & i\zeta \end{matrix} \right)
 \left( \begin{matrix} -i\lambda & q \\ \bar q  & i\lambda \end{matrix} \right) \left( \begin{matrix} -i\zeta   & r \\ \bar r & i\zeta \end{matrix} \right)     
 +\left( \begin{matrix} 0 & r' \\ \bar r' & 0 \end{matrix} \right) 
\left( \begin{matrix} -i\zeta & r \\ \bar r &i\zeta  \end{matrix} \right) \Big] v 
\\
= &   \left( \begin{matrix} -iz &0 \\ 0 & iz \end{matrix} \right) v
+ 
\left( \begin{matrix} -q_1    
&q_2   
 \\ q_3
  &q_4-q_1  \end{matrix} \right)v,
\end{split}
\end{equation} 
where $q_j$, $j=1,2,3,4$ are given  by
\begin{equation}\label{eq:q1234}\begin{split}
q_1\, & =\frac{ i\zeta \big[ |r|^2-1 + 2\Re (\bar r (q-r)) \big]  -  \bar r r'}{|r|^2-\zeta^2},
\\
q_2\, & =\frac{  r(|r|^2-1) + i \zeta r' + r^2(\bar q-\bar r)  +\zeta^2 (q-r)  }{|r|^2-\zeta^2}
\\ &  = \frac{  r\big[ |r|^2-1+ 2 \Re (\bar r(q-r))\big]  + i \zeta r'} {|r|^2-\zeta^2}- (q-r) ,
\\
q_3\, & =\frac{  \bar r (|r|^2-1) - i \zeta \bar r' + \bar r^2( q- r)+ \zeta^2 (\bar q -\bar r)  }{|r|^2-\zeta^2}
\\ & =\frac{  \bar r \big[ |r|^2-1+ 2 \Re ( \bar r(q-r)) \big]  - i \zeta \bar r'  }{|r|^2-\zeta^2}- ( \overline{q-r}) 
\\
q_4\, & =\frac{2i\zeta \big[ |r|^2-1+ 2 \Re (\bar r (q-r))\big]     + 2i \Im (r \bar r')  }{|r|^2-\zeta^2}.
\end{split}\end{equation}

\begin{lem} \label{lem:qj} The map $(2,\infty) \ni  \tau  \to E^s_\tau(q)$, $s\in [0,1]$ is monotonically decreasing. 

Let $(\lambda, z)\in \cR$  with $\Im\lambda\geq 0$ and $\Im z=\frac{\tau}{2}$, $2\leq  \tau_0 \le  \tau$.
Let  $q\in X^0$ with   $ E^0_\tau(q) \le \tau^{1/2}$.
Let  $r$ and $q_j$'s  be given in Lemma \ref{lem:regular} and  \eqref{eq:q1234} respectively. 
Then  we have the following  estimates for $q_j$, $j=1,2,3,4$:
\begin{equation}\label{eq:qr}\begin{split} 
  \Vert q_j \Vert_{L^p} 
  &\lesssim  \frac1{\tau}  \Vert (|r|^2-1, r') \Vert_{L^p}   +  \Vert q-r \Vert_{L^p},
  \quad\forall p\in [1,\infty],
\end{split} \end{equation} 
and 
\begin{equation}\label{eq:qrp}\begin{split} 
 \Vert q_j \Vert_{L^p}\leq C_p \|(|q|^2-1, q')\|_{W^{-1,p}_\tau} \leq C_p \tau^{-1} E^{1+\frac12-\frac1p}_{\tau_0}  (q),\quad \forall p\in [2,\infty),
\end{split}\end{equation} 
where in particular
\begin{equation}\label{eq:qr2}\begin{split} 
  \Vert q_j \Vert_{L^2}  
&\le c 
E^{0}_\tau(q)\leq cE^0_{\tau_0} (q).   
\end{split}  \end{equation}  
\end{lem} 
\begin{proof} 
Monotonicity of $ \tau \to \Vert f \Vert_{H^{s-1}_\tau} $, $s\in [0,1]$ is obvious after a Fourier transform.

We recall that $ |r| \lesssim \tau $ if $ E^0_\tau(q) \le \tau^{1/2}$ by Lemma \ref{lem:regular} and observe that  $\Im\zeta\in [\frac{\tau}{2}, \tau)$. Thus   
\begin{equation} \label{eq:lower}  ||r|^2- \zeta^2|= | |r|^2 - (\Re \zeta)^2+ (\Im \zeta )^2- 2i \Im \zeta \Re \zeta|   \gtrsim     |\zeta|^2 + |r|^2 . \end{equation}  
The   inequality \eqref{eq:qr} is an immediate consequence of the structure of the $q_j$.


We use   the estimates \eqref{eq:estimate:r} in Lemma \ref{lem:regular} to derive from \eqref{eq:qr} that 
\begin{equation*}\begin{split} 
  \Vert q_j \Vert_{L^p} 
 \le C_p (1+\tau^{-1/2} E^0_\tau(q))    \Vert ( |q|^2-1, q') \Vert_{W^{-1,p}_\tau},
 \quad \forall p\in [2,\infty),
\end{split} \end{equation*} 
and hence \eqref{eq:qr2} holds.
We use further the estimate
\begin{equation} \label{eq:interW} \Vert f \Vert_{W^{s,p}_\tau} \lesssim \tau^{s-\sigma} \Vert f \Vert_{W^{\sigma,p}_\tau} \end{equation}  
whenever $ s,\sigma \in \R$, $ s \le \sigma$ and 
the Sobolev inequality $ \Vert f \Vert_{L^p} \le C_p \Vert f \Vert_{H^{\frac12-\frac1p}}$, $p\in [2,\infty)$,  to derive \eqref{eq:qrp} by the following inequality and the trivial bound $E^{1+\frac12-\frac1p} (q)\leq E^{1+\frac12-\frac1p}_{\tau_0} (q)$, $\tau_0\geq 2$,
\begin{equation*}\begin{split} 
 \|q_j\|_{L^p}\leq C_p  \Vert ( |q|^2-1, q') \Vert_{W^{-1,p}_\tau}  
  \le C_p \tau^{-1}    \Vert ( |q|^2-1, q')\Vert_{L^p}  
  \le C_p \tau^{-1} E^{1+\frac12-\frac1p} (q). 
\end{split} \end{equation*} 
The proof of Lemma \ref{lem:qj} is complete.
\end{proof}

\subsection{The  renormalized Lax equation: Estimates for \texorpdfstring{$T_{2n}$, $ n \ge 2$}{Tn}.}\label{subsec:T2n}
Recall the (approximately) diagonalized  Lax equation \eqref{eq:approxdiag}-\eqref{eq:q1234} for the unknown vector-valued function $v$, which is related to the original Jost solution $u$ by the transformation \eqref{u,v}.
The Jost solution $u$ satisfies  the original Lax equation \eqref{Lax} and has the asymptotic behaviour \eqref{Asymul}.

Let  
\begin{equation}\label{u-w} 
w=-\frac{1}{2iz}e^{izx+ \int\limits_{0}^x q_1 \dm} v
=-\frac{1}{2iz}e^{izx+\int\limits_{0}^x q_1 \dm}
\left( \begin{matrix} -i\zeta & r \\ \bar r & i\zeta \end{matrix} \right)u,
\end{equation}
then it satisfies the renormalized ODE  (of the original ODE \eqref{Lax} for $u$)
\begin{equation} \label{ReLax} 
\begin{split} 
w_x = &     \left( \begin{matrix} 0 &0 \\ 0 & 2iz \end{matrix} \right) w 
+\left( \begin{matrix} 0 
&  {q_2 } 
 \\  {q_3} 
  &  {q_4 }   \end{matrix} \right)w,
  \quad \lim_{x\rightarrow-\infty}w(x) =\begin{pmatrix} 1\\0\end{pmatrix}.
\end{split} 
 \end{equation}    
We call $w$ the renormalized Jost solution. It 
satisfies    the following integral equation  
\begin{equation*}\label{eq:w,integral}
\begin{split} 
 w (x)
&=\lim_{y\rightarrow -\infty} w(y) 
\\
&\quad +
\int_{-\infty}^x  
\left( \begin{array}{cc}
0
&  q_2(x_1)
\\
e^{2iz(x-x_1)+\int^x_{x_1} q_4 \dm} q_3(x_1)
 & 0
\end{array}\right) \,  w(x_1)\,dx_1,
\end{split}
\end{equation*}   
with the following asymptotics as 
$x\rightarrow\pm\infty$ (recalling $u$'s asymptotics \eqref{Asymul}):
\begin{equation}\label{Asymw}\begin{split}
&   w (\lambda,x)= \left(\begin{matrix}
1\\0
\end{matrix}\right)+o(1)
\hbox{ as }x\rightarrow -\infty,
\\
&  w (\lambda,x)= 
\left(\begin{matrix}
e^{   \int_{-\infty}^{\infty} q_1\dm } T^{-1}(\lambda)\\0
\end{matrix}\right)+o(1)
\hbox{ as }x\rightarrow +\infty.
\end{split}\end{equation}  
Hence we use the following Picard type iterative procedure to derive the first component of the vector-valued Jost solution $w$, taking two steps at one time 
\begin{equation}\label{eq:integral} 
\begin{split}
&w^1=\sum_{n=0}^\infty w_{n},
\quad w_{0}=  1 ,
\\
&w_{n}(x)=\int^x_{-\infty} q_2(y_1) \int_{-\infty}^{y_1} e^{2iz(y_1-y_2) + \int_{y_2}^{y_1}  q_4 dm}  q_3(y_2) w_{n-1}(y_2) dy_2 dy_1.  
\end{split} 
\end{equation} 

\begin{lem} \label{lem:iteration} 
Let $w_n$ be defined in \eqref{eq:integral} iteratively.
Suppose that 
\begin{equation}\label{q4}
\Vert q_4 \Vert_{L^2}^2 \le 4C\Im z,
\end{equation}
for some positive constant $C>0$.
Then the following estimate holds 
\[ \Vert w_n \Vert_{C_b} + \Vert \partial_x w_n  \Vert_{L^1}
\le e^C (\Im z)^{-1}   \Vert q_2 \Vert_{L^2} \Vert q_3 \Vert_{L^2} \Vert w_{n-1} \Vert_{L^\infty}. \] 
\end{lem} 

\begin{proof} 
Let $\tau=2\Im z$. We observe for $y_1<y_2$,
\[ \begin{split} -2\Im z (y_1-y_2) + \int_{y_1}^{y_2} \Re q_4 dm \, & \le - \tau (y_1-y_2) +  |y_1-y_2|^{\frac12}\Vert q_4 \Vert_{L^2}\\
& \le - \frac12 \tau (y_1-y_2) +\frac12 \tau ^{-1} \Vert q_4 \Vert_{L^2}^2 \\ 
& \le -\Im z (y_1-y_2) +C. \end{split} \] 
Hence, by Young's inequality for convolutions 
\begin{align*} \Vert w_n \Vert_{L^\infty} 
&\le   \int_{y_2<y_1} e^{C- \Im z (y_1-y_2)}  |q_2(y_1)||q_3(y_2)| dy_2 dy_1 \Vert w_{n-1} \Vert_{L^\infty} \\
&\le e^C(\Im z)^{-1} \Vert q_2 \Vert_{L^2} \Vert q_3 \Vert_{L^2}\Vert w_{n-1} \Vert_{L^\infty}.
\end{align*}
Finally, since
\[ |\partial_x w_n (x)| \le  |q_2(x)|\cdot \int_{-\infty}^x  e^{C-\Im z(x-y)} |q_3(y)|  \dy\cdot \Vert w_{n-1} \Vert_{L^\infty},   \] 
we have
\begin{align*} 
\Vert \partial_x w_n \Vert_{L^1} 
&\le   \int_{y<x} |q_2(x)| e^{C-\Im z(x-y)} |q_3(y)| \dx\dy 
\Vert w_{n-1} \Vert_{L^\infty}
\\
&\le e^C (\Im z)^{-1} \Vert q_2 \Vert_{L^2} \Vert q_3 \Vert_{L^2} \Vert w_{n-1} \Vert_{L^\infty}.
\end{align*} 
\end{proof} 




 If \eqref{q4} holds, then  we define the limit of $w_n$ at infinity as 
 \begin{equation}\label{T2n}\begin{split}
& T_{2n}(\lambda)
  := \lim_{x\to \infty} w_{n}(x)
 \\
& =\int\limits_{x_1<y_1< \dots < x_n < y_n} \!\! \prod_{j=1}^n  e^{\big( 2iz(y_j-x_j) +\! \int_{x_j}^{y_j} q_4(m) dm \big)}  q_3(x_j) q_2(y_j) dx_j dy_j,\quad n\geq 0,
\end{split} \end{equation}
(with $T_0=1$), such that by Lemma \ref{lem:iteration} it satisfies
\begin{equation}\label{eq:T2n}   |T_{2n}(\lambda)| \le  \Big(e^C(\Im z)^{-1}  \Vert q_2 \Vert_{L^2} \Vert q_3 \Vert_{L^2} \Big)^n.    \end{equation}   
If $q\in 1+\cS$, recalling the limit of the first component $w^1$ at infinity in \eqref{Asymw}: $\lim_{x\to \infty} w^1(x) =e^{\int_{\R}q_1\dx}T^{-1}(\lambda)$, it holds
 \begin{equation}\label{w1,T2n}
 \sum_{n=0}^\infty T_{2n} 
 =e^{\int_{\R}q_1\dx}T^{-1}(\lambda).
 \end{equation}

\begin{lem}[Estimates for $T_{2n}$, $n\geq 2$] \label{lem:T2n} 
Let $(\lambda, z)\in \cR$  with $\Im\lambda\geq 0$ and $\Im z=\frac{\tau}{2}$, $\tau\geq 2$, and $q\in X^0$ with   $ E^0_\tau(q) \le \tau^{1/2}$.
Let  $r$, $q_j$'s, $T_{2n}$'s be given in Lemma \ref{lem:regular}, \eqref{eq:q1234} and \eqref{T2n} respectively.

There exist $c, \delta$  (independent of $\lambda, z, q$) such that
\begin{equation}\label{T2n,E0tau} |T_{2n}(\lambda)|
\leq c^n \tau^{-n}(E^0_\tau(q))^{2n},   \quad n\geq 1,\end{equation}
and hence  for $q\in X^s$,  $s\in [0,2]$, and some fixed $\tau_0\geq 2$,
\begin{align*}
    |T_{2n}(\lambda)| 
\lesssim  c^n \tau^{-n-2s}  (E^0_{\tau_0}  (q))^{2n-2}   (E^s_{\tau_0} (q))^2,\quad n\geq 2,
\quad\forall \tau\geq\tau_0,
\end{align*}
which implies for $q\in X^s$, $s\in[0,2]$ with smallness condition  $   E^0_{\tau_0}(q) \le  \delta{\tau_0}^{1/2}$ for some $\tau_0\geq 2$,
\[ \sum_{n=2}^\infty |T_{2n}(\lambda)| \le c \tau^{-2-2s} (E^0_{\tau_0}(q))^{2} (E^s_{\tau_0}(q))^2,\quad\forall \tau\geq\tau_0 . \]
\end{lem} 
\begin{proof} 
We derive \eqref{T2n,E0tau} straightforward from $E^0_\tau(q)\leq \tau^{\frac12}$, \eqref{eq:qr2} and \eqref{eq:T2n}:
\[ \begin{split} 
|T_{2n}(\lambda)  |\, 
&  \le  \Big (\frac{c}{\tau}  \Vert q_2 \Vert_{L^2} \Vert q_3 \Vert_{L^2}\Big)^n 
  \le \Big(\frac{c}{\tau}\Big)^n   (E_\tau^0(q))^{2n},\quad\forall n\geq 1.
\end{split} 
\]

If $s\in [0,1]$,   then   by virtue of the trivial estimates $ E^0_\tau(q) \le \tau^{-s} E^s_\tau(q)$ 
which follows from $ \Vert f \Vert_{H^{-1}_\tau} \le \tau^{-s} \Vert f \Vert_{H^{-1+s}_\tau }$ if $s\geq 0$ and $ E^s_\tau(q) 
\le E^s_{\tau_0}(q) $ if $ 0 \le s \le 1$ and $\tau\geq\tau_0$, we derive
\begin{align*}
    |T_{2n}(\lambda)  |\,  
  \leq c^n \tau^{-n-2s}  (E ^0_\tau(q))^{2n-2} (E^s_\tau(q))^2
  \leq c^n \tau^{-n-2s}  (E ^0_{\tau_0}(q))^{2n-2} (E^s_{\tau_0}(q))^2.
\end{align*}
If $s\in (1,2]$ and $n\geq 2$, then we take   $\sigma=s/2\in (0,1]$ such that
\[ \begin{split} |T_{2n}(\lambda)  |
\, &  \le  c^n \tau^{-n-2s}  (E ^0_\tau(q))^{2n-4} (E^{\sigma}_\tau(q))^4 
\\ & \le  c^n \tau^{-n-2s}  (E^0_{\tau_0}(q))^{2n-4} (E^{\sigma}_{\tau_0}(q))^4,
\end{split} 
\] 
which, together with interpolation 
$$
 E^\sigma_{\tau_0}(q)  \le  (E^0_{\tau_0}(q))^{1-\sigma}  (E^1_{\tau_0}(q))^\sigma \leq (E^0_{\tau_0}(q)   E^s_{\tau_0}(q))^{1/2},$$
implies
$$ |T_{2n}(\lambda)  | 
\le  c^n \tau^{-n-2s}  (E ^0_{\tau_0}(q))^{2n-2}   (E^s_{\tau_0}(q))^2,\quad \forall \tau\geq \tau_0.
$$
The bound for 
$$ \sum_{n=2}^\infty |T_{2n}|\leq \sum_{n=2}^\infty c^n (\tau^{-\frac12}E^0_\tau(q))^{2n-4}(\tau^{-\frac12} E^0_\tau(q))^4 $$ follows by a geometric sum and the trivial bound
$$
\tau^{-\frac12}E^0_\tau(q)\leq {\tau_0}^{-\frac12}E^0_{\tau_0}(q),
\quad\forall \tau\geq\tau_0.
$$
The proof is complete.
\end{proof}

\begin{cor} \label{coro:T2n}
Assume the same assumptions as in Lemma \ref{lem:T2n}.
There exists a small  constant $\delta_1$, such that for all   $q\in X^s$, $s\in [0,2]$ with  smallness condition $ E^0_{\tau_0}(q) \le  \delta_1 {\tau_0}^{1/2}$ for some $\tau_0\geq 2$, we have
\begin{equation}\label{T2n,T2}  
\left| \ln \Big(\sum_{n=0}^\infty  T_{2n}(\lambda) \Big) - T_2(\lambda)  \right| \le     c  \tau^{-2 }   (E^0_{\tau}(q))^4\le     c  \tau^{-2- 2s }   (E^0_{\tau_0}(q))^2 (E^s_{\tau_0}(q))^2,\quad \forall \tau\geq\tau_0.  \end{equation}

\end{cor}  
\begin{proof} 
By the Lipschitz continuity of the logarithm on $\{\zeta \in \mathbb{C}: |\zeta-1 | \le \frac12 \} $ and the triangle inequality 
\[ \begin{split}   \left| \ln \Big(\sum_{n=0}^\infty  T_{2n}(\lambda) \Big) - T_2(\lambda)  \right| \, & \lesssim  \Big|  \Big(\sum_{n=0}^\infty  T_{2n}(\lambda) \Big) - (1+ T_2) \Big| 
\\ & \qquad + \Big| \Big(\sum_{n=0}^\infty  T_{2n}(\lambda) \Big) \Big| \Big| \exp( T_2) -(1+T_2) \Big|   
\\ & \lesssim \sum_{n=2}^\infty |T_{2n}(\lambda)| + |T_2(\lambda)|^2. 
\end{split}\] 
Under the smallness condition $ E^0_{\tau_0} (q)\le \delta_1 \tau_0^{1/2}  $ and $s\in [0,2]$, the estimate \eqref{T2n,T2} follows from Lemma \ref{lem:T2n}: 
\[ 
\Big|\sum_{n\geq 2}T_{2n}\Big|\leq c \tau^{-2-2s}(E^0_{\tau_0}(q))^2(E^s_{\tau_0}(q))^2  \] 
and 
\[ 
|T_2(\lambda)|^2\leq c\tau^{-2}(E^0_\tau(q))^4\leq c\tau^{-2-2s}(E^0_{\tau_0}(q))^2(E^s_{\tau_0}(q))^2.
\] 



\end{proof}

\subsection{The term \texorpdfstring{$T_2$}{T2} }
 \label{subsec:T2}
 Let $(\lambda, z)\in \cR$  with $\Im\lambda\geq 0$ and $\Im z=\frac{\tau}{2}$, $\tau\geq 2$.
Let $q\in X^0$, and let $r$ be given in Lemma \ref{lem:regular}.
We recall that in \eqref{T2n} we have defined
\[ T_2 = \int_{x<y} \exp\Big( -2iz(x-y) + \int_x^y q_4(m) dm \Big) q_3(x) q_2(y) dx dy, \] 
where 
\begin{align*}
    q_2\, &   = \frac{  r\big[ |r|^2-1+ 2 \Re (\bar r(q-r))\big]  + i \zeta r'} {|r|^2-\zeta^2}- (q-r) ,
\\
q_3\, &   =\frac{  \bar r \big[ |r|^2-1+ 2 \Re ( \bar r(q-r)) \big]  - i \zeta \bar r'  }{|r|^2-\zeta^2}- ( \overline{q-r}) ,
\\
q_4\, & =\frac{2i\zeta \big[ |r|^2-1+ 2 \Re (\bar r (q-r))\big]     + 2i \Im (r \bar r')  }{|r|^2-\zeta^2}.
\end{align*}
Our task in this subsection is to extract the leading term in $T_2$ with respect to large $\tau$. 
We begin with bounds for  multi-linear terms. 
\begin{lem}\label{lem:multilinear} 
Let $a>0$.
Suppose that for $y>x$
\[ \Re \phi(x,y) \ge a (y-x)  \] 
then for $1\le p_1, p_2 \le \infty$ with $ \frac1{p_1}+\frac1{p_2}\ge 1$, 
\begin{equation}\label{eq:bilinear}  
\left| \int_{x<y} e^{-\phi(x,y)} f(x) g(y) dx dy\right| \le  a^{-2+\frac1{p_1}+\frac1{p_1}}  \Vert f \Vert_{L^{p_1}}  \Vert g \Vert_{L^{p_2}}.  \end{equation}  
Similarly, if for $ x_1 \le x_2 \le x_3 $ 
\[ \Re \phi(x_1,x_2,x_3) \ge a (x_3-x_1), \]  
then,  for $ 1 \le p_1,p_2,p_3 \le \infty $ with $ \frac1{p_1}+\frac1{p_2}+\frac1{p_3} \ge 1 $
\begin{equation}\label{eq:lecubic}  \begin{split}
&\Big| \int_{x_1<x_2< x_3 }  e^{-\phi(x_1,x_2,x_3)} f(x_1) g(x_2) h(x_3) dx_1 dx_2 dx_3 \Big|
\\
&\le  a^{-3+\frac1{p_1} +\frac1{p_2} +\frac1{p_3}} \Vert f \Vert_{L^{p_1}}  \Vert g \Vert_{L^{p_2}} \Vert h \Vert_{L^{p_3}}.
\end{split}\end{equation} 
\end{lem} 
\begin{proof} 
We estimate by Young's inequality for convolutions, with $ \frac1r+\frac1{p_1}+\frac1{p_2}=2$ 
\[ \begin{split}  \left| \int_{x<y} e^{-\phi(x,y)} f(x) g(y) dx dy\right| \, & \le 
 \int_{\R} \big( ( e^{a x}\chi_{x<0} )\ast |f|\big)(y)  |g(y)| \dy 
\\ & \le  \Vert f \Vert_{L^{p_1}} \Vert  g \Vert_{L^{p_2}} \Vert e^{a x}\chi_{x<0} \Vert_{L^{r}}, 
\end{split} 
\] 
which implies \eqref{eq:bilinear}. We apply a dual of this estimate  with respect to $x_3$
\[ 
\begin{split} \hspace{1cm} & \hspace{-1cm} 
\Big| \int_{x_1<x_2< x_3 }  e^{-\phi(x_1,x_2,x_3)} f(x_1) g(x_2) h(x_3) dx_1 dx_2 dx_3 \Big|
\\ & \le \int_{x_1<x_2<x_3} e^{-a(x_3-x_1)}|f(x_1)||g(x_2)|\,| h(x_3)| dx_1 dx_2 dx_3 
\\ & = \int_{x_1 < x_2} e^{-a(x_3-x_2) }|f(x_1)| |g(x_2)|\, \bigl(e^{ a x }\chi_{x<0} * |h|\bigr)(x_2) dx_1 dx_2  
\\ & \le a^{-2+\frac1{p_1} + \frac1{r} }  \Vert f \Vert_{L^{p_1}} \Vert |g|\, 
( e^{a x}\chi_{x<0} * |h|) \Vert_{L^r} 
\\ & \le a^{-2+\frac1{p_1} +\frac1{p_2}+\frac1{p_3}} \Vert f \Vert_{L^{p_1}} \Vert g \Vert_{L^{p_2} } 
\Vert   e^{a x}\chi_{x<0} * h \Vert_{L^{p_3}} 
\end{split} 
\] 
where we chose $r\ge 1 $ so that $ \frac1r = \frac1{p_3} +\frac1{p_2} $.
Then \eqref{eq:lecubic} follows since 
\[ \Vert   e^{ a x}\chi_{x<0} * h \Vert_{L^{p_3} }  \le    a^{-1}    \Vert h \Vert_{L^{p_3}}. \]

\end{proof}

To simplify the presentation we formalize the notion of tolerable cubic errors.  
\begin{defi} [Tolerable cubic error]\label{def:error} 
Let  $\tau\geq 2$.
We call a term $F$ a tolerable cubic error if there exists a positive constant $c$ such that for all $p\in [2,3]$ 
\[ |F| \le c \tau^{-3+ \frac3{p}  } 
    \Vert( |q|^2-1 , \d_x q)\Vert_{W^{-1,p}_\tau} ^3.\]
We write 
\[ F = O_\tau ( \mathcal{E}^3). \] 
\end{defi} 
In this subsection we simplify $T_2$ up to tolerable cubic errors. Their relevance is described in the following lemma.  
\begin{lem} [Estimates for tolerable cubic errors]\label{lem:error}
Suppose $2\le \tau_0\le \tau $ and $ F = O_\tau( \mathcal{E}^3)$. Then for $0 \le s \le  \frac74 $
\[ |F| \le c \tau^{-\frac32 - 2s}   E^0_{\tau_0} (q) (E^s_{\tau_0}(q))^2.  \] \end{lem} 
\begin{proof} 
The ideas of proof are exactly as in the proof of Lemma \ref{lem:T2n}.
If $ 0\le s \le \frac32$ we take $p=2$ to  bound
\[ |F| \le c \tau^{-\frac32}  (E^0_\tau(q) )^3. \]
If $ s \in [0,1]$, then we are done by virtue of   $ E^0_\tau(q) \le \tau^{-s} E^s_\tau(q)$   and $ E^s_\tau(q) 
\le E^s_{\tau_0}(q) $.
If $ 1<s \le \frac32$ we obtain by interpolation 
\[ 
\begin{split} 
|F|\, &  \le c\tau^{-\frac32-2s}   (E^{\frac23 s}_\tau  (q))^3 
\\ & \le c \tau^{-\frac32-2s} (E^{\frac23 s}_{\tau_0}  (q) )^3 
\\ & \le c\tau^{-\frac32-2s} E^0_{\tau_0} (q) (E^s_{\tau_0} (q))^2. 
\end{split} 
\]

If $ s= \frac74$ we take $p=3$ to bound by Sobolev embedding $H^{\frac16}(\R)\hookrightarrow L^3(\R)$ and interpolation 
\[ |F| \le c \tau^{-5} \Vert( |q|^2-1, q') \Vert_{L^3} ^3 
\le c \tau^{-5} (E^{\frac76}(q))^3
\le c \tau^{-5} (E^{\frac76}_{\tau_0}(q))^3
\leq  c \tau^{-5} E^0_{\tau_0} (q) (E^{\frac74}_{\tau_0}(q))^2. \]    
If $ \frac32 < s < \frac74$ we proceed in the same fashion but with an exponent $p=(\frac32-\frac{2s}{3})^{-1}\in (2,3)$: \begin{align*}
  &  |F|\leq c\tau^{-3+\frac3p}\|(|q|^2-1, q')\|_{W^{-1,p}_\tau}^3\leq c\tau^{-6+\frac3p} (E^{\frac32-\frac1p}(q))^3
  \\
  &\qquad\qquad\leq c\tau^{-6+\frac3p} (E^{\frac32-\frac1p}_{\tau_0}(q))^3\leq
  c\tau^{-\frac32-2s}E^0_{\tau_0}(q)(E^s_{\tau_0}(q))^2.
\end{align*}
\end{proof} 
\begin{rmk}

We observe that $ \frac32+ 2 \frac{7}{4}= 5$, which gives a decay $ \tau^{-5}$ and  corresponds to the Hamiltonian $H_3$.

\end{rmk}
The main result of this section is: 

\begin{prop}[Leading term in $T_2$]\label{prop:quadratic}  

Assume the same assumptions as in Lemma \ref{lem:T2n}. 
Then 
\[ \begin{split} 
T_2(\lambda) =\,&    \int_{x<y} e^{2iz(y-x)} \Big\{   \frac1{4z^2} \big[ \bar q'(x)q'(y)  + (|q(x)|^2-1)(|q(y)|^2-1)\big]  
\\ & +\frac{1}{4z^2 \zeta}  (|q|^2-1)(x)\Im[  r \bar r'](y) 
+\Im[  r\bar r'](x) (|q|^2-1)(y)  
\\& +\frac{i}{2z\zeta  }  \big[ (|q|^2-1)(x) \Im[\bar r(q-r)](y)  +\Im [r(\overline{q-r})](x) (|q|^2-1)(y)\big]\Big\}  dx dy 
\\ &+\frac{1}{4z^2}\int_{\R}   (q\bar q'-r\bar r')    \dx
\\
&
-  \frac{1}{2iz}\int_{\R}|q-r|^2\dx-\frac{i}{4z^2\zeta}\int_{\R} (|q|^2-1)(|r|^2-1) \dx
 + O_\tau( \mathcal{E}^3).
\end{split} 
\] 
\end{prop} 
\begin{proof}

We are going to simplify $T_2$ up to tolerable cubic errors in  four steps.
\begin{paragraph}{\bf Step 1.} We can remove the integral of $q_4$ from the exponential:
\begin{equation} \label{eq:T20} 
T_2 = \int_{x<y}   e^{-2iz(x-y)} q_2(x) q_3(y) \dx \dy+O_\tau (\cE^3). \end{equation}  
Indeed, since by Lemma \ref{lem:qj}
 \[ \left| \int_x^y q_4dm \right| \le |x-y|^{\frac12} \Vert q_4 \Vert_{L^2} \le c \tau^{1/2} |x-y|^{1/2} 
\tau^{-1/2} E^{0}_\tau(q),  \]
we have if $ \tau^{-1/2} E_\tau^0(q) \le 1 $
\[ \left |1 - \exp\Big( \int_x^y q_4 dm \Big)\right| \le    \int_x^y |q_4(m)|dm 
\exp( c\tau^{1/2} |x-y|^{1/2}   ). \] 
Since 
\[ \exp(c\tau^{1/2} |x-y|^{1/2}) \lesssim \exp(  \tau |x-y|/2 ), \] 
we have 
\[ \begin{split} 
\hspace{1cm} & \hspace{-1cm} 
\left| \int_{x<y} e^{- \tau(y-x)}  \Big| 1- \exp\Big( \int_x^y q_4 dm \Big)\Big|  |q_2(x)| |q_3(y)| dx dy \right|
\\ & \lesssim  \int_{x<m<y}  e^{-\tau(y-x)/2}|q_2(x)||q_4(m)||q_3(y)| dx dm dy  .
\end{split} 
\]
The simplication \eqref{eq:T20} follows now from Lemma \ref{lem:qj} and  Lemma \ref{lem:multilinear}. 

 \end{paragraph} 
 
 \begin{paragraph}{\bf Step 2.} 
 We can replace the denominator $ |r|^2-\zeta^2$ by $1-\zeta^2$: 
\begin{equation} \label{eq:denominator}  
\begin{split} 
T_2 = \, & \int_{x<y} e^{2iz(y-x)} \Big[   \frac{ \bar r (|r|^2-1 + 2 \Re (\bar r(q-r)) ) - i \zeta \bar r'}{1-\zeta^2}- (\overline{q-r} )\Big] (x)  
\\ & \qquad \times \Big[ \frac{  r (|r|^2-1 + 2 \Re (\bar r(q-r)) ) + i \zeta r'}{1-\zeta^2} - (q-r)\Big](y)  \dx\dy 
\\ & + O_\tau( \mathcal{E}^3) .
\end{split} 
\end{equation} 
Noticing $\frac{1}{|r|^2-\zeta^2}=\frac{1}{1-\zeta^2}-\frac{|r|^2-1}{(|r|^2-\zeta^2)(1-\zeta^2)}$, and by  \eqref{eq:lower}  $||r|^2-\zeta^2|\gtrsim |\zeta^2|+|r|^2$ and by Lemma \ref{lem:regular}   $\|r\|_{L^\infty}\lesssim\tau$, we apply the bilinear estimate in  Lemma \ref{lem:multilinear} followed by H\"older's inequality with $p_1=p_2=p_3=p\in [2,3]$ 
\[ 
\begin{split} \hspace{1cm} & \hspace{-1cm} 
\Big| \int_{x<y}  e^{2iz(y-x)} q_3(x) \frac{ (|r|^2-1) (r\big(|r|^2-1  + 2\Re[ \bar r(q-r)]\bigr)+ i \zeta r')}{(|r|^2-\zeta^2)(1-\zeta^2)} dx dy\Big|  
\\ & \lesssim  \tau^{-3}  \int_{x<y} e^{-\tau(y-x) }|q_3(x)| \Big[ ||r|^2-1|\cdot  
\bigl|\bigl(|r|^2-1, \tau(q-r), r'\bigr) \bigr|  (y)\Bigr] \dx \dy 
\\ & \le \tau^{-3-2+\frac3{p}  }  
\Vert q_3 \Vert_{L^{p}}   \Vert |r|^2-1 \Vert_{L^{p}}   \Vert\bigl (|r|^2-1, \tau(q-r), r'\bigr) \Vert_{L^{p}}   ,
\end{split} \]  
which is $O_\tau(\mathcal{E}^3)$  by Lemma \ref{lem:regular} and Lemma \ref{lem:qj}.
In the same fashion we replace the   denominator $|r|^2-\zeta^2$  by $1-\zeta^2$ in $q_3$ and  we obtain \eqref{eq:denominator}.  
\end{paragraph} 

\begin{paragraph}{\bf Step 3.} 
We exchange  $ r(x) $ resp. $ \bar r(x) $ and $ r(y) $ resp. $ \bar r(x)$ and replace $|r|^2$ by $1$
with tolerable cubic errors: 
\begin{equation}\label{T2,Step3} \begin{split}  
T_2 
= \, &   \int_{x<y} e^{2iz(y-x)} \Big\{  \frac{\zeta^2}{(1-\zeta^2)^2}  \bar r'(x)  r'(y)  + \frac1{(1-\zeta^2)^2} (|q(x)|^2-1)(|q(y)|^2-1) 
\\ & + \frac{i\zeta} {(1-\zeta^2)^2}\Big[  ( |q|^2-1)(x) \bar r r'(y) 
- r \bar r'(x) (|q|^2-1)(y) \Big] 
\\ & - \frac{1}{1-\zeta^2} \Big[ (\bar r (|q|^2-1)-i\zeta \bar r')(x) (q-r)(y) 
  + (\overline{q-r})(x) (r ( |q|^2-1)+ i \zeta r')(y) \Big]    
   \\ &  + (\overline{q-r})(x) (q-r)(y) \Big\} dx dy
  + O_\tau( \mathcal{E}^3)  .
\end{split} 
\end{equation} 
We first notice that replacing $r(x)$ by $r(y) $ leads to a tolerable cubic error via \eqref{eq:lecubic} 
\[ 
\begin{split}\hspace{1cm} & \hspace{-1cm} 
\left| \int_{x<y}  e^{2iz(y-x) } (r(x)-r(y)) f_1(x) f_2(y) dx dy \right| 
\\ & = \left| \int_{x<m<y} e^{2iz(y-x)} r'(y) f_1(x) f_2(y) dx dm dy \right| 
\\ & \le c\tau^{-3+\frac1{p_1}+\frac1{p_2}+ \frac1{p_3}}    \Vert r' \Vert_{L^{p_1}}       \Vert f_1 \Vert_{L^{p_2}} \Vert f_2 \Vert_{L^{p_3}} 
\end{split} 
\] 
for $ 1\le p_1,p_2,p_3$ with $ \frac1{p_1} +\frac1{p_2} + \frac1{p_3} \ge 1$.
We thus derive from \eqref{eq:denominator} in Step 2 that
\begin{align*}
T_2=\int_{x<y} e^{2iz(y-x)}  \,     &\Big[\frac{1}{(1-\zeta^2)^{2}}  f(x)\big[ |r|^2 f\big](y)+  \frac{\zeta^2}{(1-\zeta^2)^2} \bar r'(x) r'(y)   
\\ &  + \frac{i\zeta}{(1-\zeta^2)^{2}}\Big(  f(x)\big[  \bar r  r'\big](y) 
 - \big[ r\bar  r'\big](x) f(y) \Big) 
\\ & -\frac1{1-\zeta^2}\Big( (\bar rf)(x)  (q-r)(y)  +   ( \overline{q-r})(x)(r f)(y)\Big)  
\\ & +\frac{i\zeta}{1-\zeta^2} \Big( -(\overline{q-r})(x)r'(y) + \bar r'(x) (q-r)(y)  \Big) 
\\ & + (\overline{q-r})(x)(q-r)(y) \Big] dx dy + O_\tau( \mathcal{E}^3) ,
\end{align*}
where we denote
$$
f:=|r|^2-1+ 2\Re[\bar r (q-r)]= |r|^2-1 + r (\overline{q-r})+\bar r (q-r).
$$ 
We can then harmlessly replace $|r|^2$ by $1$ in the first summand, since by \eqref{eq:bilinear} combined with H\"older's inequality 
\[
\begin{split} \hspace{.3cm} & \hspace{-.3cm} 
\tau^{-4} \int_{x<y}e^{-\tau(y-x)} |f(x)| |(|r(y)|^2-1)f(y)| dx dy 
\\ & \le   \tau^{-3+ \frac{1}{p_1} + \frac1{p_2} +\frac{1}{p_3}}   \big( \tau^{-1} \Vert f \Vert_{L^{p_1}}\bigr) \big( \tau^{-1} \Vert f \Vert_{L^{p_2}} \big) \big(\tau^{-1}  \Vert |r|^2-1 \Vert_{L^{p_3}}\big)  .
\end{split} 
\] 
  Noticing that
\begin{align*}
|q|^2-1=|r|^2-1+r(\ov{q-r})+\ov r(q-r)+|q-r|^2 = f  + (q-r)^2,
\end{align*} 
we replace $f $ by $|q|^2-1$ to arrive at  \eqref{T2,Step3}, 
\[ \begin{split} \hspace{1cm} & \hspace{-1cm} \left| \frac{1}{(1-\zeta^2)^2} \int_{x<y} e^{2iz(y-x)} f(x) (q-r)^2(y) dx dy \right| 
\\ & \le \tau ^{-5+\frac1{p_1}+\frac1{p_2} +\frac{1}{p_3}} \big( \tau ^{-1} \Vert f \Vert_{L^{p_1}}\Big)  \Vert q-r \Vert_{L^{p_2}} \Vert q-r \Vert_{L^{p_3}}.
\end{split} 
\]  
We argue similarly with the remaining terms.  
\end{paragraph}

\begin{paragraph}{\bf Step 4.}
We derive the claim of Proposition \ref{prop:quadratic}  by integration by parts.  
We first work with the $\bar r'(x) r'(y)$ term in \eqref{T2,Step3}. 
Recall that $\frac{\zeta^2}{(1-\zeta^2)^2} = (2z)^{-2}$,   and integrate by parts, 
\[
\begin{split} \hspace{1cm} & \hspace{-1cm} 
\frac{1}{4z^2} \int_{x<y}  e^{2iz(y-x)}( \bar q'(x) q'(y)- \bar r'(x) r'(y)) dx dy 
\\  & =\frac{1}{4z^2} \int_{x<y}  e^{2iz(y-x)} \Big[ (\overline{q-r})'(x) r'(y) + \bar r'(x) (q-r)'(y) \\ & \qquad + ( (\overline{q-r})'(x) (q-r)'(y) \Big] dx dy 
\\ & = \int_{x<y} e^{2iz(y-x)} \Big[-\frac1{2iz} \Big(  (\overline{q-r})(x) r'(y)-\bar r'(x) (q-r)(y)  \Big) \\ & \qquad +  (\overline{q-r} )(x)(q-r)(y)\Big]  dx dy 
\\ & \qquad -\frac{1}{4z^2}\int_{\R}   (q\bar q'-r\bar r')  \dx
+ \frac{1}{2iz}\int_{\R} |q-r|^2  \dx.   
\end{split} 
\] 
This gives 
\[ 
\begin{split} 
T_2(z)= \, &  \int_{x<y} e^{2iz(y-x)} \Big\{ \frac1{4z^2} \bar q'(x) q'(y) + \frac{1}{(2z\zeta)^2} (|q(x)|^2-1)(|q(y)|^2-1) 
\\ & \qquad +\frac{i}{4z^2\zeta} \Big[ (|q|^2-1)(x) \bar r r'(y) - r \bar r'(x) (|q|^2-1)(y) \Big] 
\\ & \qquad +\frac{1}{2z\zeta  } \Big[ (|q|^2-1)(x) (\bar r(q-r))(y) + (r(\overline{q-r}))(x) (|q|^2-1)(y)\Bigr]\Big\}  dx dy
\\
&\qquad +\frac{1}{4z^2}\int_{\R}     (q\bar q'-r\bar r')  \dx
- \frac{1}{2iz}\int_{\R} |q-r|^2  dx
  + O_\tau( \mathcal{E}^3)  . 
\end{split} 
\] 
We rewrite the third term as 
\[ 
\begin{split} \hspace{.3cm} & \hspace{-.3cm} 
\frac{i}{8z^2\zeta} \int_{x<y} e^{2iz(y-x)} (|q|^2-1)(x) (\bar r r'-r \bar r')(y) 
+ (\bar r r' -r \bar r')(x) (|q|^2-1)(y) dx dy 
\\ & + \frac{i}{8z^2\zeta} \int_{x<y} e^{2iz(y-x)} (|q|^2-1)(x)  (|r|^2-1)'(y) 
-  (|r|^2-1)'(x) (|q|^2-1)(y) dxdy, 
\end{split} 
\] 
where by integration by parts the second integral reads as 
\[ 
\begin{split}\hspace{2cm} & \hspace{-2cm}  
\frac{1}{4z\zeta} \int_{x<y} e^{2iz(y-x)} \Big[ (|q|^2-1)(x) (|r|^2-1)(y) + (|r|^2-1)(x)(|q|^2-1)(y)\Bigr] \dx\dy
\\ & -\frac{i}{4z^2\zeta}\int_{\R} (|q|^2-1)(|r|^2-1) \dx.
\end{split} 
\]
Hence  using $\frac{1}{4z^2\zeta^2}+2\cdot\frac{1}{4z\zeta}=\frac{1}{4z^2}$,  we derive Proposition \ref{prop:quadratic}.
\end{paragraph}
\end{proof}

\subsection{The term \texorpdfstring{$T_2+\Phi$}{T2+Phi}}
 \label{subsec:T2,Phi}
 Recall the corrected function $\Phi$   given in \eqref{Phi} if $q\in X^0$ with $|q|^2-1, q'\in L^1$:
\begin{align}\label{Phi,q1}
    \Phi&=-\int_{\R} q_1\dx-\frac{i}{2z}\cM-\frac{i}{2z\zeta}\Theta,
\end{align} 
 where $q_1$ is given in \eqref{eq:q1234}:
 \begin{align*}
     q_1\, & =\frac{ i\zeta \big[ |r|^2-1 + 2\Re (\bar r (q-r)) \big]  -  \bar r r'}{|r|^2-\zeta^2},
 \end{align*}$
 \cM=\int_{\R} (|q|^2-1)\dx$ denotes the mass and $\Theta$ denotes the asymptotic change of the phase given in Theorem \ref{thm:E1}.
 
 We correct the term $T_2$ by $\Phi$   as follows.
\begin{prop}[Corrected term $T_2+\Phi$]\label{prop:T2,Phi}
Assume the same assumptions as in Lemma \ref{lem:T2n}. 

The corrected function $\Phi$ has a unique continuous and smooth extension to $X^0$ modulo $\pi i (z\zeta)^{-1}\Z$, and more precisely, with a choice of $\tilde q$ such that $\tilde q\in X^2$, $q-\tilde q\in L^2$ and $|\tilde q|\geq \frac14$, 
  \begin{equation}\label{Phi,tildeq}\begin{split}
     & \Phi=-\frac{i}{2z}\int_{\R}\frac{(|q|^2-1)(|r|^2-1)}{|r|^2-\zeta^2}\dx
      +i\zeta\int_{\R}\frac{|q-r|^2}{|r|^2-\zeta^2}\dx 
      \\
&\quad       -\frac{i}{2z\zeta} \Im \int_{\R} \Bigl( \bar rr'-\frac{\tilde q'}{\tilde q}\Bigr)\dx 
      +\frac{1}{2z\zeta}\int_{\R} \frac{(|r|^2-1)\bar rr'}{|r|^2-\zeta^2}\dx
      \mod \pi i (z\zeta)^{-1}\Z,
  \end{split}\end{equation}
  where the integral $\Im\int_{\R} \Bigl( \bar rr'-\frac{\tilde q'}{\tilde q}\Bigr)\dx $ is understood as 
  $$  \Im\int_{\R}  \Bigl( (\bar r-\bar {\tilde q}) r' -\bar{\tilde q}'(r-\tilde q)+\frac{\tilde q'}{\tilde q}(|\tilde q|^2-1)\Bigr)  \dx.$$
Then
\begin{equation}\label{AB}
     T_2+\Phi=A+B + O_\tau( \mathcal{E}^3), \mod \pi i (z\zeta)^{-1}\Z,
\end{equation}
where  
\begin{equation}\label{A}
\begin{split} 
A(\lambda)=\, &  \frac{i}{4z^2}\int_{\R} \Im   (q\bar q'-r\bar r')  \dx
\\ &  +\frac1{4z^2}  \int_{x<y} e^{2iz(y-x)}    \Big( \bar q'(x)q'(y)  + (|q|^2-1)(x)(|q|^2-1)(y)\Big)  dx dy ,
\end{split} 
  \end{equation}
and  
\begin{equation}\label{B}  \begin{split} \hspace{1cm} & \hspace{-1cm} 
B(\lambda)= -\frac{i}{2z\zeta}  \Im \int_{\R}  ( \bar r r'-\frac{\tilde q'}{\tilde q}) \dx 
-\frac{i}{4z^2\zeta^2}\int_{\R} (|r|^2-1)\Im (\bar rr')\dx
\\ &   +\frac{1}{4z^2\zeta}\int_{x<y} e^{2iz(y-x)} \Big[ (|q|^2-1)(x) \Im [ r\bar r'](y)
\\ & \hspace{3cm} + \Im [r \bar r'](x) (|q|^2-1)(y) \Big] \dx\dy
\\ &  +\frac{i}{2z\zeta  }\int_{x<y} e^{2iz(y-x)}  \Big[ (|q|^2-1)(x) \Im[\bar r(q-r)](y) 
\\ & \hspace{3cm} + \Im[r(\overline{q-r})](x) (|q|^2-1)(y)\Bigr]  \dx \dy.
\end{split} 
\end{equation}
In particular, we have that
\begin{enumerate}
    \item $B(\lambda)+\overline{B(-\bar\lambda)}=0$ and 
\begin{align}\label{A,Re}
    \frac{1}{2}(A(\lambda)+\overline{A(-\bar \lambda)}) 
     =-\frac{i}{2z}\int_{\R} \frac{1}{\xi^2-4z^2} \Bigl(|\widehat{q'}(\xi)|^2
    + |(\widehat{|q|^2-1})(\xi)|^2\Bigr)\dxi.
\end{align}
\item $\frac{1}{2i}\bigl( A(\lambda)-\overline{A(-\bar\lambda)}\bigr)$ reads as
\begin{align}\label{A,Im}
   \frac{1}{4z^2} \int_{\R} \Bigl( -\frac{1}{\xi} \bigl( |\widehat{q'}(\xi)|^2-|\widehat{r'}(\xi)|^2\bigr) +\frac{\xi}{\xi^2-4z^2}  |\widehat{q'}(\xi)|^2 \Bigr) \dx .
\end{align}
In particular with  $r=\tau^2 D_\tau^{-2}q$,   it reads as
\begin{align}\label{A,Im,r}
     \frac{1}{4z^2} \int_{\R}   \frac{\xi(\tau^4+4z^2(2\tau^2+\xi^2))}{(\tau^2+\xi^2)^2(\xi^2-4z^2)}    |\widehat{q'}(\xi)|^2  \dx ,
\end{align}
which can be bounded by $c\tau^{-1}(E^0_\tau(q))^2$.

 \item There exists $\delta_2>0$ and $\tau_0\geq 2$ such that if $q\in X^s$, $s\in (\frac12, \frac32)$ does not have zeros and satisfies the smallness condition
 \begin{equation}\label{KL,small}
 \tau_0^{-1-\varepsilon}E^{\frac12+\varepsilon}_{\tau_0}(q)\leq \delta_2,
 \end{equation}
for some $0<\varepsilon<\min\{s-\frac12,\frac12\}$,
then we can take $r=q=\tilde q$, such that 
 \begin{equation}\label{AB,larges} \begin{split}
     & \Bigl|\frac{1}{2i} \Bigl( (T_2+\Phi)(\lambda)-\overline{(T_2+\Phi)(-\bar\lambda)}\Bigr)-\frac{\lambda}{4z^3}H_1-\frac{\lambda}{4z^3}\int_{x<y } e^{2iz(y-x)}  \Im[ \bar q' (x)   q' (y)] dx dy \Bigr|
     \\
     &\leq c \tau^{-1 -2s} (\tau_0^{-1-\varepsilon}E^{\frac12+\varepsilon}_{\tau_0}(q))(E^s_{\tau_0}(q))^2,\quad \forall\tau\geq \tau_0,
\end{split} \end{equation}
where $H_1$ is defined in Theorem \ref{thm:E1}. 
Correspondingly \eqref{eq:highmomentae} holds.
\end{enumerate} 
\end{prop}
\begin{proof}
We do more detailed calculations for $\Phi, A$ and the case $s>\frac12$ respectively as follows.

\noindent\textbf{Calculation of $\Phi$.}
We first recall the definition of  $q_1$ in \eqref{eq:q1234}
\begin{align*}
    \int_{\R} q_1 dx =\, & \int_{\R} \frac{ i\zeta (|r|^2-1) +i \zeta \big[r (\bar q -\bar r)+\bar r (q-r) \big]  -  \bar r r'}{|r|^2-\zeta^2} \dx\\ 
=&\int_{\R}\Bigl( \frac{i\zeta}{|r|^2-\zeta^2}(|q|^2-1-|q-r|^2)-\frac{\bar rr' }{|r|^2-\zeta^2}\Bigr)\dx.
\end{align*}
Noticing $\frac{1}{|r|^2-\zeta^2}=\frac{1}{1-\zeta^2}-\frac{|r|^2-1}{(1-\zeta^2)(|r|^2-\zeta^2)}
=\frac{1}{1-\zeta^2}-\frac{|r|^2-1}{(1-\zeta^2)^2}+\frac{(|r|^2-1)^2}{(1-\zeta^2)^2(|r|^2-\zeta^2)}$ and $\frac{i\zeta}{1-\zeta^2}=-\frac{i}{2z}$, we derive
\[ \begin{split} \hspace{2cm}&\hspace{-2cm} 
\int_{\R} q_1 dx  
= 
-\frac{i}{2z}\!\int_{\R} (|q|^2-1)-\frac{(|r|^2\!-\!1)}{1-\zeta^2}(|q|^2-1)+\frac{(|r|^2-1)^2}{(1-\zeta^2)(|r|^2-\zeta^2)} (|q|^2-1) \dx 
\\
 & +\frac{i}{2z} \int_{\R}\Bigl[ |q-r|^2-\frac{|r|^2-1}{|r|^2-\zeta^2}|q-r|^2\Bigr]\dx
 \\
 &+\frac{1}{2z\zeta}\int_{\R} \Bigl[ \bar rr'+\frac{|r|^2-1}{2z\zeta}\bar rr'-\frac{(|r|^2-1)^2}{2z\zeta(|r|^2-\zeta^2)}\bar rr'\Bigr]\dx.
\end{split} 
\] 
Thus the correction function $\Phi=-\int_{\R} q_1\dx-\frac{i}{2z}\cM-\frac{i}{2z\zeta}\Theta$ reads as in \eqref{Phi,tildeq}, and  
\begin{align}\label{Phi,E3}
    \Phi& =\frac{i}{4z^2\zeta}\int_{\R} (|r|^2-1)(|q|^2-1)\dx -\frac{i}{2z}\int_{\R}|q-r|^2\dx 
    \\
& 
   \quad 
    -\frac{1}{4z^2\zeta^2}\int_{\R} (|r|^2-1)\bar rr' \dx 
    -\frac{i}{2z\zeta}  \Im \int_{\R}\Bigl( \bar rr'-\frac{\tilde q'}{\tilde q}\Bigr)\dx    +O_\tau(\mathcal{E}^3).\notag
\end{align} 
Hence \eqref{AB}  follows from Proposition \ref{prop:quadratic}.

\smallbreak
\noindent\textbf{Calculation of $A$.}
To calculate more precisely $A$, we  first notice the following fact by use of the definition of Fourier transformations
\[ \begin{split} 
\hspace{1cm} & \hspace{-1cm} 
\int_{x<y}  e^{2iz(y-x)} \bar f(x)   g(y) dx dy = \frac1{2\pi} 
\int_{\R^2} \int_{x<y}  e^{2iz(y-x)- i \xi x + i \eta y}  \overline{\hat f(\xi)} \hat g(\eta)  d\xi d\eta dy dx 
\\ & =\frac1{2\pi} \int_{\R} \frac{1}{-i (\xi+2z) }   \int_{\R^2} e^{-i(\xi-\eta) y }  \overline{\hat f(\xi)} \hat g(\eta) dy  d\eta d\xi  
  =   \int_{\R} \frac{i}{ (\xi+2z)}  \overline{\hat f(\xi)} \hat g(\xi) d\xi .
\end{split} 
\] 
Similarly, we derive 
\[ \begin{split} 
\hspace{1cm} & \hspace{-1cm} 
\int_{x<y}  e^{2iz(y-x)}  f(x)  \bar g(y) dx dy  
  =   \int_{\R} \frac{-i}{ (\xi-2z)} \hat f(\xi) \overline{\hat g(\xi)}  d\xi .
\end{split} 
\] 
Since if $(\lambda, z)\in \cR$, then  $(-\bar\lambda, -\bar z)\in \cR$ with $\Im(-\bar\lambda)=\Im\lambda$ and $\Im(-\bar z)=\Im z$, we have $B(\lambda)+\overline{B(-\bar\lambda)}=0$ and
\begin{align*}
    &\frac12\bigl( A(\lambda)+\overline{A(-\bar \lambda)}\bigr)
    \\
&    =\frac{1}{4z^2}\int_{x<y} e^{2iz(y-x)} \Bigl( \frac12(\bar q'(x)q'(y)+q'(x)\bar q'(y))
    +(|q|^2-1)(x)(|q|^2-1)(y)\Bigr)\dx\dy,
\end{align*}
which reads in terms of their Fourier transforms as
\begin{align*}
    \frac{1}{4z^2}\int_{\R}  \frac12\bigl( \frac{i}{\xi+2z}-\frac{i}{\xi-2z}\bigr) \Bigl( |\widehat{q'}(\xi)|^2+|(\widehat{|q|^2-1})(\xi)|^2\Bigr)\dx\dy.
\end{align*}
This is \eqref{A,Re}.

Similarly  we arrive at \eqref{A,Im} for $\frac{1}{2i}(A(\lambda)-\overline{A(-\bar\lambda)})$  by use of Plancherel's identity,
which is \eqref{A,Im,r} if $r=\tau^2 D_\tau^{-2}q$.
 It can be bounded by $c\tau^{-1}(E^0_\tau(q))^2$ by virtue of $\tau=2\Im z\geq 2$.
 
 \smallbreak 
 
\noindent \textbf{Special case $s\in (\frac12,\frac32)$.}
 If $q\in X^s$, $ s >\frac12$, then we can take $q=r$ (i.e. we do not need the regularisation procedure) in the definitions of $q_j$'s in \eqref{eq:q1234}. 
 This is the setting of \cite{KL}, and we are going to show below similar estimates as in Lemma \ref{lem:T2n} (for terms $T_{2n}$), Lemma \ref{lem:error} (for cubic errors) and \eqref{AB}-\eqref{A,Re}-\eqref{A,Im} (for $A,B$)   in this setting.
 
We proved in \cite[Proposition 5.1]{KL} that  
 \[ \Big|T_{2n}(\lambda)|_{\lambda= i\sqrt{\tau^2/4-1}}\Big| \le  \left( C (1+\tau^{-1}\|q'\|_{l^2_\tau DU^2}) \tau^{-1} \Vert( |q|^2-1, q') \Vert_{l^2_\tau DU^2}  \right)^{2n},    \]
 which can be compared with \eqref{T2n,E0tau} above.
We refer to \cite{KL} for the definition of the space $l^2_\tau DU^2$, and we have   $\frac{1}{\sqrt{\tau}}E^0_\tau(q)\lesssim \frac1\tau \Vert( |q|^2-1, q') \Vert_{l^2_\tau DU^2}  $ and  the crucial property that
 \[ \Vert f \Vert_{l^2_\tau DU^2} \le c_\sigma \tau^{-\frac12-\sigma} \Vert f \Vert_{H^\sigma_\tau},
 \quad\forall \sigma>-\frac12,\]
 that is, 
 $$\tau^{-1} \Vert( |q|^2-1, q') \Vert_{l^2_\tau DU^2} \leq c_\sigma \tau^{-\frac12-\sigma}E^\sigma_\tau(q), \quad \forall \sigma>\frac12.$$
  Then, provided $E^{\frac12+\varepsilon}_{\tau_0}(q) \le \delta_2 {\tau_0}^{1+\varepsilon}$ for some $\varepsilon>0$ such that $\sigma=\frac12+\varepsilon<\min\{s,1\}$ and for some small enough $\delta_2$, and  some $\tau_0\geq 2$,
 \begin{equation}\label{KL,T2n} \Big|  \sum_{n=2}^\infty T_{2n}(\lambda)|_{\lambda= i\sqrt{\tau^2/4-1}} \Big| \le c \tau^{-1-2s}  (\tau_0^{-1-\varepsilon}E^{\frac12+\varepsilon}_{\tau_0}(q))^{2}  ( E^s_{\tau_0}(q))^2,\quad\forall\tau\geq\tau_0,  \end{equation}
 and we continue with a variant of the argument above. 
 
 We call $F$ a cubic error if 
\[  |F | \le c  \tau^{-3}    \Vert( |q|^2-1, q') \Vert_{l^3_\tau DU^2}   ^3.  \] 
Using 
\[ \Vert f \Vert_{l^3_\tau DU^2} \le \Vert f \Vert_{l^2_\tau DU^2} \quad \text{ and } \quad 
\Vert f \Vert_{l^3_\tau DU^2} \le  c \tau^{-\frac23} \Vert f \Vert_{L^3} ,   \] 
we have the following estimate for cubic errors for $s\in (\frac12, \frac32)$:
\begin{equation}\label{KL,cubic} |F| \le c \tau^{-1- 2s} (\tau_0^{-1-\varepsilon}E^{\frac12+\varepsilon}_{\tau_0}(q)) (E^s_{\tau_0}(q))^2. \end{equation} 

If $q=r$, then we have from \eqref{A}-\eqref{B} that
\[
\begin{split} \hspace{1cm} & \hspace{-1cm} 
A+B  = -\frac{i}{2z\zeta}  \Im \int_{\R}  ( \bar q q'-\frac{\tilde q'}{\tilde  q}) \dx  
-\frac{i}{4z^2\zeta^2}\int_{\R} (|q|^2-1)\Im (\bar qq')\dx
\\
&  +\frac1{4z^2}  \int_{x<y} e^{2iz(y-x)}    \Big( \bar q'(x)q'(y)  + (|q|^2-1)(x)(|q|^2-1)(y)\Big)  dx dy 
\\ & + 
\frac{1}{4z^2\zeta} \int_{x<y} e^{2iz(y-x)} \Big[ (|q|^2-1)(x) \Im [ q \bar q'] (y)+ \Im [q \bar q'](x) (|q|^2-1)(y) dx dy.
\end{split} 
\] 
We can integrate by parts to rewrite the last integral above as
\begin{align*}
   & \frac{i}{4z^3\zeta} \int_{\R} (|q|^2-1)\Im[q\bar q']\dx
   \\
&    +\frac{1}{4iz^3\zeta}\int_{x<y }  e^{2iz(y-x)} \Big( \Re[q\bar q'](x) \Im [ q \bar q'] (y)- \Im [q \bar q'](x) \Re[q\bar q'](y) \Bigr) dx dy,
\end{align*}
where the double integral  reads further as
\begin{equation}\label{larges:cubic}\begin{split}
    &-\frac{1}{4iz^3\zeta}\int_{x<y } e^{2iz(y-x)}  \Im[( q\bar q')(x) (\bar q q')(y)] dx dy
    \\
&    = \frac{i}{4z^3\zeta}\int_{x<y } e^{2iz(y-x)}  \Im[ \bar q' (x)   q' (y)] dx dy
\\
&\quad - 
\frac{i}{4z^3\zeta}\int_{x<y } e^{2iz(y-x)}  \Im\Bigl[ \bar q' (x) \int_{x}^y q'(m)\dm  (\bar qq') (y)\Bigr] dx dy
\\
&\quad+\frac{i}{4z^3\zeta}\int_{x<y } e^{2iz(y-x)}  \Im[ \bar q' (x)    ( (|q|^2-1)q') (y)] dx dy.
\end{split}\end{equation} 
We have showed in  \cite[Appendix A]{KL} that \footnote{Instead of $\Phi$ given by \eqref{Phi,q1} if $q\in X^0$ with $|q|^2-1, q'\in L^1$ here, the correction function $\underline{\Phi}$ in \cite{KL} was given by
$ \underline{\Phi}=-\int_{\R} q_1\dx -\frac{i}{2z}\cM-\frac{i}{2z\zeta}\cP, $
if $q\in X^{(\frac12)_+}$ with $|q|^2-1, q'\in L^1$.
Thus with a choice of $\tilde q$,   $T_2+\Phi=T_2+\underline{\Phi}-\frac{i}{2z\zeta}\Im \int_{\R}  ( \bar q q'-\frac{\tilde q'}{\tilde  q}) \dx$. }  
\begin{align*}
  &  (T_2 +\Phi) -( A+B) , 
  \\
  &\hbox{ with }A+B=-\frac{i}{2z\zeta}  \Im \int_{\R}  ( \bar q q'-\frac{\tilde q'}{\tilde  q}) \dx  
-\Bigl(\frac{i}{4z^2\zeta^2}+\frac{i}{4z^3\zeta}\Bigr)\int_{\R} (|q|^2-1)\Im (\bar qq')\dx 
\\
& \qquad  +\frac1{4z^2}  \int_{x<y} e^{2iz(y-x)}    \Big( \bar q'(x)q'(y)  + (|q|^2-1)(x)(|q|^2-1)(y)\Big)  dx dy 
\\
&\qquad +\frac{i}{4z^3\zeta}\int_{x<y } e^{2iz(y-x)}  \Im[ \bar q' (x)   q' (y)] dx dy
\end{align*}  
is a cubic error, which can be estimated as in \eqref{KL,cubic} under the smallness assumption \eqref{KL,small}.  

Finally if  $q\neq 0$ and we take $\tilde q=q$, then by Theorem \ref{thm:E1} we have
$$
H_1=\int_{\R} (|q|^2-1)\Im[\frac{\bar q'}{\bar q}]\dx \mod 2\pi\Z,
$$
and the integral difference
\begin{align*}
    &\int_{\R}(|q|^2-1)\Im[q\bar q']\dx-H_1 
 =\int_{\R}(|q|^2-1)\frac{|q|^2-1}{|q|}\Im[\frac{q\bar q'}{|q|}]\dx,
\end{align*}
can be controlled as in \eqref{KL,cubic}.  
Thus we have \eqref{AB,larges} by view of $\frac{1}{2z\zeta}+\frac{1}{4z^2\zeta^2}=\frac{1}{4z^2}$ and $\frac{1}{4z^2}+\frac{1}{4z^3\zeta}=\frac{\lambda}{4z^3}$.

Correspondingly \eqref{eq:highmomentae} follows by virtue of \eqref{KL,T2n}. 

\end{proof}

\subsection{Conclusions}\label{subsec:conclusion}
Finally we deduce the claims made in Theorem \ref{thm:energies} about the holomorophy of the renormalized transmission coefficient. 
\subsubsection{Case $(\lambda,z)\in \cR$ with $\Im \lambda\geq 0$, $2\Im z=\tau> 2$ and $q\in X^0$ with $E^0_\tau(q)\leq \delta_1\tau^{\frac12}$}\label{subss:main}
This is the case considered from Subsection \ref{subsec:renormalized} to Subsection \ref{subsec:T2,Phi}. 
Let $r$ be given in Lemma \ref{lem:regular}.
Let $\Phi$ be given in Proposition  \ref{prop:T2,Phi}.

By the obtained results in Subsections \ref{subsec:renormalized}-\ref{subsec:T2,Phi}, we can define for $q\in X^0$ under the smallness condition $E^0_\tau(q)\leq \delta_1\tau^{\frac12}$,
$$
\ln\Tc^{-1}(\lambda)=(T_2+\Phi)+\Bigl(\ln\sum_{n=0}^\infty T_{2n}-T_2\Bigr)\in \C/(\pi i (z\zeta)^{-1}\Z),
$$
where $T_2+\Phi$ is characterized in Proposition \ref{prop:T2,Phi}, and the estimates for $\ln\sum_{n=0}^\infty T_{2n}-T_2$ can be found in Corollary \ref{coro:T2n}.
Hence 
\eqref{eq:energies} and \eqref{eq:lowmomentae} in Theorem \ref{thm:energies} follow from  Corollary \ref{coro:T2n},  Lemma \ref{lem:error} and Proposition \ref{prop:T2,Phi}.

Correspondingly the renormalized transmission coefficient $\Tc^{-1}$ reads as 
\begin{equation}\label{Tc}
\Tc^{-1}(\lambda) =e^{  \Phi(\lambda)} \sum_{n=0}^\infty  T_{2n}(\lambda) \in \C /(e^{\pi i(z\zeta)^{-1}\Z}),
\end{equation} 
where $\Phi$ is defined in Proposition  \ref{prop:T2,Phi}, and $T_{2n}$ are defined in \eqref{T2n}.

If $q\in 1+\cS$, then the (original) transmission coefficient $T^{-1}$   reads as $T^{-1}(\lambda)=e^{-\int_{\R}q_1\dx}\sum_{n=0}^\infty T_{2n}$  (see \eqref{w1,T2n}),
and we have the relation 
\begin{equation*}
    \Tc^{-1}(\lambda)= T^{-1}(\lambda)e^{-\frac{i}{2z}\cM-\frac{i}{2z\zeta}\Theta}.
\end{equation*}
We recall that $\Theta$ is uniquely defined modulo $2\pi\Z$.
This implies that the definition of the renormalized transmission coefficient is independent of the choice of $r$.
We hence can typically choose $r=\tau^2 D_\tau^{-2}q$. 

By the proof of Theorem \ref{thm:E1} and more precisely Lemma \ref{lem:tildeq}, we can choose $\tilde q $ continuously in small balls and hence we find a continuous branch of $\Phi(\lambda; q)$ and hence  $T_c^{-1}(\lambda; q)$ on the covering space of $X^0\ni q$. For $ q= 1$ we may choose $\tilde q=1$ so that $T_c^{-1}(\lambda; 1) = 1 $ (we identify $1$ with  $\tilde 1$  in the covering space).

\subsubsection{Case $(\lambda,z)\in \cR$ with $\Im \lambda\geq 0$, $2\Im z=\tau> 2$ and $q\in X^0$}\label{subss:nosmall}
We are now in the position to remove the smallness condition on the energy.

Given any small enough $\delta_0$,   there exists $R$ so that 
\[ \Vert |r|^2-1 \Vert_{L^2(\R \backslash (-R,R))}
+\tau \Vert r-q \Vert_{L^2(\R \backslash (-R,R))}
+ \Vert  r'  \Vert_{L^2(\R \backslash (-R,R))} < \delta_0 \tau .
 \] 
As the first step, following exactly the ideas in Subsection \ref{subsec:renormalized} and Subsection \ref{subsec:T2n}, we can  solve the $w$-ODE \eqref{ReLax}   on $(-\infty, -R]$. 
Then we solve the original spectral ODE \eqref{Lax} for $u$  on the finite interval  $ [-R, R]$ with initial data at $x=-R$ given by $w(-R)$ under the transformation \eqref{u-w} obtained in the first step. 
On $[R,\infty)$ we solve the $w$-ODE  again
by iteration with initial data at $x=R$ given by $u(R)$ and the transformation \eqref{u-w}. 
The constant $\delta_0$ exists such that these arguments work and we can still define the renormalized transmission coefficient in terms of the limit of the first component of the solution $w$ as
$$
\Tc^{-1}(\lambda)=e^\Phi\lim_{x\rightarrow\infty} w^1(x).
$$

To prove holomorphy of $\Tc^{-1}(\lambda; q)$ in $\lambda$, for any $z_0$ with $\Im z_0>1$, we take a small open disk around $z_0$ and fix $\tau=2\Im z_0$ (instead of $\tau=2\Im z$) in this small disk, to define $r=\tau^2D_\tau^{-2}q$.
Holomorphy of $ T_c^{-1}(\lambda)$ is then a consequence of holomorphy of the $q_j$ with respect to $\lambda$ and of  the Picard iteration-mapping \eqref{eq:integral}. 
Similarly the analyticity of $\Tc^{-1}(\lambda; q)$ in $q$ is obvious from the construction. 

\subsubsection{Case $(\lambda,z)\in \cR$ with $\Im \lambda< 0$, $\tau=2\Im z> 2$ and $q\in X^0$}\label{subss:negative}
In the case $(\lambda, z)\in \cR$ with $\Im\lambda<0$ (we recall that always $ \Im z >0$), we have $\zeta^{-1}=(\lambda+z)^{-1})= \lambda-z$ satisfying $-\Im\zeta^{-1}=-\Im\lambda+\Im z\in (\Im z, 2\, \Im z)\in \R^+$.
We define 
$$
v^{-}= \left( \begin{matrix} -i\zeta^{-1} & r \\ \bar r & i\zeta^{-1} \end{matrix} \right)u,
\hbox{ or equivalently,  }u=\frac1{|r|^2-\zeta^{-2}} 
\left( \begin{matrix}-i\zeta^{-1} & r \\ \bar r & i\zeta^{-1} \end{matrix} \right)v^{-}, $$
such that $v$ solves 
\[ 
\begin{split} 
(v^{-})_x  
= &   \left( \begin{matrix} iz &0 \\ 0 & -iz \end{matrix} \right) v^{-}
+ 
\left( \begin{matrix} q_4^{-}-q_1^{-}    
&q_2^{-}   
 \\ q_3^{-}
  &-q_1^{-}  \end{matrix} \right)v^{-},
\end{split}
\] 
where $q_k$, $k=1,2,3,4$ are given  by  
\begin{equation}\label{eq:q1234-}\begin{split}
&q_1^{-}=\frac{- i\zeta^{-1} \bigl[|r|^2-1+2\Re(\bar r (q-r)) \big]  -    r \bar r'}{|r|^2-\zeta^{-2}},
\\
&q_2^{-}=\frac{  r\bigl[ |r|^2-1+2\Re(\bar r (q-r)) \big] + i \zeta^{-1} r'    }{|r|^2-\zeta^{-2}}
-(q-r),
\\
&q_3^{-}=\frac{  \bar r \bigl[ |r|^2-1 +2\Re(\bar r (q-r)) \big] - i \zeta^{-1} \bar r'    }{|r|^2-\zeta^{-2}}-(\overline{q-r}),
\\
&q_4^{-}=\frac{-2i\zeta^{-1}\bigl[ |r|^2-1+2\Re(\bar r (q-r)) \big]    - 2i\Im ( r \bar r' ) }{|r|^2-\zeta^{-2}}.
\end{split}\end{equation}

There are similar estimates for $ q_j^-$ as in Lemma \ref{lem:qj} if $\Im \lambda < 0$, and all the results in Subsection \ref{subsec:renormalized}-Subection \ref{subsec:T2,Phi} hold correspondingly. 
The arguments above in Paragraph \ref{subss:nosmall}   work also correspondingly.
 
\subsubsection{Case $(\lambda,z)\in \cR$ with  $2\Im z=\tau\leq 4$  and $q\in X^0$}\label{subss:smalltau}
We consider now the case when $\tau=2\Im z\in (0,4]$ is bounded.
For $\delta_0>0$ sufficient small (to be determined later), we can take $\tau_0\geq 2$ such that $E_{\tau_0}^0(q)\leq\delta_0$.
We  take $r=\tau_0^2 D_{\tau_0}^{-2}q$, such that by Lemma \ref{lem:regular},
\begin{align*}
    &\|r\|_{L^\infty}\leq c\tau_0,\hbox{ and } 
    \\
    &\|(\frac{1}{\tau_0}r', \frac{1}{\tau_0}(|r|^2-1), q-r )\|_{L^p}\leq C_p\|(|q|^2-1, q')\|_{W^{-1,p}_{\tau_0}},\quad \forall p\in [2,\infty).
\end{align*}

If $(\lambda, z)\in \cR$ with $\Im\lambda\geq 0$, then $\zeta\in \cU$ with $\Im\zeta>0$ and $|\zeta|\geq 1$, and hence $\zeta^2\in \C\backslash {\R^+}$ with $|\zeta^2|\geq 1$.
Thus similarly as in Lemma \ref{lem:qj} we have also
\begin{align*}
\|q_j\|_{L^p}\leq C_p(\zeta) \|(|q|^2-1, q')\|_{W^{-1,p}_{\tau_0}},\quad \forall p\in [2,\infty).
\end{align*}
We then have by the argument in Subsection \ref{subsec:T2n}  that
\begin{align*}
    |T_{2n}(\lambda)|\leq (C_\lambda)^n (E^0_{\tau_0}(q))^{2n}
    \leq (C_\lambda {\delta_0}^2)^n,\quad \forall n\geq 1,
\end{align*}
where $C_\lambda$ is some positive constant depending on $\lambda, z, \zeta$.
We then take $\delta_0>0$ sufficiently small (depending on $\lambda$) such that $\sum_{n=0}^\infty T_{2n}$ converges, and we define the renormalized transmission coefficient $\Tc^{-1}$ as before (see e.g. \eqref{Tc}).

Together with Paragraph \ref{subss:nosmall} and Paragraph \ref{subss:negative}, we have the holomorphy of  the renormalized transmission coefficient   across the real interval $\lambda \in (-1,1)$.

\subsubsection{The renormalized transmission coefficient and the Lax operator}\label{subss:Tc}
We conclude from Paragraphs \ref{subss:main}-\ref{subss:nosmall}-\ref{subss:negative}-\ref{subss:smalltau} above that   the renormalized transmission coefficient $T_c^{-1}(\lambda; q)$ is well-defined for $\lambda\in \C \backslash (( -\infty, -1] \cup [1,\infty) )$ and $q\in X^0$.
It is understood as a continuous map from the universal covering space of $X^0$ to the holomorphic functions on $ \C \backslash (( -\infty, -1] \cup [1,\infty) )$. 
By construction it is independent of time.

We are now in a position to study the spectrum of the Lax operator \eqref{LaxOp}, if $q\in X^0$. By use of Lemma \ref{lem:tildeq}, we can take its regularisation $ \tilde q$ which satisfies $ \tilde q - q \in L^2 $, $ |\tilde q|=1 $
and $ \tilde q' \in H^1 $.   Then the spectrum of $L$ is the same as the spectrum of conjugated  Lax operator:   
\[ \begin{split}  \left( \begin{matrix} \overline{\sqrt{\tilde q}} & 0 \\ 0 & \sqrt{\tilde q} \end{matrix} \right) 
\left(\begin{matrix} i \partial_x & - i q \\ i \bar q & -i \partial_x \end{matrix}  \right)  \left( \begin{matrix}  \sqrt{\tilde q}  & 0 \\ 0 & \overline{\sqrt{\tilde q}} \end{matrix} \right) 
\, & =\left( \begin{matrix} i \partial_x + \frac{i}2 \overline{\tilde q} \tilde q'   & -i \overline{ \tilde q} q   
\\ i \tilde q \bar q & -i \partial_x -\frac{i}2 \tilde q \overline{\tilde q}'   \end{matrix} 
\right) 
\\ & = \left( \begin{matrix} i\partial_x & -i \\  i & -i \partial_x \end{matrix} \right) 
+ \left( \begin{matrix}  \frac{i}2 \overline{\tilde q} \tilde q' & -i (\overline{ \tilde q} q -1) \\  i (\tilde q \bar q-1)  &  -\frac{i}2 \tilde q \overline{\tilde q}'     \end{matrix} \right) 
\end{split} 
\]   
where the entries of the second matric are in $L^2$. We obtain a compact perturbation of the Lax operator as an operator from $H^{1/2} \to H^{-1/2}$. By a Fourier transform the left operator becomes the multiplication operator by 
\[ \left( \begin{matrix} -\xi & -i \\ i & \xi \end{matrix} \right),\quad \xi\in\R, \]
whose eigenvalues are $\pm \sqrt{\xi^2+1} $, and  hence the spectrum is the continuous spectrum, which is $ (-\infty,-1] \cup[1,+\infty)$. When we consider the Lax operator as an operator from $H^{1/2}$ to $H^{-1/2}$ then the multiplication by the second operator (whose entries are in $L^2$) is a compact perturbation. Hence the essential spectrum of the Lax operator is $(-\infty, -1] \cup [1,\infty)$ and the spectrum outside the essential spectrum consists of isolated eigenvalues (in $(-1,1)$ since
the Lax operator is self adjoint). If $\lambda $ is outside the continuous spectrum then the space of solutions in $L^2((-\infty,0] )$ of 
\[ L u - \lambda u = 0 \] 
is spanned by the left Jost function and thus the geometric multiplicity of eigenvalues in $(-1,1)$ is $1$. The algebraic multiplicity equals the geometric multiplicity since the operator is selfadjoint. Any eigenfunction (with eigenvalue outside the essential spectrum) is a multiple of the left and the right Jost function and hence the transmission coefficient $T_c^{-1}$ vanishes at eigenvalues.  The algebraic multiplicity is the order of the zero. Since the algebraic multiplicity is $1$ the zeroes are simple.  
Thus the renormalized transmission coefficient has simple zeros in $(-1,1)$.

\section*{Acknowledgments}
Herbert Koch was partially supported by the Deutsche Forschungsgemeinschaft (DFG, German Research Foundation)  - EXC-2047/1 - 390685813 -  Hausdorff Center for Mathematics and Project ID 211504053 - SFB 1060.
Xian Liao is funded by the Deutsche Forschungsgemeinschaft (DFG, German Research 
Foundation) – Project-ID 258734477 – SFB 1173.

\printbibliography 
\end{document}